\documentclass{article}
\usepackage{amsmath,amssymb,amsthm,mathtools}
\usepackage{cite}
\usepackage{tikz,tikz-cd}
\usepackage{enumitem}
\setlist[enumerate]{label={\rm(\alph*)},ref={\rm(\alph*)}}
\usepackage[unicode]{hyperref}
\hypersetup{
 pdftitle={Orientations for non-compact Calabi-Yau 4-folds},
 pdfauthor={Ivan Karpov and Felix Thimm}
}
\theoremstyle{plain}
\newtheorem{theorem}{Theorem}[section]
\newtheorem{lemma}[theorem]{Lemma}
\newtheorem{proposition}[theorem]{Proposition}
\newtheorem{corollary}[theorem]{Corollary}
\theoremstyle{definition}
\newtheorem{definition}[theorem]{Definition}
\newtheorem{remark}[theorem]{Remark}
\numberwithin{equation}{section}
\title{GW/PT-correspondence for CY $4$-folds-I:\\Canonical Orientations}
\author{Ivan Karpov}
\date{}
\begin{document}
\maketitle
\begin{flushright}
\textit{To Misha Finkelberg with admiration}\hspace*{3em}
\begin{abstract}
    We build a cohesive topological theory of orientations on moduli stacks for non-compact complex-analytic Calabi-Yau $4$-folds by extending and generalizing the work of Joyce and Upmeier, and the work of Bojko.

    As a result, we get, in particular, canonical family orientations for all moduli spaces of Pandharipande-Thomas pairs.

    Our proofs use the algebraic topology of mapping spaces representing $KU$- and $KO$- theory, as well as some properties of the Lie groups $E_8$ and, surprisingly, $\operatorname{HSpin}(16)$.

    This is part of a joint project with Felix Thimm dedicated to the proof of the $GW/PT$-correspondence for Calabi-Yau $4$-folds by implementing Pardon's programme in dimension $4$.
\end{abstract}
\end{flushright}
\setcounter{tocdepth}{2}
\tableofcontents

\section{Introduction}
\label{sec:introduction}

\subsection{Orientations in Algebraic Geometry}
\label{intro:orientations}

\subsubsection{}
Many questions in algebraic geometry concern counting geometric objects. For example, one may wish to count
curves in a complex space in a certain homology class. The purpose of
such a count is usually to obtain numbers which remain unchanged when
the space is deformed: for example, various versions of \textit{Donaldson-Thomas invariants}.

In many geometric counting problems these contributions have signs. An
\emph{orientation} is part of the data which determine them: the
simplest example is the choice of a positive direction on a line. 

It is also useful to collect the geometric objects under consideration into a
space whose points represent those objects, called a \emph{moduli space}.
An orientation problem then asks whether the necessary choices of signs
can be made consistently over this whole space, and whether this can be done canonically.

We are interested here in orientations needed to count algebraic objects on
Calabi--Yau fourfolds. These are complex spaces of dimension four with
additional geometric structure. The constructions of Borisov and Joyce
\cite{BorisovJoyce2017}, and of Oh and Thomas \cite{OhThomas2023}, attach
counts to suitable moduli spaces on these fourfolds, once orientations
have been chosen. 

Joyce and Upmeier \cite{JU} developed a general method for studying these
questions using topology. Under a condition on the topology of the
fourfold, they prove that orientations exist for counting problems on \textit{compact} Calabi-Yau $4$-folds, that they are invariant under suitable bordisms, and that in certain cases they are canonical.

Our main subject in the present note is the extension of their theory to the \textit{non-compact} case, and, in particular, providing relevant canonical orientations. Being arguably interesting on its own, this extension is needed even if one's original interest is in compact
fourfolds. 

For example, a holomorphic curve in a Calabi-Yau $4$-fold has small open neighbourhoods which are often quite manageable but usually not
compact. Many arguments in curve counting simplify significantly if one can compare the original problem
with a problem on such a neighbourhood. 

We show in the present work that this comparison can be carried out for the orientation data. This is known as the \textit{excision} principle, and turns out to be a suitable framework for carrying out many applications: for example, to Pandharipande-Thomas theory.

The main motivation for this is the use of this machinery in the upcoming
joint work with Felix Thimm which will prove the analogue of Pardon’s theorem
(see \cite{Pardon2023}) for Calabi-Yau 4-folds.

\subsubsection{}
Arguably, the simplest  Calabi-Yau $4$-folds are those of the form $\operatorname{Tot}_Y(K_Y)$ where $Y$ is a $3$-fold. For these local
fourfolds there is an elementary and explicit construction of an orientation, analogous
to a construction discussed by Shi \cite{Shi}. Below we compare this explicit
orientation with the canonical one obtained by our general method. 
\subsubsection{}
We also prove that the spaces
of stable pairs used in the theory of Pandharipande and Thomas
\cite{PandharipandeThomas2010} have preferred orientations.
(We recall that a stable pair is a way of recording a holomorphic curve together with additional
algebraic data; its precise definition is explained below.)

Our construction
uses the topology of a neighbourhood of the curve and the excision property discussed above.

Notably, the canonical orientations of Joyce-Upmeier are defined \textit{only for Calabi-Yau varieties which satisfy a certain topological condition}. However, our canonical orientations for PT-moduli spaces exist \textit{always}: the reason is the excision property.

We also study what happens when the $4$-fold varies, and provide canonical orientations in families as well.
\subsubsection{}
Finally, we prove that \textit{orientations do not always exist}. We construct an explicit smooth
noncompact Calabi--Yau fourfold and a family of objects on it for which
the compactly supported moduli
space is  not orientable showing, thus, that Joyce-Upmeier's topological
condition is non-empty. As far as we understand, this is the first such example in the literature.

All our results also have a complex-analytic (that is, non-algebraic) versions which we prove as well.

\subsection{Main results}
\label{intro:main-results}

 We start with a brief formal recollection of the orientation data for Calabi-Yau $4$-folds.

 Let $X$ be a smooth complex Calabi--Yau fourfold, with a chosen
trivialisation of its canonical bundle, and let
$\mathcal M=\mathcal M_X^{\mathrm{cs}}$ be the derived moduli stack of
perfect complexes with compact support on $X$. Write $\mathcal E$ for the
universal complex on $\mathcal M\times X$, and $\pi$ for the projection
to $\mathcal M$. The complex
\[
\mathcal A=\mathbf R\pi_*\mathbf R\mathcal Hom(\mathcal E,\mathcal E)
\]
has fibre $\mathbf R\operatorname{Hom}_X(E,E)$ at an object $E$.
The tangent and cotangent complexes of $\mathcal M$ are
\[
\mathbb T_{\mathcal M}\simeq\mathcal A[1],\qquad
\mathbb L_{\mathcal M}\simeq\mathcal A^\vee[-1]
\simeq\mathcal A[3].
\]
Here the last isomorphism follows from Serre duality, which gives
$\mathcal A^\vee\simeq\mathcal A[4]$. At $E$, this duality pairs
$\operatorname{Ext}^i_X(E,E)$ with
$\operatorname{Ext}^{4-i}_X(E,E)$; in particular, it gives a symmetric
pairing on $\operatorname{Ext}^2_X(E,E)$.

Taking determinants, we obtain a line bundle
$K_{\mathcal M}=\det\mathbb L_{\mathcal M}\cong\det\mathcal A$.
Serre duality equips this line bundle with a nondegenerate quadratic form,
written as an isomorphism
\[
Q_{\mathcal M}:K_{\mathcal M}^{\otimes2}
\xrightarrow{\ \cong\ }\mathcal O_{\mathcal M}.
\]
An \emph{orientation} is a trivialisation
$o:K_{\mathcal M}\xrightarrow{\cong}\mathcal O_{\mathcal M}$ satisfying
$o^{\otimes2}=Q_{\mathcal M}$; see
Joyce and Upmeier~\cite[Definition~13.2]{JU}.
There are two such choices locally,
differing by a sign. They form a principal $\mathbb Z_2$ bundle
$O_{\mathcal M}^{\omega}$, called the \emph{orientation bundle}.
An orientation is a global section of this bundle (which may, or may not, exist).

There is an analogous topological construction, using the \textit{categorified index} in
complex $K$-theory to define an orientation bundle over the space
representing $K$-theory with compact support on $X$, see
Subsection~\ref{sec:algebraic-topological-orientations}.

\subsubsection{Non-compact fourfolds}
\label{intro:noncompact}

Very roughly speaking, the method of Joyce and Upmeier \cite{JU} associates orientation data to compact
spin-manifolds equipped with bundles, and compares these data along bordisms. This helps them to consider orientations as functors from \textit{bordism categories} using certain analytic machinery, and subsequently reduce the problem to the topological one.

In our paper, the manifolds can be non-compact, but bundles are trivial outside a compact set. We formalise this using what can be described as \textit{relative} (or \textit{compactly-supported})
bordism categories.

The main change from the compact theory is in the construction of orientation
functors. In order to perform it, we use pushforwards in real and complex $K$-theory in place of
indices of differential operators used by Joyce and Upmeier. The comparison with algebraic
orientations is then proved in \S\ref{sec:comparison-without-quasi-projectivity}
and \S\ref{compactly-supported-comparison}. The argument for complex
analytic spaces is given in Appendix~\ref{sec:appendix-analytic-comparison}.

Once this comparison is established, the compactly-supported analogues of bordism calculations of Joyce
and Upmeier give the following result.

\begin{theorem}\label{intro:main-orientation-theorem}
Let $X$ be a smooth complex algebraic Calabi--Yau fourfold, not necessarily
proper, with a chosen trivialization of $K_X$. Write
$\mathcal M_X^{\rm cs}$ for the derived moduli stack of perfect complexes
with compact support, and $O^\omega_{\mathcal M_X^{\rm cs}}$ for its
orientation $\mathbb Z_2$-bundle. Suppose that
\begin{equation}\label{intro:steenrod-condition}
 \int_X\bar\alpha\smile\operatorname{Sq}_c^2(\bar\alpha)=0
 \quad\text{in }\mathbb Z_2
 \quad\text{for every }\alpha\in H_c^3(X;\mathbb Z),
\end{equation}
where $\bar\alpha$ is reduction modulo $2$ and
$\operatorname{Sq}_c^2$ is the Steenrod square on cohomology with compact
support. Then:

\begin{enumerate}
\item The bundle $O^\omega_{\mathcal M_X^{\rm cs}}$ is trivializable.
A flag structure on $X$ natural at zero, understood with compact supports
when $X$ is noncompact, determines such a trivialization. For the definition of flag structures, see \S\ref{sec:canonical-orientations-and-flag-structures}.
\item For $\xi\in K_c^0(X)$, let $\mathcal M_{X,\xi}^{\rm cs}$ be the
substack on which the compactly supported topological $K$-theory class is $\xi$.
If
\begin{equation}\label{intro:kappa-condition}
 \kappa(\xi):=c_2(\xi)-c_1(\xi)^2=0
 \quad\text{in }H_c^4(X;\mathbb Z),
\end{equation}
then the induced orientation on $\mathcal M_{X,\xi}^{\rm cs}$ is
independent of the flag-structure natural at zero. It is therefore
canonical. These canonical orientations are compatible with excision
under open embeddings preserving the Calabi--Yau structures, whenever
the theorem applies to both source and target.
\end{enumerate}
The same conclusions hold for a complex analytic Calabi--Yau fourfold
and its moduli stack of analytic perfect complexes with compact support.
\end{theorem}

The theorem follows from
Proposition~\ref{relative-flag-existence-and-classification} and
Proposition~\ref{canonical-orientation-at-zero-class}, together with the
comparison results above. The analytic assertion uses
Theorem~\ref{thm:analytic-comparison}. 

\subsubsection{Local fourfolds}
\label{intro:local-fourfolds}

On the total space of the canonical bundle of a threefold, Serre duality gives
an explicit orientation. The construction is analogous to the orientation
data discussed by Shi \cite{Shi} and by Joyce and Upmeier
\cite[Example~4.11]{3-folds}, so we call it Shi's orientation. 

Our comparison
uses the compatibility with bordism to reduce the question to local
curves, and we do the calculation for those curves in \S\ref{S6C2}.

\begin{theorem}\label{intro:local-theorem}
\begin{enumerate}
\item Let $Y$ be a smooth complex algebraic threefold and
$p:X=\operatorname{Tot}_Y(K_Y)\to Y$. There is a canonical orientation
$o_{\rm Shi}$ on $\mathcal M_X^{\rm cs}$. More precisely, for the
universal complex $\mathcal E$, put
$\mathcal G=\mathbf R(\operatorname{id}\times p)_*\mathcal E$ and
$\mathcal V=\mathbf R\pi_*\mathbf R\mathcal Hom(\mathcal G,\mathcal G)$,
where $\pi:\mathcal M_X^{\rm cs}\times Y\to\mathcal M_X^{\rm cs}$.
The orientation is given by the canonical trivialization
\[
 K_{\mathcal M_X^{\rm cs}}
 \cong\det\mathcal V\otimes(\det\mathcal V)^\vee
 \cong\mathcal O_{\mathcal M_X^{\rm cs}}.
\]

The analogous
trivialization defines $o_{\rm Shi}$ on the topological $K$-theory space which we denote by 
$\mathcal C_{X}^{\rm cs}$.

\item Let $N$ be an almost  complex manifold of real dimension six,
let $K_N=\det_{\mathbb C}(T_N)^\vee$, and consider the Thom pair
$(Z_N,\partial Z_N)=(D(K_N),S(K_N))$ of the disc bundle and the corresponding sphere bundle. 

For
$\beta\in K^0(N)$, set
$\xi=\operatorname{Th}^{KU}_{K_N}(\beta)\in K^0(Z_N,\partial Z_N)$.
The topological version of \eqref{intro:steenrod-condition} holds on
$(Z_N,\partial Z_N)$. Moreover, if
$c_2(\xi)-c_1(\xi)^2=0$ in $H^4(Z_N,\partial Z_N;\mathbb Z)$, then
\[
 o_{\rm Shi}|_\xi=o_{\rm JU}|_\xi,
\]
where $o_{\rm JU}$ is the canonical orientation obtained from the
relative bordism construction. 

\item Let $C$ be a compact Riemann surface and $E\to C$ a holomorphic
vector bundle of rank three with an isomorphism $\det E\cong K_C$.
On the compactly-supported $K$-theory space $X=\operatorname{Tot}_C(E)$, both orientations are well-defined everywhere and agree.
\end{enumerate}
\end{theorem}

Part~\textup{(a)} is Proposition~\ref{canonical-Shi-orientation},
part~\textup{(b)} is Theorem~\ref{local-fourfold-orientation-comparison},
and part~\textup{(c)} follows from
Proposition~\ref{local-curve-orientation-comparison} straightforwardly.

\subsubsection{Excision and families of stable pairs}
\label{intro:excision-and-families}

Excision identifies orientation choices before one has chosen an
orientation.

More precisely, for an open embedding $j:U\hookrightarrow X$ preserving
the Calabi--Yau structures, extension by zero gives an isomorphism
\begin{equation}\label{intro:excision-isomorphism}
 \operatorname{Exc}_{U,X}:
 O^\omega_{\mathcal M_U^{\rm cs}}
 \xrightarrow{\ \sim\ }
 j_*^*O^\omega_{\mathcal M_X^{\rm cs}}.
\end{equation}
These isomorphisms respect composition of open embeddings, and are explained in \S\ref{bordcatextr}.

We prove that canonical orientations are compatible via these isomorphisms.

We prove also the following theorem.

\begin{theorem}\label{intro:PT-family-theorem}
(a) Let $\pi:\mathfrak X\to B$ be a smooth, topologically locally trivial
complex analytic family of fourfolds with a relative Calabi--Yau
structure. Let $\mathcal P=\mathcal P(\mathfrak X/B)$ be the relative
stack of stable pairs $s:\mathcal O_{X_b}\to\mathcal F$, where
$\mathcal F$ is pure of dimension one with proper support and
$\operatorname{coker}s$ has support of dimension at most zero. Then
$O^\omega_{\mathcal P}$ has a canonical trivialization compatible with
base change. 

(b) On each fibre admitting a flag structure natural at zero,
this trivialization agrees with the induced Joyce--Upmeier orientation.

\end{theorem}

This is Proposition~\ref{PT-canonical-orientation-in-families}. Its proof is local near the support of a pair.

More precisely, a compact curve has
arbitrarily small neighbourhoods on which the hypotheses needed for a
canonical orientation are automatic, by
Lemma~\ref{PT-curve-neighbourhoods}. We orient the problem on such a
neighbourhood and use excision to obtain an orientation on the
original space. Excision also shows that this choice is independent of
the neighbourhood. 

Theorem~\ref{family-canonical-orientations} makes
the construction compatible in families.

\subsubsection{A moduli space which is not orientable}
\label{intro:nonorientable}

The hypothesis \eqref{intro:steenrod-condition} turns out to have a genuine geometric
content: in \S\ref{counterexample} we construct a smooth noncompact
quasiprojective Calabi--Yau fourfold for which the moduli stack of
perfect complexes with compact support is \textit{not orientable}. We find this quite striking indeed. 

The construction
starts with the elliptic curve $\mathbb C/(\mathbb Z+i\mathbb Z)$ and
its symmetric square which a smooth surface.

A projective bundle over this surface helps to construct a
compact singular divisor, which we place in a smooth Calabi--Yau
fourfold. We then build a family of objects supported on the divisor, so that the orientation bundle is non-trivial along a specific loop in this family. This is proved in Proposition~\ref{nonorientable-example}, and immediately gives non-orientability of the whole stack.

This gives a
negative answer, for noncompact fourfolds, to the corresponding
orientability question discussed in \cite[\S 1]{JU}. 
\subsection{Applications to curve counting}
\label{intro:pardon}

The motivation for this paper is an application to the programme of
Pardon \cite{Pardon2023}. Pardon studies curve counting problems on
complex threefolds, allowing both the curves and the
ambient threefold to vary. His construction crucially depends on the existence of canonical \textit{virtual fundamental classes} on PT-moduli spaces in families of \textit{non-compact complex-analytic varieties}.

In further joint work with Felix Thimm we implement this programme for Calabi--Yau fourfolds. Unfortunately, in that situation the canonical virtual fundamental classes in \textit{complex-analytic} families are not constructed yet: our work with Thimm explicitly assumes their existence. 

One necessary condition for this existence in the world of $4$-folds is, however, the existence of canonical orientations in such families. As we already discussed, this is precisely what the present work supplies.

We hope to give the full construction of canonical virtual classes on PT-moduli spaces in families of (perhaps, non-compact) complex-analytic Calabi-Yau $4$-dimensional varieties elsewhere by using the canonical orientations from the present paper and by mildly modifying the methods of the work \cite{BorisovJoyce2017}.

One may have tried to generalize the more widely used work \cite{OhThomas2023} instead; however, at present, its methods require what is known as \textit{resolution property}, and this is absent for general complex-analytic spaces.

\subsection{Some past work}
We would like to put here our work into some context.

First of all, we would like to mention one more time that all material below was directly inspired and much influenced by the results of \cite{JU}.

Second, some very similar results on orientations for non-compact quasi-projective CY4s were
obtained by Bojko in his extremely interesting work \cite{Bojko2021}, which was part of his thesis.

Bojko proves the existence of orientations on moduli spaces for relevant quasi-projective Calabi-Yau $4$-folds there but provides neither explicit calculations for
local curves, nor canonicity results, nor results about families.

The equivalence between the construction in this part of our theorem and that of Bojko
is an elementary exercise in our excision formalism. 

Though our methods are very different, needless to say, the work \cite{Bojko2021} was an important
starting point (and a source of inspiration) for our investigations, and we thank Arkadij for
writing it.

 \begin{remark} Let us also mention that \cite{Bojko2021} was written before Joyce and Upmeier modified their
erroneous argument in \cite{CaoGrossJoyce2020} to the one in \cite{JU}.
Once this modification was made, Bojko’s original argument for orientability (which was originally
dependent on \cite{CaoGrossJoyce2020}), on the nose, stopped working whenever $H_3^c (X)\neq 0$, that is, under
conditions somewhat weaker than the condition (1.1).

However, after the author of the present work discussed with A. B. some ideas used
below, Bojko was able to introduce minor changes which now make his train of thought work
for all quasi-projective Calabi-Yau $4$-folds satisfying (1.1).
\end{remark}

Finally, let us also mention that, as far as we understand, our work answers some questions raised in \cite{BaeKoolPark}, and that important related results (especially in the context of local fourfolds) were obtained by Cao and his coauthors, see \cite{Cao1} and \cite{Cao2}.

\subsection{Some future work}
It
turns out that, besides the problems discussed in the present paper, its analogue for \textit{3-folds} can be used to

(a) prove a certain modification, sufficient for all enumerative purposes, of the conjecture
about orientations from the work of Maulik and Toda \cite{MaulikToda};

(b) prove that canonical orientations for Calabi-Yau 3-folds defined in \cite{JU} and in \cite{JoyceUpmeierSpin2021}
are equivalent.

This will appear in the forthcoming work \cite{Karpov} of the  author. Also, we hope that
this approach will help to shed some light on the very intriguing question about strong (in
the language of Joyce and Upmeier) orientation data (see \cite{3-folds}).

\subsection{The plan of the paper}
\label{intro:plan}

Section~\ref{sec:k-theory-comparison} begins with pushforwards in real
and complex $K$-theory. We define the topological spaces which describe
the orientation problem and compare their orientation bundles with
the algebraic ones. We first discuss the proper case and then pass to
objects with compact support. The section ends with the interpretation
in \textit{real $K$-theory} used to construct the orientation functors.

In \S\ref{sec:bordism-orientations} we prove invariance under bordism
and define the relative and compactly supported bordism categories.
We construct the orientation functors, describe their factorizations,
and apply the bordism calculations of Joyce and Upmeier. This, in particular, gives the explicit description of the
obstruction to orientability in the non-compact case and the canonical orientations when $c_2-c_1^2=0$ and this obstruction vanishes.

Section~\ref{sec:local-orientations} treats local fourfolds. We construct
the explicit orientations and compare them with the orientations
from bordism. The calculation on $S^6\times\mathbb C$ is the main ingredient in this
comparison for local curves. Notably, this calculation uses both exceptional groups $\operatorname{HSpin}(16)$ and $E_8$.

Section~\ref{sec:orientation-in-families}
then shows how the canonical orientations fit together in topologically
locally trivial families under the stated cohomological hypotheses.

In \S\ref{sec:examples-and-PT} we construct the nonorientable example
and prove orientability for stable pairs. We finish by using excision
to construct their canonical orientations in families without imposing
the general obstruction condition on the fibres.

Appendix~\ref{sec:appendix-analytic-comparison} proves the comparison
needed for complex analytic spaces and families. We use the work of
Levy \cite{Levy1987,Levy2008} to compare analytic and topological
pushforwards, with properness imposed on the support. We also check
the compatibility with Serre duality which identifies the orientation
bundles.

\subsection{Acknowledgements}
\label{intro:acknowledgements}

I would like to thank Dominic Joyce, Davesh Maulik, Miguel Moreira, and Richard Thomas for their helpful input to the
present paper.

I would like to separately express my utmost gratitude
to Dominic Joyce and Markus Upmeier for writing their brilliant
monograph \cite{JU}, and to Davesh Maulik for suggesting to think about orientations and for being constantly attentive to my somewhat eccentric topological interests.

This work also owes its existence to Jeremy Hahn
who provided me with his extremely useful explanations concerning some abstract nonsense, gave the right way to think about the topological Serre duality, and wisely suggested to ``get rid of that
variable $X$''.

All sorts of very kind help from Felix Thimm served as a catalyser for starting writing this down, as well as in many other roles.

Special thanks go to Arkadij Bojko for generously coordinating an almost
simultaneous release of the present work and the update of
\cite{Bojko2021}.

\subsubsection{AI Disclosure}
All theorems and proofs in this paper are entirely human-generated, and I bear the full responsibility for any mistakes. (Moreover, most of this work was completed \textit{before} the AI boom of 2026.)

However, I used ChatGPT 5.6 Pro to find references (most importantly \cite{Levy1987}, \cite{Levy2008}, \cite{RRV}) and asked to it several questions about the contents of these three works when trying to understand them.

Once I finished the draft of this paper, I also used ChatGPT 5.6 Pro to edit it.

\section{\texorpdfstring{$K$}{K}-theory and comparison of orientations}
\label{sec:k-theory-comparison}
\subsection{Pushforwards in \texorpdfstring{$K$}{K}-theory}\label{sec:push}

Let $X$ be a smooth manifold of real dimension $d$, equipped with a stable
complex structure.  We first suppose that $X$ is compact and without boundary,
and write $\pi_X:X\to\operatorname{pt}$ for the projection.  A stable complex
structure induces a canonical $\operatorname{Spin}^c$ structure, and hence an
orientation in complex $K$-theory.

Choose an embedding $e:X\hookrightarrow\mathbb R^N$, and let $\nu_e$ be its
normal bundle.  The stable isomorphism
\[
TX\oplus\nu_e\cong\underline{\mathbb R}^{N}
\]
induces a stable complex structure on $\nu_e$, and therefore the Thom class
\[
u^{KU}_{\nu_e}\in K^{N-d}_{\rm c}(\nu_e).
\]
If $p:\nu_e\to X$ denotes the projection, the corresponding Thom
isomorphism is
\[
\operatorname{Th}^{KU}_{\nu_e}(\alpha)
=p^*\alpha\cdot u^{KU}_{\nu_e}.
\]
After a tubular neighbourhood of $X$ in $\mathbb R^N$ has been chosen, extension by
zero and Bott periodicity define the pushforward
\[
\begin{split}
\pi_{X!}^{KU,\mathbb C}:K^q(X)
&\xrightarrow{\ \operatorname{Th}^{KU}_{\nu_e}\ }
K^{q+N-d}_{\rm c}(\nu_e)
\xrightarrow{\ e_!\ }
K^{q+N-d}_{\rm c}(\mathbb R^N)\\
&\xrightarrow{\ \operatorname{Bott}\ }
K^{q-d}(\operatorname{pt}).
\end{split}
\]
Its homotopy class is independent of the embedding, tubular neighbourhood,
and stabilization.  Moreover, by the Atiyah--Singer index theorem,
$\pi_{X!}^{KU,\mathbb C}$ is the index of the Dirac operator associated with
the induced $\operatorname{Spin}^c$ structure.

Following Joyce--Tanaka--Upmeier
\cite[Definition 2.21]{JoyceTanakaUpmeier2020},  we put
\[
\mathcal C_X=\operatorname{Map}(X,BU\times\mathbb Z).
\]
At the level of representing spaces, the preceding construction gives
\[
\pi_{X!}^{KU,\mathbb C}:\mathcal C_X
\longrightarrow \Omega^\infty\Sigma^{-d}KU.
\]

More generally, let $f:X\to Y$ be a proper smooth map of relative dimension
$d_f=\dim X-\dim Y$, and suppose that its stable relative tangent bundle
\[
T_f=TX-f^*TY
\]
is equipped with a stable complex structure.  Applying the preceding
construction to an embedding of the graph of $f$ into
$Y\times\mathbb R^N$ defines
\[
f_!^{KU,\mathbb C}:K^q(X)\longrightarrow K^{q-d_f}(Y).
\]
These pushforwards are functorial and satisfy the projection formula.

If $T_f$ is equipped with a spin structure, the Atiyah--Bott--Shapiro real
Thom class of the normal bundle of the graph defines the pushforward
\[
f_!^{KO}:KO^q(X)\longrightarrow KO^{q-d_f}(Y).
\]
Complexifying the Atiyah--Bott--Shapiro Thom class gives another pushforward
in complex $K$-theory, which we denote by $f_!^{KU,\mathrm{spin}}$.
Writing $c$ for complexification, one has
$c(f_!^{KO}(y))=f_!^{KU,\mathrm{spin}}(c(y))$.
If the spin structure corresponds to a square root $K_f^{1/2}$ of
$K_f=\det_{\mathbb C}(T_f)^\vee$, then
\begin{equation}\label{complex-spin-pushforward}
f_!^{KU,\mathbb C}(x)
=f_!^{KU,\mathrm{spin}}\bigl(x\otimes K_f^{-1/2}\bigr).
\end{equation}
In particular, a trivialization of $K_f$ determines a spin structure for which
the two pushforwards in complex $K$-theory agree.

If $X$ is spin, write
\[
\mathcal C_X^{\mathbb R}=\operatorname{Map}(X,BO\times\mathbb Z),
\]
We then obtain, in particular,
\[
\pi_{X!}^{KO}:\mathcal C_X^{\mathbb R}
\longrightarrow\Omega^\infty\Sigma^{-d}KO.
\]

We shall also use compactly supported pushforwards.  For an arbitrary smooth
manifold $X$, let $X^+$ be its one-point compactification and put
\[
K^q_{\rm c}(X)=\widetilde K^q(X^+),
\qquad
KO^q_{\rm c}(X)=\widetilde{KO}^q(X^+).
\]
A proper embedding $X\hookrightarrow\mathbb R^N$ and the same Thom
construction give the first map below when $X$ has a stable complex structure,
and the second when $X$ has a spin structure:
\[
\pi_{X!}^{KU,\mathbb C}:K^q_{\rm c}(X)\longrightarrow
K^{q-d}(\operatorname{pt}),
\qquad
\pi_{X!}^{KO}:KO^q_{\rm c}(X)\longrightarrow
KO^{q-d}(\operatorname{pt}).
\]
The spin structure also defines $\pi_{X!}^{KU,\mathrm{spin}}$ as above.  More
generally, a stable complex structure and a spin structure on $T_f$ give,
respectively,
\[
f_!^{KU,\mathbb C}:K^q_{\rm c}(X)\longrightarrow K^{q-d_f}_{\rm c}(Y),
\qquad
f_!^{KO}:KO^q_{\rm c}(X)\longrightarrow KO^{q-d_f}_{\rm c}(Y).
\]
We write
\[
\mathcal C_X^{\rm cs}
=\operatorname{Map}_*(X^+,BU\times\mathbb Z),
\qquad
\mathcal C_X^{\mathbb R,\rm cs}
=\operatorname{Map}_*(X^+,BO\times\mathbb Z)
\]
for the corresponding compactly supported $K$-theory spaces.

If $X$ is compact with boundary and $X^\circ=X\setminus\partial X$, then
\[
K^q_{\rm c}(X^\circ)\cong K^q(X,\partial X),
\qquad
KO^q_{\rm c}(X^\circ)\cong KO^q(X,\partial X).
\]
Equivalently,
\[
\mathcal C_X^{\rm cs}
\simeq\operatorname{Map}_*(X/\partial X,BU\times\mathbb Z),
\qquad
\mathcal C_X^{\mathbb R,\rm cs}
\simeq\operatorname{Map}_*(X/\partial X,BO\times\mathbb Z).
\]
A stable complex structure and a spin structure on $X$ define, respectively,
the pushforwards
\[
\pi_{X!}^{KU,\mathbb C}:\mathcal C_X^{\rm cs}
\longrightarrow\Omega^\infty\Sigma^{-d}KU,
\qquad
\pi_{X!}^{KO}:\mathcal C_X^{\mathbb R,\rm cs}
\longrightarrow\Omega^\infty\Sigma^{-d}KO.
\]
When $d$ is divisible by $8$, Bott periodicity identifies the targets with
$BU\times\mathbb Z$ and $BO\times\mathbb Z$, respectively.

Finally, let $\pi:\mathfrak X\to U$ be a topologically locally trivial smooth
family of manifolds of real dimension $d$. Suppose that the relative tangent
bundle has a stable complex structure. For the pushforwards which depend on a
spin structure, suppose that one is also given.
If the fibers are compact,
let $\mathcal C_{\mathfrak X/U}\to U$ and
$\mathcal C^{\mathbb R}_{\mathfrak X/U}\to U$ denote the relative mapping
spaces whose fibers over $u\in U$ are $\mathcal C_{\mathfrak X_u}$ and
$\mathcal C^{\mathbb R}_{\mathfrak X_u}$; equivalently,
\[
\mathcal C_{\mathfrak X/U}
=\operatorname{Map}_U\bigl(\mathfrak X,U\times(BU\times\mathbb Z)\bigr),
\quad
\mathcal C^{\mathbb R}_{\mathfrak X/U}
=\operatorname{Map}_U\bigl(\mathfrak X,U\times(BO\times\mathbb Z)\bigr).
\]
The fiberwise pushforwards assemble into maps over $U$
\[
\pi_!^{KU,\mathbb C}: \mathcal C_{\mathfrak X/U}
\longrightarrow U\times\Omega^\infty\Sigma^{-d}KU,
\qquad
\pi_!^{KO}: \mathcal C^{\mathbb R}_{\mathfrak X/U}
\longrightarrow U\times\Omega^\infty\Sigma^{-d}KO.
\]
For non-compact fibers, let $\mathfrak X_U^+\to U$ be the fiberwise
one-point compactification, with infinity section
$\infty:U\to\mathfrak X_U^+$.  Then
\[
\mathcal C^{\rm cs}_{\mathfrak X/U}
=\operatorname{Map}_{U,*}
 \bigl(\mathfrak X_U^+,U\times(BU\times\mathbb Z)\bigr),
\quad
\mathcal C^{\mathbb R,\rm cs}_{\mathfrak X/U}
=\operatorname{Map}_{U,*}
 \bigl(\mathfrak X_U^+,U\times(BO\times\mathbb Z)\bigr),
\]
where the subscript $*$ means that the infinity section is mapped to the zero
section.  These are the relative spaces of classes whose supports are proper
over $U$ (equivalently, locally uniformly compact in the fibers), and the
compactly supported family indices give
\[
\pi_!^{KU,\mathbb C}: \mathcal C^{\rm cs}_{\mathfrak X/U}
\longrightarrow U\times\Omega^\infty\Sigma^{-d}KU,
\qquad
\pi_!^{KO}: \mathcal C^{\mathbb R,\rm cs}_{\mathfrak X/U}
\longrightarrow U\times\Omega^\infty\Sigma^{-d}KO.
\]
A spin structure on the relative tangent bundle gives
$\pi_!^{KU,\mathrm{spin}}$ in the same way.
In relative real dimension $8n$, we shall use Bott
periodicity to write their targets simply as $U\times KU$ and $U\times KO$.

We shall also use the natural quadratic operation of
Joyce--Tanaka--Upmeier \cite[Subsection 2.4.4]{JoyceTanakaUpmeier2020},
\[
\operatorname{Ad}^i:K^i(X)\longrightarrow KO^i(X).
\]
For a principal $U(m)$-bundle $P\to X$, it satisfies
$\operatorname{Ad}^0([P])=[\operatorname{Ad}(P)]$.  At the level of
representing spaces it is induced, up to homotopy, by a quadratic map
\[
\operatorname{Ad}:BU\times\mathbb Z\longrightarrow BO\times\mathbb Z,
\]
and hence gives fiberwise maps
\[
\operatorname{Ad}_{\mathfrak X/U}:
\mathcal C_{\mathfrak X/U}^{\rm cs}
\longrightarrow\mathcal C_{\mathfrak X/U}^{\mathbb R,\rm cs}.
\]
If $c:BO\times\mathbb Z\to BU\times\mathbb Z$ denotes complexification,
then
\[
c\circ\operatorname{Ad}(\alpha)\simeq\alpha^\vee\otimes\alpha.
\]

\subsection{Algebraic and topological moduli stacks}\label{stacks}
Let $X$ be any smooth proper complex variety of complex dimension $4n$.

We have already introduced the space
\[
\mathcal C_X=\operatorname{Map}(X^{\rm an},BU\times\mathbb Z).
\]
Its algebraic counterpart is the derived moduli stack $\mathcal M_X$ of
perfect complexes on $X$.  There is a natural comparison map
\[
\Gamma_X:\mathcal M_X^{\rm top}\longrightarrow\mathcal C_X,
\]
obtained by sending a perfect complex to its underlying topological
$K$-theory class; see \cite[Sections 2.2 and 2.3]{JoyceTanakaUpmeier2020}
and \cite[Sections 5.1--5.2]{3-folds}.

The pushforward in complex $K$-theory defined by the complex structure gives
a natural map
\[
\pi_!^{KU,\mathbb C}:\mathcal C_X\longrightarrow BU\times\mathbb Z.
\]
It has the algebraic analogue
\[
\pi_*^{\rm alg}:\mathcal M_X\longrightarrow\operatorname{Perf}(\mathbb C).
\]
Both maps are obtained by pushing forward the corresponding universal
perfect complex.

We record two variants of this construction.  First, suppose that $X$ is
smooth but may be non-compact.  We write $\mathcal M_X^{\rm cs}$ for the
derived moduli stack of compactly supported perfect complexes on $X$, and
\[
\mathcal C_X^{\rm cs}=\operatorname{Map}_*(X^+,BU\times\mathbb Z),
\]
where $X^+$ denotes the one-point compactification.  As in
\cite[Section 4.6]{Karpov}, there is a natural comparison map
\[
\Gamma_X:\bigl(\mathcal M_X^{\rm cs}\bigr)^{\rm top}
\longrightarrow\mathcal C_X^{\rm cs}.
\]
It can be constructed using any smooth compactification of $X$; independence
of the compactification follows from Hironaka's resolution of singularities.
All the constructions above extend canonically to this setting. 

(Two more constructions of the same map follow from the arguments in Section \ref{sec:two-pushforward-squares} (using motivic tools), and in Appendix \ref{sec:appendix-analytic-comparison}.)

Note that the pushforward in compactly supported topological $K$-theory is
intrinsic to $X$ and does not depend on a compactification or on any other
auxiliary choice.

Second, we shall need a relative version.  Let
$\pi:\mathfrak X\to U$ be a topologically locally trivial family of smooth
complex algebraic varieties, whose fibers may be non-compact.  As in
Subsection~\ref{sec:push}, one forms the relative $K$-theory space
$\mathcal C_{\mathfrak X/U}^{\rm cs}$ with fiber
$\mathcal C_{\mathfrak X_u}^{\rm cs}$ over $u\in U$.  Its algebraic
counterpart is the relative derived moduli stack
$\mathcal M_{\mathfrak X/U}^{\rm cs}$ of fiberwise compactly supported
perfect complexes.  The preceding comparison maps assemble into
\[
\Gamma_{\mathfrak X/U}:
\bigl(\mathcal M_{\mathfrak X/U}^{\rm cs}\bigr)^{\rm top}
\longrightarrow\mathcal C_{\mathfrak X/U}^{\rm cs}.
\]

\begin{remark}\label{H4rem}
\begin{enumerate}
\item The comparison map $\Gamma_X$ is canonical up to homotopy, but not in
general up to a unique homotopy; compare
\cite[Section 2]{JoyceTanakaUpmeier2020}.  This ambiguity is immaterial for
the constructions below, as in \cite[2.3.5]{JoyceTanakaUpmeier2020}.
\item The same construction applies to other representing spaces.  In
particular, for a smooth topologically locally trivial family
$\mathfrak X\to U$ one may form the relative space
$\mathcal H^4_{\rm cs}(\mathfrak X/U)$ whose fiber over $u\in U$ is
\[
\operatorname{Map}_*(\mathfrak X_u^+,K(\mathbb Z,4)).
\]
\end{enumerate}
\end{remark}

\subsection{Algebraic and topological orientations}
\label{sec:algebraic-topological-orientations}
\subsubsection{Comparing orientations}\label{compstatement}
Let $X$ be a smooth complex algebraic variety of complex dimension $4n$.
Write $\mathcal E$ for the universal perfect complex on
$\mathcal M_X\times X$, and write $\mathcal E^{\rm top}$ for its
topological counterpart on $\mathcal C_X\times X$.  Define the determinant
line bundles
\[
K_{\mathcal M_X}=\det\,\mathbf R\pi_{\mathcal M*}
 \mathbf R\mathcal Hom(\mathcal E,\mathcal E),
\qquad
K_{\mathcal C_X}=\det\,\pi_{\mathcal C!}^{KU,\mathbb C}
 \bigl((\mathcal E^{\rm top})^\vee\otimes\mathcal E^{\rm top}\bigr).
\]
This is the notation of Joyce--Upmeier; see
\cite[Sections 5.1--5.2]{3-folds} and
\cite[Section 3.5]{JU}.

Let $\pi:\mathfrak X\to U$ be a relatively Calabi--Yau family as in
Subsection~\ref{stacks}. Topological Serre duality (see \cite{Miller,Dyer}
and, for a categorified version, \cite[Appendix A]{Karpov}) gives a square
which commutes up to homotopy
\[
\begin{tikzcd}
\mathcal C_{\mathfrak X/U}^{\rm cs}
  \arrow[r,"(-)^\vee"] \arrow[d,"\pi_!^{KU,\mathbb C}"'] &
\mathcal C_{\mathfrak X/U}^{\rm cs}
  \arrow[d,"\pi_!^{KU,\mathbb C}"]\\
U\times KU \arrow[r,"(-)^\vee"'] & U\times KU.
\end{tikzcd}
\]
Taking $U=\operatorname{pt}$ and $\mathfrak X=X$, with $X$ Calabi--Yau,
Serre duality therefore supplies non-degenerate quadratic forms
\[
Q_{\mathcal M_X}:K_{\mathcal M_X}^{\otimes2}\xrightarrow{\cong}
\mathcal O_{\mathcal M_X},
\qquad
Q_{\mathcal C_X}:K_{\mathcal C_X}^{\otimes2}\xrightarrow{\cong}
\underline{\mathbb C}_{\mathcal C_X}.
\]

\begin{definition}
The algebraic and topological orientation bundles
$O^{\omega}_{\mathcal M_X}\to\mathcal M_X$ and
$O^{\not D}_{\mathcal C_X}\to\mathcal C_X$ are the principal
$\mathbb Z_2$-bundles whose fibers consist of square roots of
$Q_{\mathcal M_X}$ and $Q_{\mathcal C_X}$, respectively.  Equivalently, the
fiber over a point is the set of isomorphisms $\phi:K\xrightarrow{\cong}
\mathbb C$ for which
\[
\begin{tikzcd}
K^{\otimes2}\arrow[r,"Q"]\arrow[d,"\phi^{\otimes2}"']&\mathbb C
\arrow[d,"\mathrm{id}"]\\
\mathbb C^{\otimes2}\arrow[r,"m"']&\mathbb C
\end{tikzcd}
\]
commutes.
\end{definition}

\begin{remark}
All the material in this subsection has a natural relative version for
relatively Calabi--Yau, topologically locally trivial families, whose fibers
may be non-compact, as in Subsection~\ref{stacks}.  For example, the topological
determinant line is classified by the composite
\[
\mathcal C_{\mathfrak X/U}^{\rm cs}
\xrightarrow{\pi_!^{KU,\mathbb C}}U\times KU
\xrightarrow{\operatorname{pr}_2}KU\xrightarrow{\det}B\mathbb C^*.
\]
\end{remark}

\subsubsection{The projective comparison theorem}\label{proper-comparison}
Suppose that $X$ is projective and Calabi--Yau. Joyce and Upmeier prove the
following comparison;
see \cite[Theorem 5.4]{3-folds} and
\cite[Section 3.5]{JU}.

\begin{theorem}\label{origJU}
There is a canonical isomorphism of principal $\mathbb Z_2$-bundles
\[
I_X:\Gamma_X^*O^{\not D}_{\mathcal C_X}
\xrightarrow{\cong}O^{\omega}_{\mathcal M_X},
\]
compatible with direct sums.  It is induced by the canonically
homotopy-commutative diagram
\[
\begin{tikzcd}
\mathcal M_X^{\rm top}\arrow[r,"\Gamma_X"]
 \arrow[d,"\pi_*^{\rm alg}"']&
\mathcal C_X\arrow[d,"\pi_!^{KU,\mathbb C}"]\\
\operatorname{Perf}(\mathbb C)^{\rm top}
 \arrow[dr,"\det"']&KU\arrow[d,"\det"]\\
&B\mathbb C^* ,
\end{tikzcd}
\]
and this diagram is compatible with the Serre-duality quadratic forms.
\end{theorem}

\begin{remark}
\begin{enumerate}
 Joyce and Upmeier define $O^{\not D}_{\mathcal C_X}$ analytically, using
families of twisted Dirac operators.  However, the families Atiyah--Singer index theorem
identifies their construction with the topological $K$-theoretic construction
above.

\end{enumerate}
\end{remark}

\subsubsection{The relative projective comparison theorem}\label{relative-proper-comparison}
The Joyce--Upmeier comparison is continuous in the complex structure.  Thus the proof of
Theorem~\ref{origJU}, together with Ehresmann's theorem, gives the following
relative statement.

\begin{theorem}\label{compfamc}
Let $\pi:\mathfrak X\to U$ be a smooth proper family of complex varieties of
complex dimension $4n$, equipped with a relative Calabi--Yau structure.
Suppose that all fibers are projective and the total space is quasi-projective.
Then there is a canonical isomorphism
\[
I_{\mathfrak X/U}:
\Gamma_{\mathfrak X/U}^*O^{\not D}_{\mathcal C_{\mathfrak X/U}}
\xrightarrow{\cong}O^{\omega}_{\mathcal M_{\mathfrak X/U}}.
\]
It is induced by the homotopy-commutative diagram
\[
\begin{tikzcd}
\mathcal M_{\mathfrak X/U}^{\rm top}
 \arrow[r,"\Gamma_{\mathfrak X/U}"]
 \arrow[d,"\pi_*^{\rm alg}"']&
\mathcal C_{\mathfrak X/U}\arrow[d,"\pi_!^{KU,\mathbb C}"]\\
U\times\operatorname{Perf}(\mathbb C)^{\rm top}
 \arrow[d,"\operatorname{pr}_2"']&U\times KU
 \arrow[d,"\operatorname{pr}_2"]\\
\operatorname{Perf}(\mathbb C)^{\rm top}
 \arrow[dr,"\det"']&KU\arrow[d,"\det"]\\
&B\mathbb C^* .
\end{tikzcd}
\]
This diagram is compatible with the relative Serre duality.
\end{theorem}

\subsection{Comparison without quasi-projectivity}
\label{sec:comparison-without-quasi-projectivity}

The comparison in Theorem \ref{compfamc} uses quasi-projectivity through
the results of Joyce and Upmeier.  We now prove the comparison for smooth
proper algebraic varieties without this assumption.  The main point is to
compare derived direct image with the pushforward in complex topological
$K$-theory.

In the next subsection we prove the corresponding statement for non-proper varieties.

However, this is not enough for the purposes of this note: below we are concerned with complex-analytic (not necessarily algebraic) varieties as well. 

Thus, the main theorems of this subsection, as well as the next one, are reproved in Appendix~\ref{sec:appendix-analytic-comparison} for complex-analytic varieties using different and independent methods. Despite that, we decided to retain the algebraic proofs as well: they are short, arguably elegant, and help to better understand the more intricate analytic proofs in the appendix.

\subsubsection{Blanc's topological $K$-theory}
\label{sec:blanc-recollection}

Let $\mathcal T$ be a $\mathbb C$-linear dg-category.  Blanc associates to it
the presheaf
\[
M_{\mathcal T}(A)=Nw\operatorname{Perf}
 \bigl(\mathcal T\otimes_{\mathbb C}^{\mathbf L}A\bigr)
\]
on affine schemes of finite type over $\mathbb C$, together with the 
$\Gamma$-presheaf $M_{\mathcal T}^{\bullet}$ induced by direct sums.  Here
$Nw$ denotes the nerve of the subcategory of equivalences.  The associated
connective spectrum is
$E_{\mathcal T}=B|M_{\mathcal T}^{\bullet}|_{\Gamma}$.  The map from objects
to Waldhausen $K$-theory constructed in
\cite[Subsection~4.3]{Blanc} induces a natural equivalence
\[
E_{\mathcal T}\simeq \widetilde K^{\rm st}(\mathcal T);
\]
see \cite[Proposition~4.20 and Theorem~4.21]{Blanc}.  

For a smooth complex variety $X$, put
$\mathcal T_X=\operatorname{Perf}_{\rm dg}(X)$.  The natural functor
$\mathcal T_X\otimes_{\mathbb C}^{\mathbf L}A\to
\operatorname{Perf}_{\rm dg}(X\times\operatorname{Spec}A)$ is a Morita
equivalence \cite[equation~(11)]{Blanc}.  Thus $M_{\mathcal T_X}$ identifies
naturally with the moduli stack $\mathcal M_X$ used above.  Segal's
equivalence and Blanc's construction give a natural map
\[
\eta_X:\mathcal M_X^{\rm top}\longrightarrow
\Omega^\infty K^{\rm top}(\operatorname{Perf}X).
\]
More precisely, one first maps to
$\Omega^\infty\widetilde K^{\rm st}(\operatorname{Perf}X)$, then uses the
natural map to nonconnective semi-topological $K$-theory and inverts the Bott
class.  Since $X$ is smooth, the connective and nonconnective theories agree
before Bott inversion; this follows from Blanc's restriction theorem and the
vanishing of negative $K$-groups for regular schemes; see
\cite[Theorem~3.18, Remark~4.4, and Theorem~4.6]{Blanc}.

Blanc's geometric comparison is a natural equivalence
\[
\rho_X:K^{\rm top}(\operatorname{Perf}X)
 \xrightarrow{\ \simeq\ }
F(\Sigma^\infty_+X^{\rm an},KU).
\]
See \cite[Proposition~4.32]{Blanc} and
\cite[Propositions~7.6--7.8]{Moulinos}.
Here $F(-,-)$ denotes the function spectrum.  Under
$\mathcal C_X\simeq\Omega^\infty
F(\Sigma^\infty_+X^{\rm an},KU)$, the comparison map of
Subsection~\ref{stacks} is
$\Gamma_X=\Omega^\infty\rho_X\circ\eta_X$.

\subsubsection{The two squares}
\label{sec:two-pushforward-squares}

\begin{proposition}\label{prop:pushforward-without-quasiprojectivity}
Let $f:Z\to W$ be a proper smoothable morphism between smooth proper complex
varieties.  Thus $f$ admits a factorization through a closed immersion followed
by a smooth morphism.  Derived direct image defines an exact functor
$\mathbf Rf_*:\operatorname{Perf}(Z)\to\operatorname{Perf}(W)$ and hence a
map $\mathcal M_Z\to\mathcal M_W$.  Then the diagram
\[
\begin{tikzcd}
\mathcal M_Z^{\rm top}
 \arrow[r,"(\mathbf Rf_*)^{\rm top}"]
 \arrow[d,"\Gamma_Z"']&
\mathcal M_W^{\rm top}
 \arrow[d,"\Gamma_W"]\\
\mathcal C_Z\arrow[r,"f_!^{KU,\mathbb C}"']&\mathcal C_W
\end{tikzcd}
\]
commutes up to a natural homotopy, compatibly with composition of proper
smooth morphisms.
\end{proposition}

With the factorization of $\Gamma$ recalled above, this is the outer
rectangle in
\[
\begin{tikzcd}[column sep=7em,row sep=3em]
\mathcal M_Z^{\rm top}
 \arrow[r,"(\mathbf Rf_*)^{\rm top}"]
 \arrow[d,"\eta_Z"']&
\mathcal M_W^{\rm top}
 \arrow[d,"\eta_W"]\\
\Omega^\infty K^{\rm top}(\operatorname{Perf}Z)
 \arrow[r,"\Omega^\infty K^{\rm top}(\mathbf Rf_*)"]
 \arrow[d,"\Omega^\infty\rho_Z"']&
\Omega^\infty K^{\rm top}(\operatorname{Perf}W)
 \arrow[d,"\Omega^\infty\rho_W"]\\
\mathcal C_Z\arrow[r,"f_!^{KU,\mathbb C}"']&\mathcal C_W.
\end{tikzcd}
\]
The upper square expresses the naturality of Blanc's map from objects to
$K$-theory.  The lower square expresses the compatibility of Betti
realization with proper pushforwards.  We prove these two statements in turn.

\subsubsection{Naturality of the map from objects to $K$-theory}
\label{sec:blanc-object-naturality}

We prove that the upper square in Subsection~\ref{sec:two-pushforward-squares}
commutes up to a natural homotopy. The essential observation is that an
exact functor preserves finite direct sums. We explain how this applies
to the constructions used to define $\eta$.

Write $F=\mathbf Rf_*:\mathcal T_Z\to\mathcal T_W$. For an affine test
scheme $\operatorname{Spec}A$ and $X=Z,W$, put
\[
\mathcal D_{X,A}=\operatorname{Perf}
 (\mathcal T_X\otimes_{\mathbb C}^{\mathbf L}A).
\]
Extension of scalars gives an exact functor
$F_{A,!}:\mathcal D_{Z,A}\to\mathcal D_{W,A}$.
Under the Morita equivalences of
Subsection~\ref{sec:blanc-recollection}, this is
$\mathbf R(f\times1_{\operatorname{Spec}A})_*$. Indeed, the perfect kernel
$\mathcal O_{\Gamma_f}$ representing $\mathbf Rf_*$ pulls back to the graph
kernel for $f\times1_{\operatorname{Spec}A}$.

Recall that $M_{\mathcal T_X}^{\bullet}$ records objects and their direct
sums. For fixed $A$, its value at $n_+=\{0,1,\ldots,n\}$ is naturally equivalent to
$Nw(\mathcal D_{X,A}^{\,n})$. Thus its objects may be described as tuples
$(E_1,\ldots,E_n)$, with the coherent choices of sums included in the
$\Gamma$-construction. On the other hand, Waldhausen's category
$S_n\mathcal D_{X,A}$ records filtrations by cofibrations
$0=Q_0\hookrightarrow Q_1\hookrightarrow\cdots\hookrightarrow Q_n$,
together with their quotients.

The map comparing these two descriptions is
\[
\lambda_{X,A,n}:M_{\mathcal T_X}^{\bullet}(A)(n_+)
 \longrightarrow NwS_n\mathcal D_{X,A}.
\]
Here $X$ specifies the variety, $A$ the test algebra, and $n$ the length
of the filtration. The map sends a tuple to the filtration by its partial
sums:
\[
(E_1,\ldots,E_n)\longmapsto
\bigl(0\hookrightarrow E_1\hookrightarrow E_1\oplus E_2
\hookrightarrow\cdots\hookrightarrow E_1\oplus\cdots\oplus E_n\bigr).
\]
Its $i$th successive quotient is $E_i$. Applying the same construction to
chains of equivalences defines the map on the nerves. The sum diagrams
ensure compatibility with the simplicial operators. As $A$ and $n$ vary,
these maps give Blanc's map $\lambda_{X,\bullet}$ from the simplicial
presheaf underlying $M_{\mathcal T_X}^{\bullet}$ to
$A\mapsto NwS_\bullet\mathcal D_{X,A}$.

The functor $F_{A,!}$ acts on tuples term by term, and on filtrations by
applying it to every term and quotient. Since it is exact, there are
canonical equivalences
\[
F_{A,!}(E_1\oplus\cdots\oplus E_i)
 \simeq F_{A,!}E_1\oplus\cdots\oplus F_{A,!}E_i.
\]
These identify the image of the filtration associated to
$(E_1,\ldots,E_n)$ with the filtration associated to
$(F_{A,!}E_1,\ldots,F_{A,!}E_n)$. Hence the square
\[
\begin{tikzcd}[column sep=4em]
M_{\mathcal T_Z}^{\bullet}(A)(n_+)
 \arrow[r,"\lambda_{Z,A,n}"] \arrow[d,"F_{A,!}"']&
NwS_n\mathcal D_{Z,A}\arrow[d,"F_{A,!}"]\\
M_{\mathcal T_W}^{\bullet}(A)(n_+)
 \arrow[r,"\lambda_{W,A,n}"']&NwS_n\mathcal D_{W,A}
\end{tikzcd}
\]
commutes up to a natural homotopy; the vertical arrows denote the induced
maps. The homotopy is compatible with changes of $A$ and with the
simplicial operators, by the naturality and associativity of direct sums.

Blanc also uses the map of special $\Gamma$-spaces
\[
\sigma_{X,A}:M_{\mathcal T_X}^{\bullet}(A)
 \longrightarrow K^\Gamma(\mathcal D_{X,A}),
\]
where $K^\Gamma(\mathcal D)$ is the $\Gamma$-space representing Waldhausen
$K$-theory. This is the standard map from objects to $K$-theory, with its
compatibility with direct sums; see \cite[equation~(1)]{Blanc}.
It is natural in exact functors. The operation
$\sigma$ becomes an equivalence after topological realization: for each
$n$, Blanc proves that $\lambda_{X,A,n}$, as $A$ varies, is an
$\mathbb A^1$-equivalence of presheaves, and then deduces the equivalence
for $\sigma$ in \cite[Proposition~4.17 and Theorem~4.21]{Blanc}.

Applying topological realization to the naturality square for $\sigma$,
and then Segal's construction, gives the naturality of the map
$\mathcal M_X^{\rm top}\to
\Omega^\infty\widetilde K^{\rm st}(\operatorname{Perf}X)$.
The passage to nonconnective semi-topological $K$-theory and Bott inversion
are also functorial. We obtain the required square for $\eta$, with its
natural homotopy. All these constructions respect composition of exact
functors, giving the asserted compatibility with compositions of proper
morphisms.

\subsubsection{Compatibility with proper pushforwards}
\label{sec:betti-proper-pushforward}

We now prove that the lower square in the second diagram of
Subsection~\ref{sec:two-pushforward-squares} commutes already at the level
of spectra:
\[
\begin{tikzcd}
K^{\rm top}(\operatorname{Perf}Z)
 \arrow[r,"K^{\rm top}(\mathbf Rf_*)"]
 \arrow[d,"\rho_Z"',"\simeq"]&
K^{\rm top}(\operatorname{Perf}W)
 \arrow[d,"\rho_W","\simeq"']\\
F(\Sigma^\infty_+Z^{\rm an},KU)
 \arrow[r,"f_!^{KU,\mathbb C}"']&
F(\Sigma^\infty_+W^{\rm an},KU).
\end{tikzcd}
\]

By Blanc's restriction theorem, the realization may be computed on smooth
affine tests $U$.  Before realization the upper map is
\[
K\bigl(\operatorname{Perf}(Z\times U)\bigr)
\xrightarrow{\ K(\mathbf R(f\times1_U)_*)\ }
K\bigl(\operatorname{Perf}(W\times U)\bigr).
\]
The schemes $Z\times U$ and $W\times U$ are regular.  Algebraic $K$-theory
therefore agrees with $G$-theory, compatibly with proper direct image.

For every proper morphism $f$, Jin constructs a morphism
$\varphi_f:f_*\mathrm{GGL}_Z\to\mathrm{GGL}_W$ representing coherent direct
image.  Its effect on $G$-theory and its compatibility with composition are
proved in \cite[Lemma~3.1.3, Subsection~3.1.4, and
Lemma~3.3.4]{Jin}.  For a smoothable morphism, the adjoint morphism
$\mathrm{GGL}_Z\to f^!\mathrm{GGL}_W$ agrees with Jin's morphism $\xi_f$ by
\cite[Propositions~3.2.9 and~3.2.13]{Jin}.  On a regular scheme one has
$\mathrm{GGL}\simeq\mathrm{KGL}$.  Consequently, $\varphi_f$ is identified
with the $\mathrm{KGL}$ pushforward obtained from purity and the counit for
$f_*\dashv f^!$.  Equivalently, this follows from the compatibility of Jin's
identification of $G$-theory with Borel--Moore $\mathrm{KGL}$-theory with
proper direct image and Gysin morphisms
\cite[Corollary~3.3.7]{Jin}.

Betti realization sends $\mathrm{KGL}$ to $KU$, commutes with proper direct
image, and sends the algebraic Thom class
of a complex vector bundle to its $KU$ Thom class; see
\cite[Example~3.11 and Proposition~A.4]{BRTV}.  It therefore sends
$\varphi_f$ to $f_!^{KU,\mathbb C}$.

We now spell out the identification of the sources and targets after realization.
For $X=Z,W$, let $p_X:X\to\operatorname{Spec}\mathbb C$ and let
$\mathbf 1_X=\operatorname{Perf}(X)$ be the unit
$\operatorname{Perf}(X)$-linear category.  Proposition~7.7 of
\cite{Moulinos} identifies the relative topological $K$-theory sheaf
$K_X^{\rm top}(\mathbf 1_X)$ with
$\operatorname{Betti}_X(\mathrm{KGL}_X)$.  Proposition~7.8, applied to the
proper morphism $p_X$, and Proposition~7.6 then give
\[
K^{\rm top}(\operatorname{Perf}X)
\simeq
\Gamma\bigl(X^{\rm an},K_X^{\rm top}(\mathbf 1_X)\bigr)
\simeq
F(\Sigma^\infty_+X^{\rm an},KU).
\]
These are the equivalences $\rho_Z$ and $\rho_W$ in the preceding square.
Moreover,
$$\Gamma(W^{\rm an},f_*K_Z^{\rm top}(\mathbf 1_Z))
=\Gamma(Z^{\rm an},K_Z^{\rm top}(\mathbf 1_Z)).$$  Hence taking global
sections of the realization of $\varphi_f$ gives precisely the map from the
function spectrum of $Z^{\rm an}$ to that of $W^{\rm an}$ denoted by
$f_!^{KU,\mathbb C}$.  This proves the commutativity of the square.  Jin's composition
formula gives the asserted compatibility
with composition.  Together with
Subsection~\ref{sec:blanc-object-naturality}, this proves
Proposition~\ref{prop:pushforward-without-quasiprojectivity}.

\subsubsection{Determinant bundles}
\label{sec:determinants-without-quasi-projectivity}

Let $X$ be a smooth proper complex variety and apply
Proposition~\ref{prop:pushforward-without-quasiprojectivity} to the projection
$\pi_X:X\to\operatorname{pt}$.  Since
$\Gamma_{\operatorname{pt}}:\operatorname{Perf}(\mathbb C)^{\rm top}
\xrightarrow{\simeq}BU\times\mathbb Z$ is the map from a perfect complex to
its topological $K$-theory class, we obtain a natural homotopy
\[
\Gamma_{\operatorname{pt}}\circ\pi_*^{\rm alg}
\simeq
\pi_!^{KU,\mathbb C}\circ\Gamma_X.
\]
Under $\Gamma_{\operatorname{pt}}$, the topological realization of the determinant morphism
$\operatorname{Perf}(\mathbb C)\to B\mathbb G_m$ is the usual determinant
$BU\times\mathbb Z\to B\mathbb C^*$. 

Analytification respects derived duals and tensor products.  We may therefore
apply the preceding homotopy to
$\mathbf R\mathcal Hom(\mathcal E,\mathcal E)$, where $\mathcal E$ is the
universal perfect complex on $\mathcal M_X\times X$.  Applying the
determinant gives a natural isomorphism of complex line bundles
\[
\Gamma_X^*K_{\mathcal C_X}
\xrightarrow{\ \cong\ }
(K_{\mathcal M_X})^{\rm top}.
\]
Thus the determinant bundles defined in
Subsection~\ref{compstatement} agree for every smooth proper complex
variety, without any quasi-projectivity assumption.  The construction is
compatible with direct sums and with base change.  In particular, the same
argument applies to the determinant bundles of a smooth proper family.

\subsubsection{Compatibility with Serre duality}
\label{sec:serre-duality-without-quasi-projectivity}

We finally explain why the comparison of determinant bundles respects Serre
duality.

The main point is that the use of $K$-theory gives a parallel description of algebraic and topological dualities.

Let $d=\dim_{\mathbb C}X$, and write
$D_X(F)=\mathbf R\mathcal Hom(F,K_X[d])$ and
$D_{\mathbb C}(V)=\mathbf R\mathcal Hom_{\mathbb C}(V,\mathbb C)$.
The algebraic Serre isomorphism
$s_F^{\rm alg}:\mathbf Rp_*D_X(F)\xrightarrow{\simeq}
D_{\mathbb C}(\mathbf Rp_*F)$ is adjoint to the pairing
\[
\mathbf Rp_*F\otimes\mathbf Rp_*D_X(F)
\longrightarrow
\mathbf Rp_*\bigl(F\otimes D_X(F)\bigr)
\longrightarrow
\mathbf Rp_*K_X[d]
\xrightarrow{\operatorname{Tr}_p}\mathbb C.
\]
The three arrows are respectively the product supplied by the projection
formula, evaluation, and the counit of Grothendieck duality.  In the
terminology of \cite{CHN}, they form the Poincar\'e functor associated with
$\mathbf Rp_*$.

One readily describes the topological counterpart of this operation explicitly.  Let
$\tau$ denote dualization.  If $V$ is a complex vector bundle of rank $c$,
the standard $KU$-Thom class satisfies
\[
\tau u^{KU}(V)\simeq(-1)^c[\det V]u^{KU}(V),
\qquad
\tau u^{KU}(-T_X)\simeq
(-1)^d[K_X^{\rm an}]u^{KU}(-T_X).
\]
Indeed, the first identity follows directly from
$u^{KU}(V)=\lambda_{-1}(\pi^*V^\vee)$, and the second is its multiplicative
extension to virtual bundles, using
$\det(-T_X)=\det(T_X)^{-1}=K_X^{\rm an}$.  Put
$S_X(x)=(-1)^d[K_X^{\rm an}]\tau(x)$.  Dualizing the Pontryagin--Thom
construction of the trace now gives the homotopy
\[
s_X^{KU}:p_!^{KU,\mathbb C}S_X(x)
\simeq\tau p_!^{KU,\mathbb C}(x).
\]

(This is just a fancy way to rephrase the argument from \cite{Miller}; a full explanation of the spectra-level nature of this construction, and its precise origin in Atiyah-Bott-Shapiro comparison will appear in \cite{Karpov}.)

It is easy (using the description of dualities above) to see that
the construction underlying Subsection~\ref{sec:betti-proper-pushforward} may
be carried out while retaining the dualities, and that the resulting
homotopy carries a natural $2$-cell between the two Serre dualities.

More explicitly, let $H_X$ denote the homotopy comparing
$(\mathbf Rp_*)^{\rm top}$ with $p_!^{KU,\mathbb C}$ obtained in
Subsection~\ref{sec:betti-proper-pushforward}.  Analytification of duals,
tensor products, and shifts identifies $\Gamma_XD_X$ with
$S_X\Gamma_X$.  With this identification understood, the preceding argument
gives a homotopy 
\[
s_X^{KU}\circ(H_XD_X)
\simeq
(\tau H_X)\circ(s_X^{\rm alg})^{\rm top}.
\]
This is the precise compatibility with Serre duality needed below. An alternative (complex-analytic and ``transcendental'') proof is recorded in Appendix~\ref{sec:appendix-analytic-comparison}.

\subsection{The non-proper situation}

We extend here the results of the previous subsection to the case of non-proper algebraic varieties. The alternative complex-analytic proof is recorded in Appendix~\ref{sec:appendix-analytic-comparison}.

\subsubsection{The non-proper comparison theorem }
\label{compactly-supported-comparison}

\begin{proposition}\label{compnc2}
Let $X$ be a smooth quasi-projective variety over $\mathbb C$. Then the
diagram
\[
\begin{tikzcd}
	{(\mathcal M_X^{\rm cs})^{\rm top}} &&
	{\mathcal C_X^{\rm cs}} \\
	{\operatorname{Perf}(\mathbb C)^{\rm top}} && KU \\
	& {B\mathbb C^*}
	\arrow["{\Gamma_X}", from=1-1, to=1-3]
	\arrow["{\pi_*^{\rm alg}}"', from=1-1, to=2-1]
	\arrow["{\pi_!^{KU,\mathbb C}}", from=1-3, to=2-3]
	\arrow["{\det}"', from=2-1, to=3-2]
	\arrow["{\det}", from=2-3, to=3-2]
\end{tikzcd}
\]
canonically commutes.
\end{proposition}

\begin{proof}
\textbf{Step 0.} The proof is similar to the argument of \cite{BFM}. The main
statement of that paper implies that, if $X$ is a smooth projective variety, the diagram
\[
\begin{tikzcd}
	{K^{\rm alg}(X)} & {K^{\rm top}(X)} \\
	{K^{\rm alg}(\operatorname{pt})} &
	{K^{\rm top}(\operatorname{pt})}
	\arrow["{\Gamma_X}", from=1-1, to=1-2]
	\arrow["{\pi_{X*}^{\rm alg}}"', from=1-1, to=2-1]
	\arrow["{\pi_{X!}^{KU,\mathbb C}}", from=1-2, to=2-2]
	\arrow["{\Gamma_{\operatorname{pt}}}"', from=2-1, to=2-2]
\end{tikzcd}
\]
commutes.

Thus there are, as one can see, two main differences from our context:
\begin{enumerate}
    \item In their proof, the authors work not with $K$-theory but with
    so-called \textit{$K$-homology}, which, however, coincides with $K$-theory
    for smooth varieties. They also work at the level of $K$-groups, rather
    than their categorification, that is, rather than at the level of mapping
    stacks.
    \item They do not consider the compactly supported variant of the theory.
\end{enumerate}

However, inspecting the proof in \cite{BFM}, one sees that these differences
do not affect the argument, and the proof generalizes verbatim.

\textbf{Step 1.} For the convenience of the reader, we sketch the argument.
First, embed $X$ into $\mathbb P^n$ via an embedding
$\iota:X\to\mathbb P^n$, and prove that the diagram
\[
\begin{tikzcd}\label{BFMdiag}
	{(\mathcal M_X^{\rm cs})^{\rm top}} &
	{\mathcal M_{\mathbb P^n}^{\rm top}} \\
	{\mathcal C_X^{\rm cs}} & {\mathcal C_{\mathbb P^n}}
	\arrow["{\iota_*^{\rm alg}}", from=1-1, to=1-2]
	\arrow["{\Gamma_X}"', from=1-1, to=2-1]
	\arrow["{\Gamma_{\mathbb P^n}}", from=1-2, to=2-2]
		\arrow["{\iota_!^{KU,\mathbb C}}"', from=2-1, to=2-2]
\end{tikzcd}
\]
commutes, as in \cite[Sections 2, 4]{BFM}.

\textbf{Step 2.} Then observe the factorization of the morphism
$X\to\operatorname{pt}$ as
\[
X\xrightarrow{\iota}\mathbb P^n\longrightarrow\operatorname{pt},
\]
and, to deal with the second arrow, apply Proposition~\ref{prop:pushforward-without-quasiprojectivity} to the smooth
variety $\mathbb P^n$.

The fact that the whole procedure is independent of the choice of the
embedding $\iota$ is proved as in \cite[Sections 2, 4]{BFM}, using deformation
to the normal cone.
\end{proof}

\subsubsection{The comparison in non-proper families}
\label{nonproper-family-comparison}

It will be very important for us to use the version of the constructions of
Subsection~\ref{compstatement} in families, as in Section \ref{stacks}.


As in Section \ref{stacks}, we always restrict ourselves here to \textit{topologically
locally trivial families}.

The first thing to observe is that there is still a natural comparison
morphism
\[
\Gamma_{\mathfrak X/U}:
(\mathcal M_{\mathfrak X/U}^{\rm cs})^{\rm top}
\longrightarrow
\mathcal C_{\mathfrak X/U}^{\rm cs}.
\]
Our work above generalizes, furthermore, to families as follows.

\begin{theorem}
Let $\pi:\mathfrak X\to U$ be a smooth but non-proper algebraic family which
is known to be topologically locally trivial. Then there is a natural
commutative diagram
\[
\begin{tikzcd}\label{diagfam2}
	{(\mathcal M_{\mathfrak X/U}^{\rm cs})^{\rm top}} &&
	{\mathcal C_{\mathfrak X/U}^{\rm cs}} \\
	{U\times\operatorname{Perf}(\mathbb C)^{\rm top}} && {U\times KU} \\
	& {B\mathbb C^*.}
	\arrow["{\Gamma_{\mathfrak X/U}}", from=1-1, to=1-3]
	\arrow["{\pi_*^{\rm alg}}"', from=1-1, to=2-1]
	\arrow["{\pi_!^{KU,\mathbb C}}", from=1-3, to=2-3]
	\arrow["{\det\circ\operatorname{pr}_2}"', from=2-1, to=3-2]
	\arrow["{\det\circ\operatorname{pr}_2}", from=2-3, to=3-2]
\end{tikzcd}
\]
\end{theorem}

\begin{proof}
One argues similarly to what was done above.

First, note that any family $\mathfrak X\to U$ as in the theorem admits a
description as the composition of a locally closed embedding
$\mathfrak X\to\mathfrak Y$ into a family of smooth projective varieties over
the base and the projection $\mathfrak Y\to U$.

Indeed, choose a smooth compactification $\mathfrak J$ of the total space of
$\mathfrak X$, and embed $\mathfrak X$ into $\mathfrak J\times U$ via the
graph map.

Since $\mathfrak X\to\mathfrak Y$ is fiberwise a locally closed embedding of
smooth varieties, one proves that the algebraic and topological pushforwards
canonically coincide. This is analogous to diagram~\ref{BFMdiag}.

The step $\mathfrak Y\to U$ is, by now, clear: one uses Proposition \ref{prop:pushforward-without-quasiprojectivity} in families.
\end{proof}

It is now  clear that this procedure commutes with taking Serre duality as required.

\subsection{\texorpdfstring{$KO$}{KO}-theory}\label{KOSec}

\subsubsection{The $KO$-theoretic interpretation}
\label{KO-theoretic-interpretation}

For \textit{any} smooth algebraic family
$\pi:\mathfrak X\to U$, not necessarily proper, we thus have a comparison
between the algebraic and topological determinant bundles.

We will be interested only in the latter, and now give its interpretation via
the pushforward operation in $KO$-theory, which turns out to be very
convenient. This trick can be regarded as a topological reformulation of the
analytic argument in \cite[Subsection 3.5]{CaoGrossJoyce2020}.

\begin{lemma}
Let $\pi:\mathfrak X\to U$ be a relatively Calabi--Yau, not necessarily
proper, topologically locally trivial family of complex $4n$-folds over an
algebraic space $U$.

We use below the spin structure determined by the relative Calabi--Yau
structure.

Define the real determinant line bundle
$L^{\not D}_{\mathcal C_{\mathfrak X/U}^{\rm cs}}$ by the classifying map
\[
\mathcal C_{\mathfrak X/U}^{\rm cs}
\xrightarrow{\operatorname{Ad}_{\mathfrak X/U}}
\mathcal C_{\mathfrak X/U}^{\mathbb R,\rm cs}
\xrightarrow{\pi_!^{KO}}
U\times KO
\xrightarrow{\operatorname{pr}_2}
KO
\xrightarrow{\det_{\mathbb R}}
\mathbb RP^\infty.
\]
Here $\operatorname{Ad}_{\mathfrak X/U}$ is the Joyce--Upmeier adjoint-class
map. Denote by
\[
O\bigl(L^{\not D}_{\mathcal C_{\mathfrak X/U}^{\rm cs}}\bigr)
=
\bigl(L^{\not D}_{\mathcal C_{\mathfrak X/U}^{\rm cs}}\setminus0\bigr)
/\mathbb R_{>0}
\]
the corresponding $\mathbb Z_2$-torsor.

Then there is a canonical isomorphism
\[
O^{\not D}_{\mathcal C_{\mathfrak X/U}^{\rm cs}}
\cong
O\bigl(L^{\not D}_{\mathcal C_{\mathfrak X/U}^{\rm cs}}\bigr).
\]
\end{lemma}

\begin{proof}
\textbf{Step 1.} It is enough to prove this in the case of a trivial family
$\pi$.

Then one has the following commutative diagram:
\[
\begin{tikzcd}\label{ConFl}
	{\mathcal C_{\mathfrak X/U}^{\mathbb R,\rm cs}} &&
	{\mathcal C_{\mathfrak X/U}^{\rm cs}} \\
	\\
	{U\times KO} && {U\times KU}
	\arrow["c", from=1-1, to=1-3]
	\arrow["{\pi_!^{KO}}"', from=1-1, to=3-1]
	\arrow["{\pi_!^{KU,\mathrm{spin}}}", from=1-3, to=3-3]
	\arrow["{\operatorname{id}_U\times c}"', from=3-1, to=3-3]
\end{tikzcd}
\]
This is the so-called ``Atiyah--Bott--Shapiro square.'' It commutes because
the restriction of the $\operatorname{Spin}^c$ orientation of $KU$ along
$\operatorname{Spin}\to\operatorname{Spin}^c$ is the complexification of the
spin orientation of $KO$.\footnote{In the rest of the present paper, we use
the word \textit{orientation} in a different sense.} See \cite{CF}, going back
to \cite{ABS}. For a modern treatment, see \cite{HK}. Since the relative
Calabi--Yau structure trivializes the relative canonical bundle,
\eqref{complex-spin-pushforward} identifies the right vertical map with
$\pi_!^{KU,\mathbb C}$.

The statement therefore follows from Steps 2 and 3 below.

\textbf{Step 2.} As we already mentioned, \textit{topological Serre duality} (see
\cite{Miller}, \cite{Dyer}, and, for a categorification,
\cite[Appendix A]{Karpov}) asserts that the following diagram commutes:
\[
\begin{tikzcd}
	{\mathcal C_{\mathfrak X/U}^{\rm cs}} &
	{\mathcal C_{\mathfrak X/U}^{\rm cs}} \\
	{U\times KU} & {U\times KU}
	\arrow["{(-)^\vee}", from=1-1, to=1-2]
	\arrow["{\pi_!^{KU,\mathbb C}}"', from=1-1, to=2-1]
	\arrow["{\pi_!^{KU,\mathbb C}}", from=1-2, to=2-2]
	\arrow["{(-)^\vee}"', from=2-1, to=2-2]
\end{tikzcd}
\]
where both the upper and lower horizontal arrows are induced by the relative
duality map.

It has a natural real counterpart, which is proved analogously:
\[
\begin{tikzcd}
	{\mathcal C_{\mathfrak X/U}^{\mathbb R,\rm cs}} &
	{\mathcal C_{\mathfrak X/U}^{\mathbb R,\rm cs}} \\
	{U\times KO} & {U\times KO}
	\arrow["{(-)^\vee}", from=1-1, to=1-2]
	\arrow["{\pi_!^{KO}}"', from=1-1, to=2-1]
	\arrow["{\pi_!^{KO}}", from=1-2, to=2-2]
	\arrow["{(-)^\vee}"', from=2-1, to=2-2]
\end{tikzcd}
\]

The Atiyah--Bott--Shapiro square above and
\eqref{complex-spin-pushforward} prove that the complex and real duality
diagrams are compatible. The lemma therefore follows from the following
elementary linear-algebraic statement.

\textbf{Step 3.} Let $V$ be a real one-dimensional vector space, and let
$V_{\mathbb C}$ be its complexification. Suppose that $V$ is equipped with a
non-degenerate pairing $\iota:V\otimes V\to\mathbb R$.

Denote by $\omega(\iota)$ the set of isomorphisms
$\kappa:V_{\mathbb C}\xrightarrow{\sim}\mathbb C$ for which the diagram
\[
\begin{tikzcd}
	{V_{\mathbb C}\otimes V_{\mathbb C}} & {\mathbb C} \\
	{\mathbb C\otimes\mathbb C} & {\mathbb C}
	\arrow["{\iota_{\mathbb C}}", from=1-1, to=1-2]
	\arrow["{\kappa\otimes\kappa}"', from=1-1, to=2-1]
	\arrow["{\operatorname{Id}}", from=1-2, to=2-2]
	\arrow["m"', from=2-1, to=2-2]
\end{tikzcd}
\]
commutes.

Then there is a natural equivalence of $\mathbb Z_2$-torsors
\[
V/\mathbb R_{>0}\cong\omega(\iota).
\]
\end{proof}

\section{Bordism categories and canonical orientations}
\label{sec:bordism-orientations}
\subsection{Bordism invariance}
\label{sec:cobordism-invariance}

\subsubsection{The open version}\label{open-cobordism}

Following the programme outlined in Subsection~\ref{compstatement}, and using
the results of Subsections~\ref{sec:push} and
\ref{KO-theoretic-interpretation}, we will be concerned with the following
situation.

Suppose that $X$ is a smooth, possibly non-compact, spin $8$-manifold.
Write
\[
\mathcal C_X^{\rm cs}
=\operatorname{Map}_*(X^+,BU\times\mathbb Z),
\qquad
\mathcal C_X^{\mathbb R,\rm cs}
=\operatorname{Map}_*(X^+,BO\times\mathbb Z).
\]
We shall write $O^{KO}_{\mathcal C_X^{\mathbb R,\rm cs}}$ for the principal
$\mathbb Z_2$-bundle classified by
\[
\mathcal C_X^{\mathbb R,\rm cs}
\xrightarrow{\pi_{X!}^{KO}}KO
\xrightarrow{\det_{\mathbb R}}\mathbb RP^\infty.
\]
Put $O^{\not D}_{\mathcal C_X^{\rm cs}}
=\operatorname{Ad}_X^*O^{KO}_{\mathcal C_X^{\mathbb R,\rm cs}}$.
If $X$ is a complex Calabi--Yau $4$-fold and its spin structure is determined
by the Calabi--Yau structure, the lemma of
Subsection~\ref{KO-theoretic-interpretation} identifies this bundle with
the topological orientation bundle of Subsection~\ref{compstatement}.
The goal of this subsection is to prove the following statement.

\begin{proposition}\label{closedcobordism}
Consider two closed smooth spin $8$-folds $X_1$ and $X_2$.

\begin{enumerate}
\item[\rm(a)] Let $\phi_i\in\mathcal C_{X_i}$ for $i=1,2$. Suppose that there
exists a spin bordism $W$ from $X_1$ to $X_2$, together with an element
$\phi\in\mathcal C_W$, which restricts to $\phi_1$ and $\phi_2$ on the
components of
\[
\partial W=-X_1\sqcup X_2,
\]
respectively. Then these data give rise to a canonical identification of
$\mathbb Z_2$-torsors
\[
\iota_{[W,\phi]}:
O^{\not D}_{\mathcal C_{X_1}}|_{\phi_1}
\xrightarrow{\sim}
O^{\not D}_{\mathcal C_{X_2}}|_{\phi_2}.
\]

\item[\rm(b)] Let
$\psi_i\in\mathcal C_{X_i}^{\mathbb R}$ for $i=1,2$. Suppose that there exists
a spin bordism $W$ from $X_1$ to $X_2$, together with an element
$\psi\in\mathcal C_W^{\mathbb R}$, which restricts to $\psi_1$ and $\psi_2$
on the components of $\partial W=-X_1\sqcup X_2$, respectively. Then these
data give rise to a canonical identification of $\mathbb Z_2$-torsors
\[
\iota^{KO}_{[W,\psi]}:
O^{KO}_{\mathcal C_{X_1}^{\mathbb R}}|_{\psi_1}
\xrightarrow{\sim}
O^{KO}_{\mathcal C_{X_2}^{\mathbb R}}|_{\psi_2}.
\]
\end{enumerate}
\end{proposition}

Before proving the proposition, let us recall that a spin bordism from
$X_1$ to $X_2$ is a spin $9$-manifold $W$ with boundary identified
with $-X_1\sqcup X_2$, and with product spin structure near the boundary.

\begin{proof}
It is enough to prove part~\textup{(b)}. Indeed, part~\textup{(a)} then follows
by pulling back along the map $\operatorname{Ad}$. We use the ideas of
Carvalho's argument; see \cite{Carvalho}.

\textbf{Step 1.} First, choose small collars $C_1$ and $C_2$ of the boundary
components in $W$. For $N$ sufficiently large, there
exists a neat embedding
\[
e:W\longrightarrow [0,1]\times\mathbb R^N
\]
which is also of product form on the collars. This follows from the relative
Whitney embedding theorem.

\textbf{Step 2.} As in the construction of
Subsection~\ref{sec:push}, one obtains a natural Pontryagin--Thom collapse map
\[
c:[0,1]_+\wedge S^N\longrightarrow\operatorname{Th}(\nu_e),
\]
whose restrictions to the boundary are the Pontryagin--Thom collapse maps
\[
c_0:\{0\}_+\wedge S^N\longrightarrow
\operatorname{Th}(\nu_e|_{X_1}),
\qquad
c_1:\{1\}_+\wedge S^N\longrightarrow
\operatorname{Th}(\nu_e|_{X_2}).
\]

Combining $c$ with the $KO$-theory Thom class and the class $\psi$, one may
interpret it as a path in $KO$ between
$\pi_{X_1!}^{KO}(\psi_1)$ and $\pi_{X_2!}^{KO}(\psi_2)$. Applying the real
determinant therefore gives the desired isomorphism.

\textbf{Step 3.} It remains to show that the isomorphism constructed in
Step~2 is canonical. For this, it is enough to show that the data used in
Step~1 (the collars, the tubular neighbourhood, and the neat embedding with
product structure on the collars) are unique up to homotopy after
stabilization. This is standard; see, for example, \cite[p.~22]{GMTW}.
\end{proof}

\subsubsection{The closed version}\label{closed-relative-cobordism}

We begin with the definition of relative cobordism. There are various
equivalent definitions, the best known being the Conner--Floyd formalism
\cite{CF} and the formalism of Genauer \cite{Genauer}, introduced earlier by
Laures \cite{Laures}. For remarks on their equivalence, see
\cite[p.~6]{Laures}.

We summarize the situation as follows.

\begin{definition}
Let $(Y,A)$ be a topological pair. A \textit{relative $n$-bordism over
$(Y,A)$} consists of the following data:
\begin{enumerate}
\item Two smooth $n$-manifolds, possibly with boundary, say
$(M_0,\partial M_0)$ and $(M_1,\partial M_1)$, equipped with maps
$f_i:M_i\to Y$ satisfying $f_i(\partial M_i)\subseteq A$ for $i=0,1$.
\item A smooth $(n+1)$-manifold $W$ with corners, together with a decomposition
of its codimension-one boundary faces
\[
\partial W\cong -M_0\cup M_1\cup Q,
\]
where $Q$ is the relative boundary face which intersects with each of $M_i$ by its boundary, and a map $f_W:W\to Y$ such that
$f_W|_{M_i}=f_i$ for $i=0,1$ and $f_W(Q)\subseteq A$.
\end{enumerate}
This definition generalizes in the evident way to relative spin bordisms and
other tangential structures.
\end{definition}

\begin{definition}
For a compact manifold $X$, possibly with non-empty boundary, write
\[
\mathcal C_X^{\rm cs}
=\operatorname{Map}_*(X/\partial X,BU\times\mathbb Z),
\qquad
\mathcal C_X^{\mathbb R,\rm cs}
=\operatorname{Map}_*(X/\partial X,BO\times\mathbb Z).
\]
\end{definition}

If $X$ is equipped with a spin structure, an embedding of $X$ into
$\mathbb R^N$ gives the natural map
\[
\mathcal C_X^{\rm cs}
\xrightarrow{\operatorname{Ad}_X}
\mathcal C_X^{\mathbb R,\rm cs}
\xrightarrow{\pi_{X!}^{KO}}KO.
\]
It has all the expected properties. In particular,
$\mathcal C_X^{\rm cs}$ carries the natural principal $\mathbb Z_2$-bundle
$O^{\not D}_{\mathcal C_X^{\rm cs}}$, obtained by taking the real determinant
and pulling the resulting torsor back along $\operatorname{Ad}_X$.

Arguing as in Proposition~\ref{closedcobordism}, but now using the Whitney
embedding theorem for manifolds with corners, one obtains the following
statement. For background on embeddings of manifolds with corners into
$\mathbb R^n\times[0,1]^2$, and for the analogue of \cite{GMTW} in this
context, see \cite{Genauer}.

\begin{proposition}\label{relative-cobordism-invariance}
Consider two compact smooth spin $8$-manifolds $X_1$ and $X_2$, possibly
with boundary, and classes $\phi_i\in\mathcal C_{X_i}^{\rm cs}$.

Suppose that there is a relative spin bordism $(W,\phi)$ from
$(X_1,\phi_1)$ to $(X_2,\phi_2)$ over
$(BU\times\mathbb Z,0)$. Then these data give rise to a canonical
identification of $\mathbb Z_2$-torsors
\[
\iota_{[W,\phi]}:
O^{\not D}_{\mathcal C_{X_1}^{\rm cs}}|_{\phi_1}
\xrightarrow{\sim}
O^{\not D}_{\mathcal C_{X_2}^{\rm cs}}|_{\phi_2}.
\]
\end{proposition}

\subsection{Bordism categories and orientation functors}
\label{sec:bordism-categories-orientation-functors}

\subsubsection{Picard groupoids}\label{picardgen}

The main formal framework needed to work with so-called \textit{bordism
categories} is the language of Picard groupoids. We refer to
\cite[Appendix A]{JU} for a full exposition.

\begin{definition}
A \emph{Picard groupoid} is a symmetric monoidal category
\[
(\mathcal G,\otimes,1,\alpha,\sigma)
\]
in which all morphisms are invertible and every object is invertible up to
isomorphism.
\end{definition}

Thus the set
\[
\pi_0(\mathcal G)
=\{\text{isomorphism classes of objects of }\mathcal G\}
\]
is an abelian group, with addition induced by $\otimes$, and
\[
\pi_1(\mathcal G)=\operatorname{Aut}_{\mathcal G}(1)
\]
is an abelian group by the Eckmann--Hilton argument.  For any
object $x\in\mathcal G$, tensoring with $x$ identifies
$\pi_1(\mathcal G)$ with $\operatorname{Aut}_{\mathcal G}(x)$; in this way,
automorphisms of arbitrary objects may be regarded as elements of
$\pi_1(\mathcal G)$. A morphism of Picard groupoids means a symmetric
monoidal functor, and monoidal natural isomorphisms are the natural
$2$-morphisms between such functors.

One may represent the symmetry constraint
of the tensor product in $\mathcal G$ by a biadditive map
\[
\sigma:\pi_0(\mathcal G)\times\pi_0(\mathcal G)
\longrightarrow\pi_1(\mathcal G),
\]
with $\sigma(x,y)=-\sigma(y,x)$. Its diagonal
\[
q_{\mathcal G}(x)=\sigma(x,x),\qquad x\in\pi_0(\mathcal G),
\]
is a \emph{linear quadratic form}; equivalently, it is an additive map
\[
q_{\mathcal G}:\pi_0(\mathcal G)/2\pi_0(\mathcal G)
\longrightarrow\pi_1(\mathcal G).
\]
The triple
\[
\bigl(\pi_0(\mathcal G),\pi_1(\mathcal G),q_{\mathcal G}\bigr)
\]
classifies Picard groupoids up to symmetric monoidal equivalence. Moreover,
if $\mathcal G$ and $\mathcal H$ have symmetry invariants
$q_{\mathcal G}$ and $q_{\mathcal H}$, then group homomorphisms
\[
f_0:\pi_0(\mathcal G)\longrightarrow\pi_0(\mathcal H),
\qquad
f_1:\pi_1(\mathcal G)\longrightarrow\pi_1(\mathcal H)
\]
are induced by a morphism $F:\mathcal G\to\mathcal H$ precisely when
\[
q_{\mathcal H}\circ f_0=f_1\circ q_{\mathcal G}.
\]

For abelian groups $A$ and $B$, and a chosen linear quadratic form
$q:A/2A\to B$, we write $A//B$ for the corresponding Picard groupoid of
$A$-graded $B$-torsors. Concretely, an object is a pair $(a,S)$, where
$a\in A$ and $S$ is a $B$-torsor; morphisms are $B$-equivariant bijections
between torsors of the same degree; and
\[
(a,S)\otimes(a',S')=(a+a',S\otimes_BS').
\]
The symmetry is determined, up to equivalence, by $q$. The cases used most
often below are
\[
0//\mathbb Z_2=\mathbb Z_2\text{-tor},
\qquad
\mathbb Z_2//\mathbb Z_2=s\text{-}\mathbb Z_2\text{-tor},
\]
the Picard groupoids of $\mathbb Z_2$-torsors and super
$\mathbb Z_2$-torsors, respectively. 

\subsubsection{Bordism categories}\label{bordcatextr}

We now have all the prerequisites needed to adapt the argument of Joyce and
Upmeier to our compactly supported setting.  Our notation follows
\cite[Definition 4.1]{JU}.

Put $\mathbb K=BU\times\mathbb Z$, pointed by the zero class
$0\in\mathbb K$.

\begin{definition}\label{relative-bordism-category}
The \emph{relative spin bordism category over $(\mathbb K,0)$}, denoted by
$
\operatorname{Bord}^{\operatorname{Spin}}_8(\mathbb K,0),
$
is the Picard groupoid defined as follows. 

\begin{enumerate}
\item Its objects are pairs $[X,f]$, where $X$ is a compact collared smooth spin
$8$-manifold with boundary (possibly empty), and $f$ is a continuous map of
pairs
\[
f:(X,\partial X)\longrightarrow(\mathbb K,0).
\]

\item The morphism set
\[
\operatorname{Mor}\bigl([X,f],[Y,g]\bigr)
\]
consists of equivalence classes $[W,k]$ of relative spin bordisms from
$[X,f]$ to $[Y,g]$ over $(\mathbb K,0)$, in the sense of
Subsection~\ref{closed-relative-cobordism}, with $W$ compact and with
corners of codimension at most two. The incoming and outgoing faces are
disjoint and meet the relative face in their boundaries. Their
parametrisations, compatible collars, and product spin identifications
are part of the data. Components without incoming or outgoing faces are
allowed.

\item   Let $(W_0,k_0)$ and $(W_1,k_1)$ be two
representatives of relative spin bordisms from $[X,f]$ to $[Y,g]$.  Thus, for
$i=0,1$, the boundary of $W_i$ is decomposed as
\[
\partial W_i\cong -X\cup Y\cup Q_i,
\qquad k_i(Q_i)\subseteq\{0\},
\]
where $Q_i$ is the relative boundary face, and the restrictions of $k_i$ to
$X$ and $Y$ equal $f$ and $g$, respectively.

We declare
\[
(W_0,k_0)\sim(W_1,k_1)
\]
if there exist a compact spin $10$-manifold $T$ with corners of
codimension at most three, a
continuous map $K:T\to\mathbb K$, and a union $Q_T$ of boundary faces of $T$
such that
\[
\partial T\cong
(-W_0)\cup W_1\cup(X\times I)\cup(-Y\times I)\cup Q_T,
\]
where $I=[0,1]$. The $W_i$-faces are disjoint, as are the two side
faces, and $W_i$ meets them in $X\times\{i\}$ and $Y\times\{i\}$.

The remaining intersections are
\[
Q_T\cap W_i=Q_i,
\qquad
Q_T\cap(X\times I)=\partial X\times I,
\qquad
Q_T\cap(Y\times I)=\partial Y\times I.
\]
We require
\[
K(Q_T)\subseteq\{0\},
\qquad
K|_{W_i}=k_i\quad(i=0,1),
\]
and require
\[
K|_{X\times I}=f\circ\operatorname{pr}_X,
\qquad
K|_{Y\times I}=g\circ\operatorname{pr}_Y.
\]

The notation $[W,k]$ means the equivalence class of $(W,k)$ under this
relation.

\item If
\[
[W,k]:[X,f]\longrightarrow[Y,g],
\qquad
[V,l]:[Y,g]\longrightarrow[Z,h],
\]
their composite is represented by
\[
W\cup_{Y\times\{0\}}(Y\times I)
 \cup_{Y\times\{1\}}V,
\]
with map equal to $k$, $g\circ\operatorname{pr}_Y$, and $l$ on the
three pieces, respectively. The identity of $[X,f]$ is
represented by
\[
[X\times I,f\circ\operatorname{pr}_X].
\]

\item The monoidal structure is given by disjoint union.  On objects,
\[
[X,f]\otimes[X',f']=[X\sqcup X',f\sqcup f'].
\]
On morphisms,
\[
[W,k]\otimes[W',k']=[W\sqcup W',k\sqcup k'].
\]
The relative boundary face is $Q_W\sqcup Q_{W'}$, whose image lies in
$\{0\}$.

\item The monoidal identity is
\[
\mathbf 1=[\varnothing, \varnothing].
\]

\item The dual of $[X,f]$ is
\[
[X,f]^*=[-X,f],
\]
where $-X$ denotes $X$ with the opposite spin structure.  The evaluation
morphism
\[
[-X,f]\otimes[X,f]\longrightarrow\mathbf 1
\]
is represented by the cylinder
\[
[X\times I,f\circ\operatorname{pr}_X],
\]
whose relative boundary face $\partial X\times I$ is mapped to $0$.

\item The symmetry isomorphism
\[
\sigma_{[X,f],[Y,g]}:
[X,f]\otimes[Y,g]\longrightarrow[Y,g]\otimes[X,f]
\]
is represented by
\[
\bigl[(X\sqcup Y)\times I,
      (f\sqcup g)\circ\operatorname{pr}_{X\sqcup Y}\bigr],
\]
where the incoming boundary is identified with $X\sqcup Y$ and the outgoing
boundary with $Y\sqcup X$ by transposing the two components.  The relative
boundary face $(\partial X\sqcup\partial Y)\times I$ is mapped to $0$.
\end{enumerate}
\end{definition}

\begin{remark}
\begin{enumerate}
\item The same definition gives a Picard groupoid
\(
\operatorname{Bord}^{\operatorname{Spin}}_8(V,A)
\)
for any topological pair $(V,A)$.
\item Although we restrict here to $8$-manifolds, the same constructions are
meaningful for $7$-manifolds, as in the original work of Joyce and Upmeier
\cite{JU}.
\end{enumerate}
\end{remark}

To illustrate the usefulness of relative bordism, we record the following
excision observation.

\medskip
\noindent
\textbf{Excision trick.}
Let $[X,f]$ be an object of
$\operatorname{Bord}^{\operatorname{Spin}}_8(\mathbb K,0)$, and let
$Y\subseteq X$ be a compact codimension-zero submanifold such that
\[
f|_{X\setminus\operatorname{int}(Y)}=0.
\]
Then Definition~\ref{relative-bordism-category} gives an isomorphism
\[
[X,f]\xrightarrow{\sim}[Y,f|_Y]
\]
represented by the cylinder $X\times I$, regarded as a relative bordism whose
additional boundary face is mapped to $0$.

\subsubsection{Orientation functors}\label{orfun}

Proposition~\ref{relative-cobordism-invariance} and compatibility with
disjoint unions give a symmetric monoidal functor
\[
\mathrm O_8^{\operatorname{Spin},{\rm c}}:
\operatorname{Bord}^{\operatorname{Spin}}_8(\mathbb K,0)
\longrightarrow
\mathbb Z_2//\mathbb Z_2
=s\text{-}\mathbb Z_2\text{-tor}.
\]
This is the compactly supported analogue of the orientation functors of
Joyce--Upmeier; compare \cite[Section 4]{JU}.

\begin{definition}\label{global-orientation-functor}
For an object $[X,f]$, define
\[
\mathrm O_8^{\operatorname{Spin},{\rm c}}[X,f]
=
\left(
\operatorname{ind}_X(f)\bmod2,
O^{\not D}_{\mathcal C_X^{\rm cs}}|_f
\right),
\]
where
\[
\operatorname{ind}_X(f)
=\pi_{X!}^{KO}\bigl(\operatorname{Ad}_X(f)\bigr)
\in KO^0(\operatorname{pt})\cong\mathbb Z
\]
uses the notation of Subsection~\ref{sec:push}.  For a morphism
$[W,k]:[X,f]\to[Y,g]$, put
\[
\mathrm O_8^{\operatorname{Spin},{\rm c}}[W,k]=\iota_{[W,k]}:
O^{\not D}_{\mathcal C_X^{\rm cs}}|_f
\xrightarrow{\sim}
O^{\not D}_{\mathcal C_Y^{\rm cs}}|_g.
\]
Proposition~\ref{relative-cobordism-invariance} shows that this is
well-defined and functorial.
\end{definition}

We shall also need the fixed-manifold bordism categories of
\cite[Definition 4.7]{JU}.

\begin{definition}\label{fixed-bordism-category}
Let $X$ be a compact collared spin $8$-manifold with boundary.  Define
\(
\operatorname{Bord}_X(\mathbb K,0)
\)
as follows.

The objects are maps of pairs
\[
f:(X,\partial X)\longrightarrow(\mathbb K,0).
\]
For two objects $f_0,f_1$, a morphism
\[
[H]:f_0\longrightarrow f_1
\]
is an equivalence class of maps
\[
H:X\times I\longrightarrow\mathbb K
\]
satisfying
\[
H|_{X\times\{0\}}=f_0,
\qquad
H|_{X\times\{1\}}=f_1,
\qquad
H(\partial X\times I)=\{0\}.
\]
Two such homotopies $H,H':f_0\to f_1$ are equivalent if there is a map
\[
K:X\times I\times I\longrightarrow\mathbb K
\]
such that
\[
K(x,t,0)=H(x,t),
\qquad
K(x,t,1)=H'(x,t),
\]
\[
K(x,0,s)=f_0(x),
\qquad
K(x,1,s)=f_1(x),
\]
and
\[
K(a,t,s)=0
\quad\text{for all }a\in\partial X, s,t\in I.
\]

Composition is given by concatenation.  The identity of $f$ is represented by
$f\circ\operatorname{pr}_X$, and the inverse of $[H]$ by
$H^{-1}(x,t)=H(x,1-t)$.  Thus $\operatorname{Bord}_X(\mathbb K,0)$ is a
groupoid.
\end{definition}

\begin{remark}\label{Picgroup}
The same definition gives categories
$\operatorname{Bord}_X(V,A)$ for a topological pair $(V,A)$.  Choices that
will be important below include $(BU(m),0)$ and $(K(\mathbb Z,4),0)$.
\end{remark}

The construction has a compactly supported version for open manifolds.

\begin{definition}\label{openBordX}
Let $X$ be a smooth, possibly non-compact, spin $8$-manifold.  Define
\(
\operatorname{Bord}_X^{\rm cs}(\mathbb K)
\)
as follows.  Its objects are compactly supported maps $f:X\to\mathbb K$.
For two objects $f_0,f_1$, a morphism $[H]:f_0\to f_1$ is an equivalence
class of compactly supported maps
\[
H:X\times I\longrightarrow\mathbb K
\]
such that $H|_{X\times\{i\}}=f_i$ for $i=0,1$.  Here compact support means
that there is a compact set $K\subset X$ such that
$H|(X\setminus K)\times I=0$.

Two such homotopies $H,H'$ are equivalent if there exists a compactly
supported map
\[
\Lambda:X\times I\times I\longrightarrow\mathbb K
\]
satisfying
\[
\Lambda(x,t,0)=H(x,t),
\qquad
\Lambda(x,t,1)=H'(x,t),
\]
\[
\Lambda(x,0,s)=f_0(x),
\qquad
\Lambda(x,1,s)=f_1(x).
\]
Composition, identities, and inverses are defined as above.  Thus
$\operatorname{Bord}_X^{\rm cs}(\mathbb K)$ is a groupoid.
\end{definition}

For compact $X$, there is a tautological functor
\[
\Pi_X^{\operatorname{Spin}}:
\operatorname{Bord}_X(\mathbb K,0)
\longrightarrow
\operatorname{Bord}^{\operatorname{Spin}}_8(\mathbb K,0),
\]
defined by
\[
\Pi_X^{\operatorname{Spin}}(f)=[X,f],
\qquad
\Pi_X^{\operatorname{Spin}}[H]=[X\times I,H].
\]
There is an analogous functor for open manifolds.

\begin{proposition}\label{open-Pi-functor}
Let $X$ be as in Definition~\ref{openBordX}.  There is a natural functor
\[
\Pi_X^{\operatorname{Spin},{\rm cs}}:
\operatorname{Bord}_X^{\rm cs}(\mathbb K)
\longrightarrow
\operatorname{Bord}^{\operatorname{Spin}}_8(\mathbb K,0).
\]
\end{proposition}

\begin{proof}
We first define $\Pi_X^{\operatorname{Spin},{\rm cs}}$ on objects.  Let
$f:X\to\mathbb K$ be
compactly supported.  Choose a compact codimension-zero submanifold
$X'\subset X$ with boundary such that
\[
f|_{X\setminus\operatorname{int}(X')}=0,
\]
and set
\[
\Pi_X^{\operatorname{Spin},{\rm cs}}(f)=[X',f|_{X'}].
\]
On morphisms, use a compact codimension-zero submanifold containing the
support of the homotopy.  Independence of these choices follows from the
excision trick in Subsection~\ref{bordcatextr}; compatibility with
composition is immediate.
\end{proof}

Following \cite[Section 4]{JU}, let
\[
\Pi_{\mathbb Z_2\text{-tor}}:
s\text{-}\mathbb Z_2\text{-tor}
\longrightarrow\mathbb Z_2\text{-tor}
\]
be the functor which forgets the $\mathbb Z_2$-grading.  For $X$ compact with
boundary, set
\[
F_X
=\Pi_{\mathbb Z_2\text{-tor}}
 \circ\mathrm O_8^{\operatorname{Spin},{\rm c}}
 \circ\Pi_X^{\operatorname{Spin}}.
\]
For $X$ open, use instead
\[
F_X^{\rm cs}
=\Pi_{\mathbb Z_2\text{-tor}}
 \circ\mathrm O_8^{\operatorname{Spin},{\rm c}}
 \circ\Pi_X^{\operatorname{Spin},{\rm cs}}.
\]
These factorizations may be displayed as
\[
\begin{tikzcd}[column sep=4.2em]
\operatorname{Bord}_X^{\dagger}(\mathbb K)
 \arrow[r,"\Pi_X^{\dagger}"]
 \arrow[rrr,bend right=18,"F_X^{\dagger}"']&
\operatorname{Bord}^{\operatorname{Spin}}_8(\mathbb K,0)
 \arrow[r,"\mathrm O_8^{\operatorname{Spin},{\rm c}}"]&
s\text{-}\mathbb Z_2\text{-tor}
 \arrow[r,"\Pi_{\mathbb Z_2\text{-tor}}"]&
\mathbb Z_2\text{-tor},
\end{tikzcd}
\]
where we use the convention
\[
\operatorname{Bord}_X^\dagger(\mathbb K)
=
\begin{cases}
\operatorname{Bord}_X(\mathbb K,0),&X\text{ compact with boundary},\\
\operatorname{Bord}_X^{\rm cs}(\mathbb K),&X\text{ open},
\end{cases}
\]
and
\[
\Pi_X^\dagger=
\begin{cases}
\Pi_X^{\operatorname{Spin}},&X\text{ compact with boundary},\\
\Pi_X^{\operatorname{Spin},{\rm cs}},&X\text{ open},
\end{cases}
\]
and similarly for $F_X^\dagger$.

\begin{definition}\label{orientation-fixed-X}
Let $X$ be a spin $8$-manifold, either compact with boundary or open.  Write
$F_X^\dagger$ for the corresponding functor above, and let
\[
\mathbf 1_X:
\operatorname{Bord}_X^\dagger(\mathbb K)
\longrightarrow\mathbb Z_2\text{-tor}
\]
be the constant functor with value the trivial $\mathbb Z_2$-torsor
$\mathbb Z_2$.  An \emph{orientation for $X$} is a natural
isomorphism
\[
\omega_X:F_X^\dagger\Longrightarrow\mathbf 1_X.
\]
\end{definition}

\begin{remark}
The natural isomorphism $\omega_X$ should be regarded as a
bordism-coherent trivialization of the topological orientation bundle.  In
particular, evaluating $\omega_X$ on objects gives a trivialization of
$O^{\not D}_{\mathcal C_X^{\rm cs}}$.  This is the compactly supported
analogue of the discussion in \cite[Section 4]{JU}.
\end{remark}


\subsection{Functors between Picard groupoids}\label{secE8Z4}

In Subsection~\ref{picardgen}, we discussed the general theorem that
symmetric monoidal functors between Picard groupoids are classified by
algebraic data involving $\pi_0$ and $\pi_1$.  In this subsection, we give
examples of this construction for the relative bordism categories introduced
in Subsection~\ref{bordcatextr}.  We first calculate $\pi_0$ and $\pi_1$ in
the examples of interest to us.

\subsubsection{$\pi_0$ and $\pi_1$ in examples}

\begin{theorem}\label{homgrgr}
The following isomorphisms hold:
\begin{enumerate}
\item For $m\geq5$,
\[
\pi_0\!\left(
  \operatorname{Bord}^{\operatorname{Spin}}_8(BSU(m),*)
\right)
\cong
\widetilde\Omega^{\operatorname{Spin}}_8(BSU(m))
\cong\mathbb Z^3,
\]
and
\[
\pi_1\!\left(
  \operatorname{Bord}^{\operatorname{Spin}}_8(BSU(m),*)
\right)
\cong
\widetilde\Omega^{\operatorname{Spin}}_9(BSU(m))
\cong\mathbb Z_2.
\]

\item
\[
\pi_0\!\left(
  \operatorname{Bord}^{\operatorname{Spin}}_8(K(\mathbb Z,4),0)
\right)
\cong\mathbb Z^2,
\qquad
\pi_1\!\left(
  \operatorname{Bord}^{\operatorname{Spin}}_8(K(\mathbb Z,4),0)
\right)
\cong\mathbb Z_2.
\]

\item
\[
\pi_0\!\left(
  \operatorname{Bord}^{\operatorname{Spin}}_8(BE_8,*)
\right)
\cong\mathbb Z^2,
\qquad
\pi_1\!\left(
  \operatorname{Bord}^{\operatorname{Spin}}_8(BE_8,*)
\right)
\cong\mathbb Z_2.
\]
\end{enumerate}
\end{theorem}

\begin{proof}
Joyce and Upmeier prove that
\[
\pi_i\!\left(\operatorname{Bord}^{\operatorname{Spin}}_8(T)\right)
\cong
\Omega^{\operatorname{Spin}}_{8+i}(T),
\qquad i=0,1;
\]
see \cite[Proposition~7.2, equations~(7.2)--(7.3)]{JU}.  The same argument for
the relative category of Definition~\ref{relative-bordism-category} gives
\[
\pi_i\!\left(
  \operatorname{Bord}^{\operatorname{Spin}}_8(T,*)
\right)
\cong
\Omega^{\operatorname{Spin}}_{8+i}(T,*)
\cong
\widetilde\Omega^{\operatorname{Spin}}_{8+i}(T),
\qquad i=0,1,
\]
for any pointed space $(T,*)$.

The stated values for $BSU(m)$ and $BE_8$ are given in
\cite[Corollary~3.4(a), Table~3.4]{JU}; the restriction $m\geq5$ ensures that
the stable calculation applies in both degrees $8$ and $9$.  The values for
$K(\mathbb Z,4)$ are given in \cite[Theorem~3.1(a), Table~3.1]{JU}.
\end{proof}

\begin{remark}\label{generators}
The argument of Joyce and Upmeier not only calculates the relevant bordism
groups, but also gives distinguished generators.  To distinguish the three
cases, we add superscripts, which are suppressed in \cite{JU} when the
ambient space is clear:
\begin{enumerate}
\item $\frac{\zeta_1^{SU}}2,\zeta_2^{SU},\zeta_3^{SU}$ generate $\pi_0$ in
the $SU(m)$ case, and $\alpha_1\zeta_2^{SU}$ generates $\pi_1$;
\item $\zeta_2^{E_8},\zeta_3^{E_8}$ generate $\pi_0$ in the $E_8$ case, and
$\alpha_1\zeta_2^{E_8}$ generates $\pi_1$;
\item $\zeta_2^K,\zeta_3^K$ generate $\pi_0$ in the $K(\mathbb Z,4)$ case,
and $\alpha_1\zeta_2^K$ generates $\pi_1$.
\end{enumerate}
Their geometric definitions are too intricate to reproduce here; see
\cite[Theorem~3.1, Corollary~3.4 and Tables~3.1, 3.4]{JU}.
\end{remark}

\subsubsection{The functors $\mathrm H_8^{\mathbb Z,{\rm c}}$,
$\mathrm N_{8,\mathbb Z_2}^{\operatorname{Spin},E_8,{\rm c}}$, and
$F_{BE_8}^{K(\mathbb Z,4),{\rm c}}$}\label{threefun}

Before proving our main theorems, we introduce three functors following
Joyce and Upmeier.  The first two are the compactly supported analogues of
the cohomology and normalized gauge-theoretic orientation functors in
\cite[Definitions~9.10 and~9.14]{JU}:
\[
\mathrm H_8^{\mathbb Z,{\rm c}}:
\operatorname{Bord}^{\operatorname{Spin}}_8(K(\mathbb Z,4),0)
\longrightarrow\mathbb Z_2//\mathbb Z_2,
\]
and
\[
\mathrm N_{8,\mathbb Z_2}^{\operatorname{Spin},E_8,{\rm c}}:
\operatorname{Bord}^{\operatorname{Spin}}_8(BE_8,*)
\longrightarrow\mathbb Z_2//\mathbb Z_2.
\]

They are determined, up to monoidal natural isomorphism, by the generators in
Remark~\ref{generators}.  More precisely,
\[
\begin{aligned}
\pi_0\!\left(\mathrm H_8^{\mathbb Z,{\rm c}}\right):
&\ \widetilde\Omega^{\operatorname{Spin}}_8(K(\mathbb Z,4))
   \longrightarrow\mathbb Z_2,
&\zeta_2^K&\longmapsto1,
&\zeta_3^K&\longmapsto0,\\
\pi_1\!\left(\mathrm H_8^{\mathbb Z,{\rm c}}\right):
&\ \widetilde\Omega^{\operatorname{Spin}}_9(K(\mathbb Z,4))
   \longrightarrow\mathbb Z_2,
&\alpha_1\zeta_2^K&\longmapsto1,
\end{aligned}
\]
whereas
\[
\begin{aligned}
\pi_0\!\left(
 \mathrm N_{8,\mathbb Z_2}^{\operatorname{Spin},E_8,{\rm c}}
\right):
&\ \widetilde\Omega^{\operatorname{Spin}}_8(BE_8)
   \longrightarrow\mathbb Z_2,
&\zeta_2^{E_8}&\longmapsto1,
&\zeta_3^{E_8}&\longmapsto0,\\
\pi_1\!\left(
 \mathrm N_{8,\mathbb Z_2}^{\operatorname{Spin},E_8,{\rm c}}
\right):
&\ \widetilde\Omega^{\operatorname{Spin}}_9(BE_8)
   \longrightarrow\mathbb Z_2,
&\alpha_1\zeta_2^{E_8}&\longmapsto1.
\end{aligned}
\]
This is precisely the Picard-groupoid construction of
Subsection~\ref{picardgen}; compare \cite[equation~(9.6) and
Theorem~11.1]{JU}.

The adjoint representation gives a map
$\operatorname{Ad}_{E_8}:BE_8\to BO$, so a map $f:X\to BE_8$ gives the
composite $\operatorname{Ad}_{E_8}\circ f:X\to BO$.  Define
\[
\widetilde{\mathrm N}_{8,\mathbb Z_2}^{\operatorname{Spin},E_8,{\rm c}}
[X,f]
=
\left(
 \pi_{X!}^{KO}\!\left(\operatorname{Ad}_{E_8}\circ f\right)\bmod2,
 O^{KO}_{\mathcal C_X^{\mathbb R,{\rm cs}}}
 \big|_{\operatorname{Ad}_{E_8}\circ f}
\right).
\]
On morphisms, use the canonical isomorphisms supplied by the real analogue of
Proposition~\ref{relative-cobordism-invariance}, proved by the same argument;
compare Definition~\ref{global-orientation-functor}.

\begin{proposition}\label{E8-explicit-orientation-functor}
There is a monoidal natural isomorphism
\[
\widetilde{\mathrm N}_{8,\mathbb Z_2}^{\operatorname{Spin},E_8,{\rm c}}
\cong
\mathrm N_{8,\mathbb Z_2}^{\operatorname{Spin},E_8,{\rm c}}.
\]
\end{proposition}

\begin{proof}
The proof is the compactly supported version of the calculation in
\cite[Theorem~11.1, Tables~11.1--11.2]{JU}.  That calculation is topological
and uses the Atiyah--Singer Index Theorem. It applies
\textit{mutatis mutandis} to our relative bordism categories and gives the maps on $\pi_0$
and $\pi_1$ displayed above.  The classification theorem of
Subsection~\ref{picardgen} then gives the required monoidal natural
isomorphism.
\end{proof}

Finally, the canonical characteristic class $a_1\in H^4(BE_8;\mathbb Z)$
is classified by a $16$-connected map
\[
a_1:BE_8\longrightarrow K(\mathbb Z,4);
\]
see \cite[Theorem~3.3(d), equations~(3.29)--(3.30)]{JU}.  It induces the
transfer functor
\[
F_{BE_8}^{K(\mathbb Z,4),{\rm c}}:
\operatorname{Bord}^{\operatorname{Spin}}_8(BE_8,*)
\longrightarrow
\operatorname{Bord}^{\operatorname{Spin}}_8(K(\mathbb Z,4),0).
\]
There is an analogous functor
$F_{BE_8}^{K(\mathbb Z,4)}$ between the ordinary Joyce--Upmeier bordism
categories.

Under the generators in Remark~\ref{generators}, the induced maps are
\[
\pi_0\!\left(F_{BE_8}^{K(\mathbb Z,4),{\rm c}}\right)
(\zeta_i^{E_8})=\zeta_i^K\quad(i=2,3),
\qquad
\pi_1\!\left(F_{BE_8}^{K(\mathbb Z,4),{\rm c}}\right)
(\alpha_1\zeta_2^{E_8})=\alpha_1\zeta_2^K.
\]
This follows from \cite[Theorem~3.3(d), equation~(3.19), and
Corollary~3.4(b)]{JU}.  Consequently, there is a monoidal natural
isomorphism
\[
\lambda_8^{\rm c}:
\mathrm N_{8,\mathbb Z_2}^{\operatorname{Spin},E_8,{\rm c}}
\Longrightarrow
\mathrm H_8^{\mathbb Z,{\rm c}}
\circ F_{BE_8}^{K(\mathbb Z,4),{\rm c}},
\]
the compactly supported analogue of \cite[Theorem~11.10(a),
equation~(11.14)]{JU}.  Equivalently, the following diagram is
$2$-commutative:
\begin{equation}\label{E8Z4}
\begin{tikzcd}[column sep=7.2em,row sep=3.2em]
\operatorname{Bord}^{\operatorname{Spin}}_8(BE_8,*)
  \arrow[r,"{F_{BE_8}^{K(\mathbb Z,4),{\rm c}}}"]
  \arrow[d,"{\mathrm N_{8,\mathbb Z_2}^{\operatorname{Spin},E_8,{\rm c}}}"']
  \arrow[dr,phantom,"\Downarrow\,\lambda_8^{\rm c}" description]
&
\operatorname{Bord}^{\operatorname{Spin}}_8(K(\mathbb Z,4),0)
  \arrow[d,"{\mathrm H_8^{\mathbb Z,{\rm c}}}"]
\\
\mathbb Z_2//\mathbb Z_2
  \arrow[r,"\operatorname{id}"']
& \mathbb Z_2//\mathbb Z_2.
\end{tikzcd}
\end{equation}

\begin{remark}\label{orientation-conventions}
The monoidal natural isomorphism \(\lambda_8^{\rm c}\) in \eqref{E8Z4}
is not canonical.  In the compact theory, Joyce and Upmeier show that the
choices compatible with the functors to and from
\(\operatorname{Bord}^{\operatorname{Spin}}_8(*)\) form a torsor for
\(\mathbb Z_2^2\); a choice is what they call an orientation convention
\cite[Theorem~11.10(a) and Remark~12.16(c)]{JU}.  We impose the same
compatibility on \(\lambda_8^{\rm c}\) and fix one such isomorphism.  The
functor \(\mathrm N_{8,\mathbb Z_2}^{\operatorname{Spin},E_8,{\rm c}}\) is
identified with \(\widetilde{\mathrm N}_{8,\mathbb Z_2}^{
\operatorname{Spin},E_8,{\rm c}}\) using
Proposition~\ref{E8-explicit-orientation-functor}; changing this
identification is included in the choice of \(\lambda_8^{\rm c}\).  The
isomorphisms \(\lambda_{8,m}^{\operatorname{SU},{\rm c}}\) and
\(\lambda_{8,m}^{\operatorname{U},{\rm c}}\) below always mean those obtained
from this choice in Lemma~\ref{rescompc} and
Proposition~\ref{BU-BSU-transition}.

We also use this opportunity to fix some complex structures used in Lemma~\ref{comptype}.  Let
\(\pi_v:\operatorname{Spin}(16)\to\operatorname{SO}(16)\) and
\(\pi_s:\operatorname{Spin}(16)\to
\operatorname{HSpin}(16)\) be the quotient maps, and put
\(\widetilde{\operatorname{U}}(8)=
\pi_v^{-1}(\operatorname{U}(8))\) and
\(U=\pi_s(\widetilde{\operatorname{U}}(8))\).  The homomorphism
\((A,\zeta)\mapsto\zeta A\) from
\(\operatorname{SU}(8)\times\operatorname{U}(1)\) to
\(\operatorname{U}(8)\) lifts to \(\widetilde{\operatorname{U}}(8)\).
Indeed, under the determinant isomorphism
\(\pi_1(\operatorname{U}(8))\cong\mathbb Z\), its image is
\(8\mathbb Z\), which maps to zero in
\(\pi_1(\operatorname{SO}(16))\cong\mathbb Z_2\).
Its composite with \(\pi_s\) is a finite covering onto \(U\).  We take
\(\iota_8:\operatorname{SU}(8)\times\operatorname{U}(1)\to E_8\) to be
this covering followed by the inclusions
\(U\subset\operatorname{HSpin}(16)\subset E_8\); compare
\cite[Chapters~6--7]{AdamsExceptional}. This is a representative
of the morphism of complex type of \cite[Theorem~11.8]{JU} used in
Lemma~\ref{rescompc} to construct
\(\lambda_{8,8}^{\operatorname{SU},{\rm c}}\) from \(\lambda_8^{\rm c}\).

We choose the complex structure on \(\mathfrak e_8/\mathfrak u(8)\) whose
restriction to \(\operatorname{SU}(8)\) is
\(W\oplus\Lambda^3W\oplus(\Lambda^2W)^\vee\), where \(W\) is the standard
representation.  For
\(\jmath_m(A)=\operatorname{diag}(A,\det(A)^{-1})\), we regard the quotient
representation as the underlying real representation of
\(W_m\otimes\det(W_m)\), where \(W_m\) is the standard representation of
\(\operatorname{U}(m)\).  These complex structures are used throughout. 

It is easy to see that the  choice of \(\lambda_8^{\rm c}\) does not affect the 
calculations in the present paper. 
\end{remark}

\subsection{Factorizations of orientation functors}
\label{sec:factorizations-orientation-functors}

For a compact Lie group $G$ with Lie algebra $\mathfrak g$, write
\[
\mathrm N_{8,\mathbb Z_2}^{\operatorname{Spin},G,{\rm c}}:
\operatorname{Bord}^{\operatorname{Spin}}_8(BG,*)
\longrightarrow
\mathbb Z_2//\mathbb Z_2
\]
for the compactly supported normalized orientation functor associated to the
reduced adjoint virtual bundle
\(
\operatorname{Ad}(\mathfrak g)-\underline{\mathfrak g}.
\)
This is the relative analogue of the normalized orientation functor of
Joyce--Upmeier; compare \cite[Definition~9.10]{JU}.

We start with the following elementary definition and lemma.

\begin{definition}\label{complex-type-morphism}
Let $G_1,G_2$ be compact Lie groups, and let
\(
\iota:G_1\longrightarrow G_2
\)
be a morphism of Lie groups. Write $\mathfrak g_1,\mathfrak g_2$ for their
Lie algebras and
\(
d\iota:\mathfrak g_1\longrightarrow\mathfrak g_2
\)
for the induced Lie algebra morphism. We say that $\iota$ is a
\emph{morphism of complex type} if $d\iota$ is injective and the quotient
\(
\mathfrak m_\iota:=\mathfrak g_2/d\iota(\mathfrak g_1)
\)
is of complex type as a real $G_1$-representation. Equivalently, there is an
endomorphism
\(
J:\mathfrak m_\iota\longrightarrow\mathfrak m_\iota,
\ J^2=-1,
\)
which commutes with the $G_1$-action induced by the adjoint action of $G_2$
through $\iota$. Thus $\mathfrak m_\iota$ is the underlying real
representation of a complex $G_1$-representation. This is the terminology of
\cite[Definition~11.5]{JU}.
\end{definition}

\begin{lemma}\label{comptype}
Let $\iota:G_1\to G_2$ be a morphism of complex type, and let
\[
F_\iota^{\rm c}:
\operatorname{Bord}^{\operatorname{Spin}}_8(BG_1,*)
\longrightarrow
\operatorname{Bord}^{\operatorname{Spin}}_8(BG_2,*)
\]
be the transfer functor induced by $B\iota$. There is a canonical monoidal
natural isomorphism
\[
\epsilon_\iota^{\rm c}:
\mathrm N_{8,\mathbb Z_2}^{\operatorname{Spin},G_1,{\rm c}}
\Longrightarrow
\mathrm N_{8,\mathbb Z_2}^{\operatorname{Spin},G_2,{\rm c}}
\circ F_\iota^{\rm c}.
\]
Equivalently, the following diagram is $2$-commutative:
\[
\begin{tikzcd}[column sep=6.7em,row sep=3em]
\operatorname{Bord}^{\operatorname{Spin}}_8(BG_1,*)
  \arrow[r,"{F_\iota^{\rm c}}"]
  \arrow[d,"{\mathrm N_{8,\mathbb Z_2}^{\operatorname{Spin},G_1,{\rm c}}}"']
  \arrow[dr,phantom,"\Downarrow\,\epsilon_\iota^{\rm c}" description]
&
\operatorname{Bord}^{\operatorname{Spin}}_8(BG_2,*)
  \arrow[d,"{\mathrm N_{8,\mathbb Z_2}^{\operatorname{Spin},G_2,{\rm c}}}"]
\\
\mathbb Z_2//\mathbb Z_2
  \arrow[r,"\operatorname{id}"']
&
\mathbb Z_2//\mathbb Z_2.
\end{tikzcd}
\]
\end{lemma}

\begin{proof}
Over $BG_1$ there is a natural decomposition
\[
(B\iota)^*\bigl(
 \operatorname{Ad}(\mathfrak g_2)-\underline{\mathfrak g_2}
\bigr)
\cong
\operatorname{Ad}(\mathfrak g_1)-\underline{\mathfrak g_1}
\oplus
\bigl(
 \mathfrak m_\iota-\underline{\mathfrak m_\iota}
\bigr).
\]
The second summand is the underlying real bundle of a virtual complex bundle.
The rest follows as in \cite[Proposition~11.6(a),(c)]{JU} using the Atiyah--Bott--Shapiro square of Section \ref{KO-theoretic-interpretation}.
\end{proof}

\begin{lemma}\label{product-and-abelian}
The normalized orientation functors have the following properties.

\begin{enumerate}
\item Let $G_1,G_2$ be compact Lie groups, let
$\pi_i:G_1\times G_2\to G_i$ be the projections, and let
$F_{\pi_i}^{\rm c}$ be the induced transfer functors. There is a canonical
monoidal natural isomorphism $\zeta_{G_1,G_2}^{\rm c}$ making the following
diagram $2$-commutative:
\[
\begin{tikzcd}[column sep=8em,row sep=3em]
\operatorname{Bord}^{\operatorname{Spin}}_8(B(G_1\times G_2),*)
  \arrow[r,"{\mathrm N_{8,\mathbb Z_2}^{\operatorname{Spin},G_1\times G_2,{\rm c}}}"]
  \arrow[d,"{(F_{\pi_1}^{\rm c},F_{\pi_2}^{\rm c})}"']
  \arrow[dr,phantom,"\Downarrow\,\zeta_{G_1,G_2}^{\rm c}" description]
&
\mathbb Z_2//\mathbb Z_2
\\
\begin{gathered}
\operatorname{Bord}^{\operatorname{Spin}}_8(BG_1,*)\times\\
\operatorname{Bord}^{\operatorname{Spin}}_8(BG_2,*)
\end{gathered}
  \arrow[r,"{\mathrm N_{8,\mathbb Z_2}^{\operatorname{Spin},G_1,{\rm c}}
  \times
  \mathrm N_{8,\mathbb Z_2}^{\operatorname{Spin},G_2,{\rm c}}}"']
&
(\mathbb Z_2//\mathbb Z_2)\times(\mathbb Z_2//\mathbb Z_2)
  \arrow[u,"\otimes"'].
\end{tikzcd}
\]

\item If $G$ is abelian, there is a canonical monoidal natural isomorphism
\[
\mathrm N_{8,\mathbb Z_2}^{\operatorname{Spin},G,{\rm c}}
\cong\mathbf 1,
\]
where $\mathbf 1$ is the constant functor with value the monoidal identity of
$\mathbb Z_2//\mathbb Z_2$.
\end{enumerate}
\end{lemma}

\begin{proof}
The first statement follows from the canonical direct sum decomposition of
the reduced adjoint virtual bundle and the compatibility of determinant
torsors with direct sums. If $G$ is abelian, its adjoint action on
$\mathfrak g$ is trivial, so the reduced adjoint virtual bundle is canonically
zero. These identifications are compatible with relative bordisms, compare
\cite[Proposition~11.6(b)--(d)]{JU}.
\end{proof}

The functors of Subsection~\ref{threefun} and the preceding two lemmas give
the following factorization.

\begin{lemma}\label{rescompc}
For every $m\geq2$, there is a monoidal natural isomorphism
\[
\lambda_{8,m}^{\operatorname{SU},{\rm c}}:
\mathrm N_{8,\mathbb Z_2}^{\operatorname{Spin},\operatorname{SU}(m),{\rm c}}
\Longrightarrow
\mathrm H_8^{\mathbb Z,{\rm c}}
\circ F_{B\operatorname{SU}(m)}^{K(\mathbb Z,4),{\rm c}},
\]
where
\[
F_{B\operatorname{SU}(m)}^{K(\mathbb Z,4),{\rm c}}:
\operatorname{Bord}^{\operatorname{Spin}}_8(B\operatorname{SU}(m),*)
\longrightarrow
\operatorname{Bord}^{\operatorname{Spin}}_8(K(\mathbb Z,4),0)
\]
is induced by $c_2:B\operatorname{SU}(m)\to K(\mathbb Z,4)$. Equivalently,
the following diagram is $2$-commutative:
\[
\begin{tikzcd}[column sep=7.8em,row sep=3em]
\operatorname{Bord}^{\operatorname{Spin}}_8(B\operatorname{SU}(m),*)
  \arrow[r,"{F_{B\operatorname{SU}(m)}^{K(\mathbb Z,4),{\rm c}}}"]
  \arrow[d,"{\mathrm N_{8,\mathbb Z_2}^{\operatorname{Spin},\operatorname{SU}(m),{\rm c}}}"']
  \arrow[dr,phantom,"\Downarrow\,\lambda_{8,m}^{\operatorname{SU},{\rm c}}" description]
&
\operatorname{Bord}^{\operatorname{Spin}}_8(K(\mathbb Z,4),0)
  \arrow[d,"{\mathrm H_8^{\mathbb Z,{\rm c}}}"]
\\
\mathbb Z_2//\mathbb Z_2
  \arrow[r,"\operatorname{id}"']
&
\mathbb Z_2//\mathbb Z_2.
\end{tikzcd}
\]
\end{lemma}

\begin{proof}
We follow the proof of \cite[Theorem~11.11(a)(i),(b)(i)]{JU}.

\textit{Step 1.}
Consider the $2$-commutative diagram~\eqref{E8Z4}. Joyce and Upmeier construct
a morphism of complex type
\[
\iota_8:\operatorname{SU}(8)\times\operatorname{U}(1)
\longrightarrow E_8;
\]
see \cite[Theorem~11.8]{JU}. Lemmas~\ref{comptype} and
\ref{product-and-abelian}, together with diagram~\eqref{E8Z4}, therefore give
the required factorization for $\operatorname{SU}(8)$. With the normalizations
of Joyce and Upmeier, the composite
\[
B\operatorname{SU}(8)
\longrightarrow BE_8
\xrightarrow{a_1}K(\mathbb Z,4)
\]
represents $c_2$; compare \cite[Remark~8.2(a), Theorem~11.11(b)(i), and its
proof]{JU}.

\textit{Step 2.}
For $2\leq m<8$, use the complex-type morphisms
\[
\operatorname{SU}(m)\times\operatorname{U}(1)
\longrightarrow\operatorname{SU}(m+1)
\]
of \cite[Theorem~11.8]{JU}, together with
Lemmas~\ref{comptype} and~\ref{product-and-abelian}. The pullback of $c_2$
under the standard inclusion
$B\operatorname{SU}(m)\to B\operatorname{SU}(m+1)$ is again $c_2$, so the
factorization descends from $m+1$ to $m$.

For $m>8$, the stabilization maps
$B\operatorname{SU}(m)\to B\operatorname{SU}(m+1)$ induce equivalences on the
spin bordism categories in dimensions $8$ and $9$. Since $c_2$ is stable, the
factorization for $m=8$ therefore extends to every $m>8$. See
\cite[Theorem~3.3(b), Theorem~8.3(b), and the proof of
Theorem~11.11(b)(i)]{JU}.
\end{proof}

We now explain explicitly the passage from $B\operatorname{SU}$ to
$B\operatorname{U}$.

\begin{proposition}\label{BU-BSU-transition}
For $m\geq1$, define
\[
\jmath_m:\operatorname{U}(m)\longrightarrow\operatorname{SU}(m+1),
\qquad
A\longmapsto
\begin{pmatrix}
A&0\\
0&\det(A)^{-1}
\end{pmatrix}.
\]
This is a morphism of complex type. On classifying spaces, $B\jmath_m$ sends a
rank-$m$ complex vector bundle $E$ to
\(
E\oplus\det(E)^\vee.
\)
Moreover, the diagram
\[
\begin{tikzcd}[column sep=6em,row sep=3em]
B\operatorname{U}(m)
  \arrow[r,"{B\jmath_m}"]
  \arrow[dr,"{\kappa_m=c_2-c_1^2}"']
&
B\operatorname{SU}(m+1)
  \arrow[d,"c_2"]
\\
&K(\mathbb Z,4)
\end{tikzcd}
\]
commutes up to homotopy. Consequently, there is a monoidal natural
isomorphism
\[
\lambda_{8,m}^{\operatorname{U},{\rm c}}:
\mathrm N_{8,\mathbb Z_2}^{\operatorname{Spin},\operatorname{U}(m),{\rm c}}
\Longrightarrow
\mathrm H_8^{\mathbb Z,{\rm c}}
\circ F_{B\operatorname{U}(m)}^{K(\mathbb Z,4),{\rm c}},
\]
where the transfer functor on the right is induced by
$\kappa_m=c_2-c_1^2$.
\end{proposition}

\begin{proof}
The fact that $\jmath_m$ is of complex type follows from
\cite[Theorem~11.8]{JU}. Hence Lemma~\ref{comptype} compares the normalized
$\operatorname{U}(m)$-orientation functor with the pullback of the normalized
$\operatorname{SU}(m+1)$-orientation functor. For a rank $m$ bundle $E$, the
Whitney formula gives
\[
c_1\bigl(E\oplus\det(E)^\vee\bigr)=0,
\qquad
c_2\bigl(E\oplus\det(E)^\vee\bigr)
=c_2(E)-c_1(E)^2.
\]
The result now follows by composing the natural isomorphism of
Lemma~\ref{comptype} with $\lambda_{8,m+1}^{\operatorname{SU},{\rm c}}$ from
Lemma~\ref{rescompc}. This is the $\operatorname{U}(m)$ step in the proof of
\cite[Theorem~11.11(b)(i)]{JU}.
\end{proof}

Thus we obtain the following consequence for stable complex $K$-theory.

\begin{corollary}\label{cantriv}
Let $X$ be a complex $4$-fold equipped with a spin structure, which is either
compact, possibly with boundary, or open. Put
\[
X^\dagger=
\begin{cases}
X/\partial X,&X\text{ compact},\\
X^+,&X\text{ open},
\end{cases}
\qquad
\mathcal C_{X,0}^\dagger=
\operatorname{Map}_*(X^\dagger,BU),
\]
where, if $\partial X=\varnothing$, a disjoint basepoint is understood, and
$BU$ is the virtual-rank-zero component of $BU\times\mathbb Z$.

Assume that the ungraded cohomology orientation functor
\[
\Pi_{\mathbb Z_2\text{-tor}}
\circ\mathrm H_8^{\mathbb Z,{\rm c}}
\circ\Pi_X^\dagger:
\operatorname{Bord}_X^\dagger(K(\mathbb Z,4))
\longrightarrow\mathbb Z_2\text{-tor}
\]
admits a trivialization. Here $\operatorname{Bord}_X^\dagger$ and
$\Pi_X^\dagger$ use the compact/open convention of
Subsection~\ref{orfun}. Then the restriction of $F_X^\dagger$ to the
virtual-rank-zero mapping space $\mathcal C_{X,0}^\dagger$ is orientable.
Equivalently, the normalized determinant $\mathbb Z_2$-bundle over
$\operatorname{Map}_*(X^\dagger,BU)$ is trivial.
\end{corollary}

\begin{proof}
Let
\[
j_m:
\operatorname{Map}_*(X^\dagger,B\operatorname{U}(m))
\longrightarrow
\operatorname{Map}_*(X^\dagger,BU)
\]
be induced by stabilization. Proposition~\ref{BU-BSU-transition} identifies
the restriction of the normalized orientation functor along $j_m$ with the
pullback of the cohomology orientation functor along
\[
\kappa_m=c_2-c_1^2:
B\operatorname{U}(m)\longrightarrow K(\mathbb Z,4).
\]
The assumed trivialization therefore shows that the normalized determinant
$\mathbb Z_2$-bundle is trivial after pullback along every $j_m$.

It remains only to pass to the stable mapping space. The space $X^\dagger$ is
compact, and $BU=\operatorname*{colim}_mB\operatorname{U}(m)$ has its usual
CW topology. Hence every singular $1$-simplex in
$\operatorname{Map}_*(X^\dagger,BU)$, equivalently every map
\[
X^\dagger\times\Delta^1\longrightarrow BU,
\]
has image in $B\operatorname{U}(m)$ for some $m$. It follows that
\[
H_1\bigl(\operatorname{Map}_*(X^\dagger,BU);\mathbb Z_2\bigr)
=
\operatorname*{colim}_m
H_1\bigl(
 \operatorname{Map}_*(X^\dagger,B\operatorname{U}(m));\mathbb Z_2
\bigr).
\]
The first Stiefel--Whitney class of the normalized determinant bundle vanishes
on every term on the right, and therefore vanishes on the stable mapping
space. This is precisely orientability in the sense of
Definition~\ref{orientation-fixed-X}.
\end{proof}

Corollary~\ref{cantriv} pairs well with the following relative version of the
Joyce--Upmeier obstruction. For a CW-pair $(Y,A)$, we use the relative
Steenrod square
\[
\operatorname{Sq}_{(Y,A)}^i:
H^k(Y,A;\mathbb Z_2)
\longrightarrow
H^{k+i}(Y,A;\mathbb Z_2),
\]
defined, under the natural identification
$H^*(Y,A;\mathbb Z_2)\cong\widetilde H^*(Y/A;\mathbb Z_2)$, by the ordinary
Steenrod square on reduced cohomology; see \cite{SteenrodEpstein1962}. We
write $\operatorname{Sq}_\partial^i$ when $(Y,A)=(X,\partial X)$. For an open
manifold $X$, the compactly supported operation $\operatorname{Sq}_c^i$ is
defined in the same way using
$H_c^*(X;\mathbb Z_2)\cong\widetilde H^*(X^+;\mathbb Z_2)$.

\begin{proposition}\label{Steenrodprop}
Let $X$ be a compact spin $8$-manifold with boundary. Suppose that there is no
class $\alpha\in H^3(X,\partial X;\mathbb Z)$ such that
\[
\left\langle
 \bar\alpha\smile\operatorname{Sq}_\partial^2(\bar\alpha),
 [X,\partial X]
\right\rangle
=1
\quad\text{in }\mathbb Z_2,
\]
where $\bar\alpha$ is the mod-$2$ reduction of $\alpha$. Then one may choose a
natural isomorphism $\eta_X^\partial$ making the following diagram
$2$-commutative:
\[
\begin{tikzcd}[column sep=7em,row sep=3em]
\operatorname{Bord}_X(K(\mathbb Z,4),0)
  \arrow[r,"{\mathbf 1_X}"]
  \arrow[d,"{\Pi_X^{\operatorname{Spin}}}"']
  \arrow[dr,phantom,"\Downarrow\,\eta_X^\partial" description]
&
0//\mathbb Z_2
\\
\operatorname{Bord}^{\operatorname{Spin}}_8(K(\mathbb Z,4),0)
  \arrow[r,"{\mathrm H_8^{\mathbb Z,{\rm c}}}"']
&
\mathbb Z_2//\mathbb Z_2
  \arrow[u,"{\Pi_{\mathbb Z_2\text{-tor}}}"'].
\end{tikzcd}
\]
\end{proposition}

\begin{proof}
The fundamental groupoid of
$\operatorname{Map}_*(X/\partial X,K(\mathbb Z,4))$ has homotopy groups
\[
\pi_i\operatorname{Map}_*(X/\partial X,K(\mathbb Z,4))
\cong H^{4-i}(X,\partial X;\mathbb Z),
\qquad i=0,1.
\]
With this replacement of absolute by relative cohomology, the proof of
\cite[Theorem~10.7(a) and Subsection~19.2]{JU} applies unchanged. Its
obstruction is exactly
$\langle\bar\alpha\smile\operatorname{Sq}_\partial^2(\bar\alpha),
[X,\partial X]\rangle$. The relative bordism identifications needed in that
proof are supplied by Definition~\ref{fixed-bordism-category} and
Proposition~\ref{relative-cobordism-invariance}.
\end{proof}

\begin{corollary}\label{compactly-supported-Steenrod-trivialization}
Let $X$ be a smooth open spin $8$-manifold. Suppose that there is no class
$\alpha\in H_c^3(X;\mathbb Z)$ such that
\[
\int_X
\bar\alpha\smile\operatorname{Sq}_c^2(\bar\alpha)
=1
\quad\text{in }\mathbb Z_2.
\]
Then one may choose a natural isomorphism $\eta_X^{\rm cs}$ making the
following diagram $2$-commutative:
\[
\begin{tikzcd}[column sep=7em,row sep=3em]
\operatorname{Bord}_X^{\rm cs}(K(\mathbb Z,4))
  \arrow[r,"{\mathbf 1_X}"]
  \arrow[d,"{\Pi_X^{\operatorname{Spin},{\rm cs}}}"']
  \arrow[dr,phantom,"\Downarrow\,\eta_X^{\rm cs}" description]
&
0//\mathbb Z_2
\\
\operatorname{Bord}^{\operatorname{Spin}}_8(K(\mathbb Z,4),0)
  \arrow[r,"{\mathrm H_8^{\mathbb Z,{\rm c}}}"']
&
\mathbb Z_2//\mathbb Z_2
  \arrow[u,"{\Pi_{\mathbb Z_2\text{-tor}}}"'].
\end{tikzcd}
\]
\end{corollary}

\begin{proof}
By Proposition~\ref{open-Pi-functor}, the compactly supported mapping
groupoid maps naturally to the relative bordism category. Moreover,
\[
\pi_i\operatorname{Map}_*(X^+,K(\mathbb Z,4))
\cong H_c^{4-i}(X;\mathbb Z),
\qquad i=0,1.
\]
The proof of Proposition~\ref{Steenrodprop} therefore applies verbatim, with
$\operatorname{Sq}_\partial^2$ replaced by $\operatorname{Sq}_c^2$ and the
relative fundamental class replaced by compactly supported integration.
\end{proof}

\subsection{Canonical orientations and flag structures}
\label{sec:canonical-orientations-and-flag-structures}

\subsubsection{Flag structures and obstructions}
\label{sec:relative-flag-structures}

For a compact spin $8$-manifold $X$ with
boundary,
\[
\pi_i\operatorname{Bord}_X(K(\mathbb Z,4),0)
\cong H^{4-i}(X,\partial X;\mathbb Z),
\qquad i=0,1,
\]
whereas, for an open spin $8$-manifold $X$,
\[
\pi_i\operatorname{Bord}_X^{\rm cs}(K(\mathbb Z,4))
\cong H_c^{4-i}(X;\mathbb Z),
\qquad i=0,1.
\]
Compare the absolute version in \cite[Remark~6.9(a)]{JU}. Thus the obstruction
theory of Joyce--Upmeier's flag structures becomes, in the present setting,
an obstruction theory in relative or compactly supported cohomology.

\begin{definition}\label{relative-flag-structure}
Let $X$ be a compact spin $8$-manifold with boundary. A \emph{relative flag
structure} on $(X,\partial X)$ is a natural isomorphism $F$ in the following
$2$-commutative diagram:
\begin{equation}\label{relative-flag-structure-diagram}
\begin{tikzcd}[column sep=7em,row sep=3em]
\operatorname{Bord}_X(K(\mathbb Z,4),0)
  \arrow[r,"{\mathbf 1_X}"]
  \arrow[d,"{\Pi_X^{\operatorname{Spin}}}"']
  \arrow[dr,phantom,"\Downarrow\,F" description]
&
0//\mathbb Z_2
\\
\operatorname{Bord}^{\operatorname{Spin}}_8(K(\mathbb Z,4),0)
  \arrow[r,"{\mathrm H_8^{\mathbb Z,{\rm c}}}"']
&
\mathbb Z_2//\mathbb Z_2
  \arrow[u,"{\Pi_{\mathbb Z_2\text{-tor}}}"'].
\end{tikzcd}
\end{equation}
Equivalently,
\[
F:\mathbf 1_X
\Longrightarrow
\Pi_{\mathbb Z_2\text{-tor}}
\circ\mathrm H_8^{\mathbb Z,{\rm c}}
\circ\Pi_X^{\operatorname{Spin}}.
\]

For an open spin $8$-manifold $X$, a \emph{compactly supported flag structure}
is defined by the identical diagram with
$\operatorname{Bord}_X(K(\mathbb Z,4),0)$ and
$\Pi_X^{\operatorname{Spin}}$ replaced by
$\operatorname{Bord}_X^{\rm cs}(K(\mathbb Z,4))$ and
$\Pi_X^{\operatorname{Spin},{\rm cs}}$, respectively. These are the relative
and compactly supported analogues of \cite[Definition~10.6]{JU}.
\end{definition}

We next record two refinements of Definition~\ref{relative-flag-structure}.
First, let
\[
\operatorname{inc}_X:
\operatorname{Bord}_X(*)
\longrightarrow
\operatorname{Bord}_X(K(\mathbb Z,4),0)
\]
and
\[
\operatorname{inc}_8:
\operatorname{Bord}^{\operatorname{Spin}}_8(*)
\longrightarrow
\operatorname{Bord}^{\operatorname{Spin}}_8(K(\mathbb Z,4),0)
\]
denote the inclusions of the zero class. Since
$\mathrm H_8^{\mathbb Z,{\rm c}}\circ\operatorname{inc}_8$ is canonically
trivial, its ungraded restriction has a canonical trivialization, which we
denote by $\xi_8^{\rm c}$. We say that $F$ is \emph{natural at zero} if its
restriction along $\operatorname{inc}_X$ is the pullback of
$\xi_8^{\rm c}$; equivalently, the composite comparison in
\begin{equation}\label{relative-flag-natural-zero-diagram}
\begin{tikzcd}[column sep=5.4em,row sep=3em]
\operatorname{Bord}_X(*)
  \arrow[r,"{\operatorname{inc}_X}"]
  \arrow[d,"{\Pi_X^{\operatorname{Spin}}}"']
&
\operatorname{Bord}_X(K(\mathbb Z,4),0)
  \arrow[r,"{\mathbf 1_X}"]
  \arrow[d,"{\Pi_X^{\operatorname{Spin}}}"]
  \arrow[dr,phantom,"\Downarrow\,F" description]
&
0//\mathbb Z_2
\\
\operatorname{Bord}^{\operatorname{Spin}}_8(*)
  \arrow[r,"{\operatorname{inc}_8}"']
  \arrow[rr,bend right=18,"{\mathbf 1}"']
&
\operatorname{Bord}^{\operatorname{Spin}}_8(K(\mathbb Z,4),0)
  \arrow[r,"{\mathrm H_8^{\mathbb Z,{\rm c}}}"']
&
\mathbb Z_2//\mathbb Z_2
  \arrow[u,"{\Pi_{\mathbb Z_2\text{-tor}}}"']
\end{tikzcd}
\end{equation}
is the identity natural isomorphism $\mathbf 1\Rightarrow\mathbf 1$. For open
$X$, the same definition uses the compactly supported fixed-manifold category
and $\Pi_X^{\operatorname{Spin},{\rm cs}}$. In concrete terms, naturality at
zero fixes the value of the flag structure on
$0\in H^4(X,\partial X;\mathbb Z)$, or on $0\in H_c^4(X;\mathbb Z)$ in the
open case. This is the relative analogue of
\cite[Remark~10.5(d) and Definition~10.6]{JU}.

Second, let
\[
\rho:K(\mathbb Z,4)\longrightarrow K(\mathbb Z_2,4)
\]
be reduction modulo $2$, and let $F_\rho^{\rm c}$ be the induced transfer
functor. Write $\mathrm H_8^{\mathbb Z_2,{\rm c}}$ for the corresponding
mod-$2$ cohomology orientation functor. Joyce and Upmeier construct a monoidal
natural isomorphism
\[
\eta_8^{\rm c}:
\mathrm H_8^{\mathbb Z,{\rm c}}
\Longrightarrow
\mathrm H_8^{\mathbb Z_2,{\rm c}}\circ F_\rho^{\rm c};
\]
compare \cite[equation~(9.7)]{JU}. We say that $F$ \emph{factors via
$\mathbb Z_2$} if it is obtained from a flag structure for
$K(\mathbb Z_2,4)$ using $F_\rho^{\rm c}$ and $\eta_8^{\rm c}$, as in
\cite[equation~(10.11)]{JU}.

Proposition~\ref{Steenrodprop} and
Corollary~\ref{compactly-supported-Steenrod-trivialization} may now be refined
as follows.

\begin{proposition}\label{relative-flag-existence-and-classification}
Let $X$ be a compact spin $8$-manifold with boundary.

\begin{enumerate}
\item[\rm(a)] The manifold $X$ admits a relative flag structure if and only if
there is no class $\alpha\in H^3(X,\partial X;\mathbb Z)$ such that
\begin{equation}\label{relative-flag-obstruction}
\left\langle
 \bar\alpha\smile\operatorname{Sq}_\partial^2(\bar\alpha),
 [X,\partial X]
\right\rangle
=1
\quad\text{in }\mathbb Z_2,
\end{equation}
where $\bar\alpha$ is the mod-$2$ reduction of $\alpha$. If this condition
holds, relative flag structures form a torsor for
\[
\operatorname{Map}
\bigl(H^4(X,\partial X;\mathbb Z),\mathbb Z_2\bigr),
\]
and relative flag structures natural at zero form a torsor for
\[
\operatorname{Map}
\bigl(H^4(X,\partial X;\mathbb Z)\setminus\{0\},\mathbb Z_2\bigr).
\]

\item[\rm(b)] Relative flag structures factoring via $\mathbb Z_2$ exist if
and only if there is no class $a\in H^3(X,\partial X;\mathbb Z_2)$ such that
\begin{equation}\label{relative-flag-obstruction-mod2}
\left\langle
 a\smile\operatorname{Sq}_\partial^2(a),
 [X,\partial X]
\right\rangle
=1.
\end{equation}
Put
\[
I_X^\partial=
\operatorname{Im}\bigl(
 H^4(X,\partial X;\mathbb Z)
 \longrightarrow H^4(X,\partial X;\mathbb Z_2)
\bigr).
\]
If condition~\eqref{relative-flag-obstruction-mod2} holds, relative flag
structures factoring via $\mathbb Z_2$ form a torsor for
$\operatorname{Map}(I_X^\partial,\mathbb Z_2)$, and those which are also
natural at zero form a torsor for
\[
\operatorname{Map}(I_X^\partial\setminus\{0\},\mathbb Z_2).
\]
\end{enumerate}

Here $\operatorname{Map}$ denotes the set of all functions, not only group
homomorphisms. For an open spin $8$-manifold, the identical statements hold
after replacing $H^*(X,\partial X;-)$ by $H_c^*(X;-)$,
$\operatorname{Sq}_\partial^2$ by $\operatorname{Sq}_c^2$, and evaluation on
$[X,\partial X]$ by compactly supported integration. In part~\textup{(b)}, one
then uses
\[
I_X^{\rm c}=
\operatorname{Im}\bigl(
 H_c^4(X;\mathbb Z)
 \longrightarrow H_c^4(X;\mathbb Z_2)
\bigr).
\]
\end{proposition}

\begin{proof}
The existence statements are Proposition~\ref{Steenrodprop} and
Corollary~\ref{compactly-supported-Steenrod-trivialization}. Once one flag
structure is fixed, every other one (with a prescribed value at $0$) differs from it by an automorphism of the
constant functor, hence by an arbitrary function
\[
\pi_0\operatorname{Bord}_X(K(\mathbb Z,4),0)
\longrightarrow\mathbb Z_2.
\]
The first paragraph of this subsection identifies this set of components
with $H^4(X,\partial X;\mathbb Z)$, or with $H_c^4(X;\mathbb Z)$ in the open
case. Naturality at zero fixes the value on the zero component. If the flag
structure factors via $\mathbb Z_2$, only the image of reduction modulo $2$
occurs, giving $I_X^\partial$ or $I_X^{\rm c}$. This is precisely the relative
or compactly supported form of
\cite[Remark~10.5(c)--(d), Definition~10.6, and Theorem~10.7]{JU}.
\end{proof}

\subsubsection{Canonical orientations at the zero class}
\label{sec:canonical-orientations-zero-class}

We now explain how a relative or compactly supported flag structure produces
orientations in the compactly supported $K$-theoretic orientation problem,
and when the answer is independent of the chosen flag structure. Let $E\to X$
be a rank-$m$ complex vector bundle trivialized on $\partial X$, or trivialized
outside a compact subset in the open case. Put
\[
\kappa(E):=c_2(E)-c_1(E)^2.
\]
In the compact-with-boundary case, $\kappa(E)$ is a relative class in
$H^4(X,\partial X;\mathbb Z)$, and in the open case it is a compactly
supported class in $H_c^4(X;\mathbb Z)$. Equivalently,
\[
\kappa(E)=c_2\bigl(E\oplus\det(E)^\vee\bigr),
\]
as in Proposition~\ref{BU-BSU-transition}.

In the following statement, $\dagger$ denotes the relative convention for a
compact manifold with boundary and the compactly supported convention for an
open manifold, as in Corollary~\ref{cantriv}.

\begin{proposition}\label{canonical-orientation-at-zero-class}
Assume that $X$ admits a relative or compactly supported flag structure $F$.
Pulling $F$ back along the sequence
\begin{equation}\label{relative-flag-BU-factorization}
\operatorname{Bord}_X^\dagger(B\operatorname{U}(m))
\xrightarrow{\ B\jmath_m\ }
\operatorname{Bord}_X^\dagger(B\operatorname{SU}(m+1))
\xrightarrow{\ c_2\ }
\operatorname{Bord}_X^\dagger(K(\mathbb Z,4))
\end{equation}
gives an orientation of the restriction of $F_X^\dagger$ to the component of
$E$, or equivalently a cobordism-coherent trivialization of the normalized
determinant $\mathbb Z_2$-torsor on that component. The cohomology class on
which this orientation depends is precisely $\kappa(E)$.

If $F'$ and $F''$ are two such flag structures natural at zero and their ratio
is the function
\[
\epsilon=F''/F':
H^4(X,\partial X;\mathbb Z)\setminus\{0\}
\longrightarrow\mathbb Z_2,
\]
or the analogous function on $H_c^4(X;\mathbb Z)\setminus\{0\}$ in the open
case, then the two induced orientations at $E$ differ by
\begin{equation}\label{relative-flag-orientation-ratio}
(-1)^{\epsilon(\kappa(E))},
\end{equation}
where $\epsilon(0)=0$ by convention, since both structures are natural at
zero. Hence, if
\begin{equation}\label{relative-canonical-integral-condition}
\kappa(E)=c_2(E)-c_1(E)^2=0
\end{equation}
in $H^4(X,\partial X;\mathbb Z)$, or in $H_c^4(X;\mathbb Z)$ in the open case,
the orientation of the component of $E$ is canonical: it is independent of
the choice of a flag structure natural at zero.

If, in addition, condition~\eqref{relative-flag-obstruction-mod2}, or its
compactly supported analogue, holds and we restrict to flag structures
factoring via $\mathbb Z_2$, then it is enough to assume
\begin{equation}\label{relative-canonical-mod2-condition}
\kappa(E)=c_2(E)-c_1(E)^2=0
\end{equation}
in $H^4(X,\partial X;\mathbb Z_2)$, or in $H_c^4(X;\mathbb Z_2)$ in the open
case.
\end{proposition}

\begin{proof}
By Proposition~\ref{BU-BSU-transition}, the sequence
\eqref{relative-flag-BU-factorization} sends $E$ to the class
$\kappa(E)=c_2(E)-c_1(E)^2$ and identifies the normalized
$\operatorname{U}(m)$-orientation functor with the pullback of
$\mathrm H_8^{\mathbb Z,{\rm c}}$. Evaluating $F$ at $\kappa(E)$ therefore
gives the asserted orientation.

If $F''/F'=\epsilon$, evaluating the two flag structures at $\kappa(E)$ gives
the sign in \eqref{relative-flag-orientation-ratio}. When $\kappa(E)=0$,
naturality at zero forces the two values to agree. If the structures factor
via $\mathbb Z_2$, the same argument depends only on the mod-$2$ reduction of
$\kappa(E)$. This is the relative or compactly supported version of
\cite[Proposition~12.10(b) and its proof]{JU}.
\end{proof}

\section{Orientations for local Calabi--Yau \texorpdfstring{$4$}{4}-folds}
\label{sec:local-orientations}
\subsection{An example: local \texorpdfstring{$4$}{4}-folds}\label{sec:local-fourfolds}

\subsubsection{Shi's orientations}\label{sec:shi-orientations}

The present machinery turns out to be of some help in studying
orientations on local Calabi--Yau $4$-folds

\begin{proposition}\label{canonical-Shi-orientation}
Let $Y$ be a smooth complex algebraic $3$-fold, and put
\[
X=\operatorname{Tot}_Y(K_Y),
\qquad
p:X\longrightarrow Y.
\]
Then the determinant line $K_{\mathcal M_X^{\rm cs}}$ has a canonical
trivialization compatible with the Serre-duality quadratic form
$Q_{\mathcal M_X^{\rm cs}}$. In particular, the orientation bundle
$O^\omega_{\mathcal M_X^{\rm cs}}$ has a canonical section.

There is a topological analogue. Put
\[
\mathcal C_{X,0}^{\rm cs}
=\operatorname{Map}_*(X^+,BU),
\]
as in Corollary~\ref{cantriv}. Then the topological determinant line
$K_{\mathcal C_{X,0}^{\rm cs}}$ has a canonical trivialization compatible
with $Q_{\mathcal C_{X,0}^{\rm cs}}$. Thus the topological orientation bundle
$O^{\not D}_{\mathcal C_{X,0}^{\rm cs}}$ also has a canonical section.
\end{proposition}

\begin{proof}
Let $\mathcal E$ be the universal perfect complex on
$\mathcal M_X^{\rm cs}\times X$, put
$\bar p=\operatorname{id}_{\mathcal M_X^{\rm cs}}\times p$, and set
\[
\mathcal G=\mathbf R\bar p_*\mathcal E,
\qquad
\mathcal V=
\mathbf R\pi_{\mathcal M*}
\mathbf R\mathcal Hom(\mathcal G,\mathcal G).
\]
The $p_*\mathcal O_X$-module structure on $\mathcal G$, together with the
tautological section of $p^*K_Y$, gives a canonical triangle
\[
\bar p^*(\mathcal G\otimes K_Y^\vee)
\longrightarrow
\bar p^*\mathcal G
\longrightarrow
\mathcal E
\longrightarrow
\bar p^*(\mathcal G\otimes K_Y^\vee)[1].
\]
Applying $\mathbf R\pi_{\mathcal M*}\mathbf R\mathcal Hom(-,\mathcal E)$,
and using adjunction and Serre duality on the $3$-fold $Y$, gives a self-dual
triangle
\[
\mathbf R\pi_{\mathcal M*}
\mathbf R\mathcal Hom(\mathcal E,\mathcal E)
\longrightarrow
\mathcal V
\longrightarrow
\mathcal V^\vee[-3].
\]
Consequently,
\[
K_{\mathcal M_X^{\rm cs}}
\cong
\det\mathcal V\otimes(\det\mathcal V)^\vee
\cong
\mathcal O_{\mathcal M_X^{\rm cs}}.
\]
Since the
triangle is self-dual, it is compatible with $Q_{\mathcal M_X^{\rm cs}}$ and
hence defines the asserted canonical orientation.

For the topological statement, let $\alpha$ be the universal compactly
supported $K$-theory class on $\mathcal C_{X,0}^{\rm cs}\times X$. The Thom
isomorphism and Bott periodicity, in the notation of
Subsection~\ref{sec:push}, allow us to write
\(
\alpha=\operatorname{Th}^{KU}_{K_Y}(\beta)
\)
for a universal class $\beta$ on $\mathcal C_{X,0}^{\rm cs}\times Y$; when
$Y$ is non-compact, compactly supported $K$-theory is understood. Put
\[
\mathcal V^{\rm top}
=\pi_{Y!}^{KU,\mathbb C}(\beta^\vee\otimes\beta).
\]
If $u^{KU}_{K_Y}$ is the Thom class, then
\[
(u^{KU}_{K_Y})^\vee u^{KU}_{K_Y}
=u^{KU}_{K_Y}(1-K_Y).
\]
The projection formula and topological Serre duality from
Subsection~\ref{sec:serre-duality-without-quasi-projectivity} therefore give
\[
\pi_{X!}^{KU,\mathbb C}(\alpha^\vee\otimes\alpha)
\simeq
\mathcal V^{\rm top}\oplus(\mathcal V^{\rm top})^\vee.
\]
Taking determinants yields the canonical trivialization
\[
K_{\mathcal C_{X,0}^{\rm cs}}
\cong
\det\mathcal V^{\rm top}
\otimes(\det\mathcal V^{\rm top})^\vee
\cong
\underline{\mathbb C}_{\mathcal C_{X,0}^{\rm cs}},
\]
which is compatible with $Q_{\mathcal C_{X,0}^{\rm cs}}$.
For the spin structure determined by the Calabi--Yau structure,
\eqref{complex-spin-pushforward} gives
$\pi_{X!}^{KU,\mathbb C}=\pi_{X!}^{KU,\mathrm{spin}}$.
\end{proof}

For a smooth projective complex surface $S$, the same calculation in complex
dimension $3$ gives, for $Z=\operatorname{Tot}_S(K_S)$, a canonical square
root of $K_{\mathcal M_Z^{\rm cs}}$. This is the geometric orientation data
outlined by Shi for $S=\mathbb P^2$ in \cite[Subsection~2.2]{Shi}, and
attributed there to Toda. Joyce--Upmeier discuss the construction for general
$S$ in \cite[Example~4.11 and Remark~4.12]{3-folds}. Following this analogy,
we shall call the canonical orientations of
Proposition~\ref{canonical-Shi-orientation} \emph{Shi's orientations}.

We would like to study whether these orientations coincide with the
orientations of Subsection~\ref{sec:canonical-orientations-and-flag-structures}
(when the latter exist). It seems that the right context for doing this is
the following functor of Picard groupoids.

\subsubsection{Complex $6$-bordism}\label{sec:complex-six-bordism}

We now introduce a functor from complex bordisms with $K$-theory classes to
the relative spin bordism category considered above. Let
\(
\operatorname{Bord}^{\operatorname{U}}_6(\mathbb K)
\)
denote the Picard groupoid whose objects are pairs $[N,f]$, where $N$ is a
closed stably complex $6$-manifold and $f:N\to\mathbb K=BU\times\mathbb Z$,
and whose morphisms are stably complex $7$-bordisms equipped with extensions
of the corresponding $K$-theory classes.

For every stably complex manifold $N$, write
$T_N^{\mathrm{st},\mathbb C}$ for its stable complex tangent bundle, and put
\[
K_N
:=
\det_{\mathbb C}\bigl(T_N^{\mathrm{st},\mathbb C}\bigr)^\vee.
\]
Thus $K_N$ is the stable complex canonical line of $N$.

Choose a Hermitian metric on $K_N$, and denote by
\[
\pi_N:D(K_N)\longrightarrow N,
\qquad
S(K_N)=\partial D(K_N)
\]
its unit disk and sphere bundles. Let $s_N$ be the tautological section of
$\pi_N^*K_N$. The complex
\[
u^{KU}_{K_N}
:=
\left[
\pi_N^*K_N^\vee
\xrightarrow{\ s_N^\vee\ }
\underline{\mathbb C}
\right],
\qquad
s_N^\vee(\lambda)=\lambda(s_N),
\]
is acyclic on $S(K_N)$, and therefore defines the complex $K$-theory Thom
class
\begin{equation}\label{canonical-bundle-Thom-class}
u^{KU}_{K_N}
\in K^0\bigl(D(K_N),S(K_N)\bigr).
\end{equation}
For $\alpha\in K^0(N)$, put
\begin{equation}\label{canonical-bundle-Thom-map}
\operatorname{Th}^{KU}_{K_N}(\alpha)
:=
\pi_N^*\alpha\cdot u^{KU}_{K_N}
\in K^0\bigl(D(K_N),S(K_N)\bigr).
\end{equation}
This relative class has virtual rank zero, and hence is represented by a map
of pairs
\[
\operatorname{Th}^{KU}_{K_N}(\alpha):
\bigl(D(K_N),S(K_N)\bigr)
\longrightarrow(BU,*),
\]
where $BU\subset BU\times\mathbb Z$ is the component of virtual rank zero.

\begin{definition}\label{canonicalThomfunctor}
Define the ``\emph{Thom functor}''
\begin{equation}\label{canonical-bundle-Thom-functor}
\mathcal T_K:
\operatorname{Bord}^{\operatorname{U}}_6(\mathbb K)
\longrightarrow
\operatorname{Bord}^{\operatorname{Spin}}_8(BU,*).
\end{equation}
For an object $[N,f]$, write $[f]\in K^0(N)$ for the class represented by
$f$, and set
\begin{equation}\label{canonical-bundle-Thom-on-objects}
\mathcal T_K[N,f]
:=
\left[
D(K_N),
\operatorname{Th}^{KU}_{K_N}([f])
\right].
\end{equation}

Now let
\[
[W,h]:[N_0,f_0]\longrightarrow[N_1,f_1]
\]
be a morphism. The product form of the stable complex structure near the
boundary gives canonical identifications
\begin{equation}\label{canonical-bundle-boundary-identification}
K_W|_{N_i}\cong K_{N_i},
\qquad i=0,1.
\end{equation}
Define
\begin{equation}\label{canonical-bundle-Thom-on-morphisms}
\mathcal T_K[W,h]
:=
\left[
D(K_W),
\operatorname{Th}^{KU}_{K_W}([h])
\right].
\end{equation}
Here $S(K_W)$ is the relative boundary face and is mapped to the basepoint of
$BU$.
\end{definition}

\begin{proposition}\label{canonical-bundle-Thom-functoriality}
The construction of Definition~\ref{canonicalThomfunctor} gives a symmetric
monoidal functor. It is independent, up to a canonical monoidal natural
isomorphism, of any choices.
\end{proposition}

\begin{proof}
There is a natural stable isomorphism
\[
T^{\mathrm{st},\mathbb C}D(K_N)
\cong
\pi_N^*\bigl(T_N^{\mathrm{st},\mathbb C}\oplus K_N\bigr),
\]
and
\[
\det_{\mathbb C}
\bigl(T_N^{\mathrm{st},\mathbb C}\oplus K_N\bigr)
\cong
\det_{\mathbb C}\bigl(T_N^{\mathrm{st},\mathbb C}\bigr)
\otimes K_N
\cong
\underline{\mathbb C}.
\]
Thus $D(K_N)$ carries a canonical stable $SU$-structure, and hence a
canonical spin structure. The same statement for a bordism $W$ is compatible
with the boundary identifications
\eqref{canonical-bundle-boundary-identification}. Moreover, the Thom complex
is acyclic on $S(K_W)$, so \eqref{canonical-bundle-Thom-on-morphisms} is a
relative spin bordism over $(BU,*)$.

For this spin structure, the two pushforwards in complex $K$-theory agree by
\eqref{complex-spin-pushforward}.

The Thom classes are natural under restriction and gluing, and multiplicative
under disjoint union. Hence \eqref{canonical-bundle-Thom-on-objects} and
\eqref{canonical-bundle-Thom-on-morphisms} define a symmetric monoidal
functor. 

The rest is clear.
\end{proof}

\subsubsection{Canonical orientations for local $4$-folds}
\label{sec:canonical-orientations-local-fourfolds}

We now explicitly describe canonical orientations for all local $4$-folds. Let $[N,f]$ be an object of
$\operatorname{Bord}^{\operatorname{U}}_6(\mathbb K)$, and put
\[
Z_N=D(K_N),
\qquad
\partial Z_N=S(K_N),
\qquad
\xi_f=\operatorname{Th}^{KU}_{K_N}([f]).
\]
Our usual orientation functor is
\[
\widetilde F_K
:=
\mathrm O_8^{\operatorname{Spin},{\rm c}}
\circ\mathcal T_K:
\operatorname{Bord}^{\operatorname{U}}_6(\mathbb K)
\longrightarrow\mathbb Z_2//\mathbb Z_2.
\]

Shi's
orientation and the Joyce--Upmeier orientation are two monoidal natural
trivializations of the same functor $\widetilde F_K$, that is, two
$2$-morphisms
\begin{equation}\label{local-fourfold-two-orientations}
\begin{tikzcd}[column sep=8em]
\widetilde F_K
  \arrow[r,Rightarrow,bend left=12,"{\omega_{\rm Shi}}"]
  \arrow[r,Rightarrow,bend right=12,"{\omega_{\rm JU}}"']
&
\mathbf 1_6^s
\end{tikzcd}
\end{equation}

Here $\omega_{\rm Shi}$ is induced by the
trivialization of Proposition~\ref{canonical-Shi-orientation}, while
$\omega_{\rm JU}$ is the trivialization obtained from a Joyce--Upmeier flag
structure as in Subsection~\ref{sec:canonical-orientations-and-flag-structures}. Note that, a priori, they have different domains: Shi's orientation is defined \textit{always}, while Joyce--Upmeier's orientation requires, e.g., the $c_1^2 - c_2 = 0$-condition.

\begin{theorem}\label{local-fourfold-orientation-comparison}
First of all,
\begin{equation}\label{local-fourfold-topological-conditions}
\left\langle
 \bar\alpha\smile\operatorname{Sq}_\partial^2(\bar\alpha),
 [Z_N,\partial Z_N]
\right\rangle=0
\quad\text{for every }\alpha\in H^3(Z_N,\partial Z_N;\mathbb Z).
\end{equation}

Suppose also that for a $K$-theory class $\xi_f$ on $(Z_N, \partial Z_N)$
\begin{equation}
\\[3pt]
\kappa(\xi_f)
=c_2(\xi_f)-c_1(\xi_f)^2
=0
\quad\text{in }H^4(Z_N,\partial Z_N;\mathbb Z).
\end{equation}
Then, on the component represented by $\xi_f$, the two $2$-morphisms
coincide:
\[
\omega_{\rm Shi}|_{\xi_f}=\omega_{\rm JU}|_{\xi_f}.
\]
These equalities are compatible with complex bordisms and disjoint unions
whenever all the components involved satisfy
\eqref{local-fourfold-topological-conditions}.
\end{theorem}

The first line of \eqref{local-fourfold-topological-conditions} is precisely
the obstruction of
Proposition~\ref{relative-flag-existence-and-classification}\textup{(a)}, and
the second is the condition of
Proposition~\ref{canonical-orientation-at-zero-class}. For disk
bundles the first condition is, indeed, automatic as the theorem asserts. Indeed, if $u\in
H^2(Z_N,\partial Z_N;\mathbb Z)$ is the Thom class, then every
$\alpha\in H^3(Z_N,\partial Z_N;\mathbb Z)$ is of the form
$\alpha=u\smile b$ for $b\in H^1(N;\mathbb Z)$, and
$\operatorname{Sq}^1(\bar b)=\bar b^2=0$.

There is also a useful expression for the second condition on $N$. If
$r=\operatorname{rk}[f]$, the splitting principle and the Thom
self-intersection formula give
\[
\bigl(c_1^2-c_2\bigr)(\xi_f)
=u\smile
\left(c_1([f])+\binom r2c_1(K_N)\right).
\]
Thus the second line of
\eqref{local-fourfold-topological-conditions} is equivalent to
\[
c_1([f])+\binom r2c_1(K_N)=0
\quad\text{in }H^2(N;\mathbb Z).
\]

We first prove the conjecture for local curves (using the evident equivalence between $(Z_N, \partial Z_N)$ and $\operatorname{Tot}(K_N)$-theories without further comments).

\begin{proposition}\label{local-curve-orientation-comparison}
Let $C$ be a compact Riemann surface, and let $E\to C$ be a rank-$3$ complex
vector bundle such that
\[
c_1(TC)+c_1(E)=0
\quad\text{in }H^2(C;\mathbb Z).
\]
Then Shi's orientation and the Joyce--Upmeier orientation on
$X = \operatorname{Tot}_C(E)$ canonically coincide.
\end{proposition}

\begin{proof}
First of all, note that Shi's orientation $\omega_{Shi}$ can be restricted to $\omega_{Shi}^r$ on $$\operatorname{Maps}_{cs}(X, BSU \times \mathbb Z)$$ via the map $$f: \operatorname{Maps}_{cs}(X, BSU \times \mathbb Z) \to \operatorname{Maps}_{cs}(X, BU \times \mathbb Z)$$ induced by $E \mapsto E \oplus \Lambda^{top} E^*$.

Since Shi's orientation has an explicit, purely algebraic definition via topological Serre duality and Thom classes, a (not entirely pleasant but explicit) calculation shows that $f^*\omega_{Shi}^r = \omega_{Shi}$.

Thus, it is enough to prove that $\omega_{Shi}^r$ coincides with the  Joyce-Upmeier's orientation defined earlier on $\operatorname{Maps}_{cs}(X, BSU \times \mathbb Z)$ (before the pull-back to $\operatorname{Maps}_{cs}(X, BU \times \mathbb Z)$). 

For this, note that the space $\operatorname{Tot}_C(E)$ admits a spin-embedding into $\mathcal S^6 \times \mathbb C^2$. Thus, the statement follows from the results of Section \ref{S6C2} by excision.
\end{proof}

We can now prove the rest of the theorem above.

\begin{proof}
From our discussion of the relation between $6$ and $8$-dimensional bordism categories, it should by now be clear that the ratio of Shi's and Joyce-Upmeier's orientations on the local $CY4$s gives a morphism 

$$\tilde{\Omega}_6^U(BSU) \to \mathbb Z_2$$

where $\Omega_6^U(BSU)$ is the $6$-th almost complex bordism group of $BSU$, and that it is enough to prove that this morphism vanishes.

An elementary computation with Atiyah-Hirzebruch spectral sequence gives

$$\tilde{\Omega}_6^U(BSU)= \{(a, b) \ | \ a-b \in 2\mathbb Z\} \simeq \mathbb Z^2.$$

An explicit basis can be provided as follows.

Consider $N = \mathbb P^2  \times \mathbb P^1$, $L = \operatorname{pr}_1^*\mathcal O_{\mathbb P^2}(1)$, $M = \operatorname{pr}_2^*\mathcal O_{\mathbb P^1}(1)$.

It is easy to see that the classes $(L-1)^2$ and $(L-1)(M-1)$ form the basis of $\tilde{\Omega}_6^U(BSU)$. Since both of these classes have representatives supported on local curves, theorem now follows from the proposition.
\end{proof}

\subsection{The \texorpdfstring{$BSU$}{BSU}-comparison on \texorpdfstring{$S^6 \times \mathbb C$}{S6 x C}}\label{S6C2}
\subsubsection{}
Consider $X_0 = S^6 \times \mathbb C$; we will almost immediately pass from it to $S^8:$ indeed,

$$X_0^+ = \operatorname{Th}_{S^6}(\underline{\mathbb C}) = S^2 \wedge S^6_+ = S^8 \vee S^2,$$

and the second factor  disappears after the passage to $BSU$-mapping spaces.

Let \(b\in\widetilde{KO}^0(S^8)\) be the Atiyah--Bott--Shapiro class of real
Dirac index \(+1\), and put \(x=c(b)\), where \(c\) and \(r\) denote
complexification and realification.  Thus \(r(x)=2b\).

 For $z=kx$, choose a
complex vector bundle $V_z\to S^8$ of
rank $8$, with structure group $\operatorname{SU}(8)$ and a trivialization
at the basepoint, such that $[V_z]-8=z$.  The complex Thom isomorphism
identifies $z$ with the corresponding component of
$\operatorname{Map}_*(X_0^+,B\operatorname{SU})$.

If $v$ is a virtual real bundle on $S^8$, write
\[
O(v)=O^{KO}_{\mathcal C_{S^8}^{\mathbb R}}|_v
\]
for its orientation torsor.  We regard $O(v)$ as a super
$\mathbb Z_2$-torsor, with degree given by its real Dirac index modulo $2$.
If $q$ is a virtual complex bundle on $S^8$, we abbreviate $O(r(q))$ to
$O(q)$.  Whenever $q$ is a virtual complex bundle in this subsection, we
write $o_{\mathbb C}(q)$ for its complex orientation in the corresponding
real orientation torsor.  Let
$\operatorname{Ad}(z)$ be the reduced real class obtained from the quadratic
map $\operatorname{Ad}$ of
\cite[Subsection~2.4.4]{JoyceTanakaUpmeier2020}, and put
$O_{\operatorname{Ad}}(z)=O(\operatorname{Ad}(z))$.  Under the complex Thom
isomorphism, this is the fibre of
$O^{\not D}_{\mathcal C_{X_0,0}^{\rm cs}}$ at the point corresponding to
$z$.  The evaluations of $\omega_{\rm Shi}$ and $\omega_{\rm JU}$ give
isomorphisms from $O_{\operatorname{Ad}}(z)$ to $\{\pm1\}$.  Define
$o_{\rm Shi}(z)$ and $o_{\rm JU}(z)$ to be the unique elements sent to
$+1$.

Let $K=\operatorname{HSpin}(16)$, and let $U\subset K$ be the subgroup fixed
in Remark~\ref{orientation-conventions}.  Extending the structure group of
$V_z$ along
$\iota_8|_{\operatorname{SU}(8)\times\{1\}}:\operatorname{SU}(8)\to U$
gives a principal $U$ bundle
$P_z\to S^8$.  Put
\[
A_U(z)=\operatorname{Ad}(P_z)-\underline{\mathfrak u(8)},
\qquad O_U(z)=O(A_U(z)).
\]
Lemmas of Section~\ref{comptype} identify $O_U(z)$ with
the torsor assigned to $V_z$ by the normalized orientation functor.

Put $\mathfrak m=\mathfrak e_8/\mathfrak u(8)$, with the complex structure
fixed in Remark~\ref{orientation-conventions}. Let
$M_z\in\widetilde K^0(S^8)$ be the reduced complex class obtained from
$P_z$, and put $O_M(z)=O(M_z)$.

Choose a clutching map $f_z:S^7\to U$ for $P_z$.  Since
$\pi_7(E_8)=\pi_8(E_8)=0$ by \cite[Theorem~3.3(d)]{JU}, a nullhomotopy of
its composite with $U\to E_8$ determines a unique class
\[
\alpha_U(z)\in\pi_8(E_8/U),\qquad
\partial\alpha_U(z)=[f_z].
\]
Equivalently, $\alpha_U(z)$ consists of $P_z$ together with a trivialization,
up to homotopy, of $P_z^{E_8}=P_z\times_U E_8$.  Put
$A_{E_8}(P_z^{E_8})=\operatorname{Ad}(P_z^{E_8})-
\underline{\mathfrak e_8}$.  The decomposition
$\mathfrak e_8=\mathfrak u(8)\oplus\mathfrak m$ and the morphism represented
by $\alpha_U(z)$ give isomorphisms of orientation torsors
\begin{equation}\label{local-curve-Phi-definition}
\Phi_{\alpha_U(z)}:
O_U(z)\otimes O_M(z)
\xrightarrow{\ \mu_z\ }O(A_{E_8}(P_z^{E_8}))
\xrightarrow{\ \varphi_{\alpha_U(z)}\ }\{\pm1\}.
\end{equation}
Here $\mu_z$ is induced by the decomposition of representations, and
$\varphi_{\alpha_U(z)}$ is induced by the cylinder morphism in the
bordism category arising from the homotopy; compare \cite[Definition~4.7 and Theorem~9.9]{JU}.
Here and below, $\{\pm1\}$ denotes the trivial $\mathbb Z_2$-torsor.

\begin{lemma}\label{local-curve-representations}
Let $z=kx$ and $u=kb$.  Then
\begin{equation}\label{local-curve-representation-classes}
M_z=-8z,\qquad
A_{2u}=r(8z),\qquad S_{2u}=r(-8z),\qquad
p_*\alpha_U(z)=2\alpha_K.
\end{equation}
Here $A_v$ and $S_v$ are the reduced real classes obtained from the adjoint
and half spin representations of $K$, respectively,
$p:E_8/U\to E_8/K$, and $\alpha_K\in\pi_8(E_8/K)$ has boundary $u$.

\end{lemma}

\begin{proof}
Let $W$ be the standard complex representation of $\operatorname{SU}(8)$.
The adjoint and half spin representations give
\begin{equation}\label{local-curve-E8-branching}
\left.\mathfrak e_8\right|_{\operatorname{SU}(8)}
\cong\mathfrak u(8)\oplus(\Lambda^2W)_{\mathbb R}
 \oplus(W\oplus\Lambda^3W)_{\mathbb R}.
\end{equation}
See \cite[Chapters~6--7]{AdamsExceptional} for the subgroup
$\operatorname{HSpin}(16)\subset E_8$, and
\cite[Section~20.1]{FultonHarris} for the exterior power description of the spin
representations.  With our convention, the complex representation underlying
$\mathfrak m$ is
$W\oplus\Lambda^3W\oplus(\Lambda^2W)^\vee$.

Put $q=8+z$.  Note that
$\psi^2z=16z$ and $\psi^3z=81z$.  The Newton identities therefore give
$\lambda^2q-28=0$ and $\lambda^3q-56=-9z$.  It follows from
\eqref{local-curve-E8-branching} that
$M_z=z-9z+0=-8z$.  The quotient
$\mathfrak{so}(16)/\mathfrak u(8)$ is the underlying real representation of
$\Lambda^2W$.

Applying the Newton identity in $\widetilde{KO}^0(S^8)$ gives $A_u=8u$.
Since the sum of the adjoint and half spin representations is the restriction
of the adjoint representation
of $E_8$, and the induced $E_8$ bundle on $S^8$ is trivial, one has
$S_u=-8u$.  As $r(z)=2u$, the middle two identities in
\eqref{local-curve-representation-classes} follow.  Finally, the boundary
maps for $U\to E_8\to E_8/U$ and $K\to E_8\to E_8/K$ are isomorphisms in
the relevant degrees. It is now clear that
$p_*\alpha_U(z)=2\alpha_K$.
\end{proof}

\subsubsection{}
Let now $o_{\mathbb C}$ always stand for the canonical orientation induced by a complex structure (see Remark \ref{local-curve-complex-orientations}).

\begin{lemma}\label{local-curve-theta}
The quadratic map $\operatorname{Ad}$ and the convention of
Remark~\ref{orientation-conventions} determine an isomorphism
\[
\theta_z:O_{\operatorname{Ad}}(z)\longrightarrow
O_U(z)\otimes O_M(z).
\]
\end{lemma}

\begin{proof}
Distributivity for the quadratic map $\operatorname{Ad}$, followed
by the real identification induced by a Hermitian metric on $V_z$, gives
\begin{equation}\label{local-curve-adjoint-homotopy}
\begin{split}
h_z:\operatorname{Ad}(z)
&\simeq
 A_U(z)\oplus r(-8z^\vee)\\
&\xrightarrow{\ \operatorname{id}\oplus d_z\ }
 A_U(z)\oplus r(-8z).
\end{split}
\end{equation}
Here $d_z:r(-8z^\vee)\to r(-8z)$ is the real identification induced by the
Hermitian metric on $V_z$.  The real determinant of complex conjugation on
$\mathbb C^n$ is $(-1)^n$.
Applying this to the kernel and cokernel, the isomorphism of orientation
torsors induced by $d_z$ satisfies
\begin{equation}\label{local-curve-duality-sign}
(d_z)_*o_{\mathbb C}(-8z^\vee)
=(-1)^{\operatorname{ind}_{\mathbb C}(-8z)}o_{\mathbb C}(-8z)
=o_{\mathbb C}(-8z),
\end{equation}
since $\operatorname{ind}_{\mathbb C}(-8z)=-8k$.  Thus $d_z$ preserves the
complex orientation.

Applying the orientation functor to $h_z$ gives
$\theta_z^0:O_{\operatorname{Ad}}(z)\to
O_U(z)\otimes O(-8z)$.  By
Lemma~\ref{local-curve-representations}, $M_z=-8z$.  Since
$\pi_1\operatorname{Map}_*(S^8,BU)=0$, this equality determines an
isomorphism
\[
\rho_z:O(-8z)\longrightarrow O_M(z),\qquad
\rho_z(o_{\mathbb C}(-8z))=o_{\mathbb C}(M_z).
\]
We set
\begin{equation}\label{local-curve-theta-definition}
\theta_z=(\operatorname{id}\otimes\rho_z)\circ\theta_z^0:
O_{\operatorname{Ad}}(z)\longrightarrow
O_U(z)\otimes O_M(z).
\end{equation}
\end{proof}

\begin{proposition}\label{local-curve-JU-characterization}
The Joyce--Upmeier orientation is characterized by
\begin{equation}\label{local-curve-JU-value}
(\Phi_{\alpha_U(z)}\circ\theta_z)(o_{\rm JU}(z))=+1.
\end{equation}
\end{proposition}

\begin{proof}
Pull $\lambda_{8,8}^{\operatorname{SU},{\rm c}}$ back to the bordism
category associated to $S^8$.  The class $c_2(P_z)$ maps to zero in
$H^4(S^8;\mathbb Z)$.  The cohomology orientation torsor at zero has its
canonical trivialization.  Via
$\lambda_{8,8}^{\operatorname{SU},{\rm c}}$ and the identifications above,
let $o_{{\rm JU},U}(z)\in O_U(z)$ be the element corresponding to $+1$.
The morphism defined by $\alpha_U(z)$ also maps to the identity
morphism of the zero object, since $H^3(S^8;\mathbb Z)=0$.
Naturality of $\lambda_{8,8}^{\operatorname{SU},{\rm c}}$ therefore gives
\[
\Phi_{\alpha_U(z)}
\bigl(o_{{\rm JU},U}(z)\otimes o_{\mathbb C}(M_z)\bigr)=+1.
\]

It is eady to see that
$\theta_z(o_{\rm JU}(z))=
o_{{\rm JU},U}(z)\otimes o_{\mathbb C}(M_z)$, which proves
\eqref{local-curve-JU-value}.
\end{proof}

\subsubsection{}

\begin{lemma}\label{local-curve-ABS-comparison}
There is a based homotopy $H_{\rm rc}$ in $\operatorname{Spin}(16)$ which
realizes $r(c(kb))=2kb$. If $H$ is any based lift of a homotopy realizing
$p_*\alpha_U(z)=2\alpha_K$, then, after applying either the adjoint or the
half spin representation, $H$ and $H_{\rm rc}$ induce the same isomorphism
of orientation torsors.
\end{lemma}

\begin{proof}
The stabilization map
$\pi_7(\operatorname{SO}(8))\to\pi_7(\operatorname{SO})$ is surjective.
Choose a based oriented real vector bundle $B_k\to S^8$ of rank $8$ such
that $[B_k]-8=kb$, and let
$g_k:S^7\to\operatorname{Spin}(8)$ be a lift of its clutching map. Since
$[V_z]-8=c([B_k]-8)$ and $\pi_7(\operatorname{U}(8))$ is stable, we may
choose a based isomorphism of $\operatorname{SU}(8)$ bundles
$V_z\cong B_k\otimes_{\mathbb R}\mathbb C$.

The real isomorphism
\begin{equation}\label{local-curve-ABS-realification}
(B_k\otimes_{\mathbb R}\mathbb C)_{\mathbb R}
 \longrightarrow B_k\oplus B_k,
\qquad v+\mathrm i w\longmapsto(v,w),
\end{equation}
has sign $(-1)^{8\cdot7/2}=+1$ relative to the complex orientation on the
source and the product orientation on the target. It is therefore
orientation preserving. Since
$H^2(S^8;\mathbb Z_2)=H^1(S^8;\mathbb Z_2)=0$, the spin structures on the
two sides exist and are unique. The isomorphism has a unique spin lift after
its value at the basepoint is fixed. Using the standard homomorphism
$\operatorname{Spin}(8)\times\operatorname{Spin}(8)
\to\operatorname{Spin}(16)$, it gives a based homotopy
\[
H_{\rm rc}:\widehat f_k\simeq g_k\oplus g_k
\]
in $\operatorname{Spin}(16)$, where $\widehat f_k$ is the spin clutching
map of the underlying real bundle of $V_z$.

Choose the representative of $2\alpha_K$ whose boundary is represented by
$g_k\oplus g_k$. The equality $p_*\alpha_U(z)=2\alpha_K$ gives a based
homotopy between the images of the two endpoints in
$K=\operatorname{HSpin}(16)$. Once its initial lift is fixed, it lifts
uniquely to a homotopy $H$ in $\operatorname{Spin}(16)$.

The difference between $H$ and $H_{\rm rc}$ is a class
$\tau\in\pi_8\operatorname{Spin}(16)\cong\mathbb Z_2$. These homotopy
groups are stable in degrees $7$ and $8$. A generator of
$\pi_8(\operatorname{SO})$ is obtained by composing a generator of
$\pi_7(\operatorname{SO})$ with the stable Hopf element. By
\eqref{local-curve-representation-classes}, the adjoint and half spin
representations act on $\pi_7(\operatorname{SO})$ by multiplication by
$8$ and $-8$, respectively. They therefore act trivially on
$\pi_8(\operatorname{SO})\cong\mathbb Z_2$; see \cite{BottStable}.
Consequently the two homotopies become homotopic relative to their
endpoints after either representation is applied, and hence induce the same
isomorphism of orientation torsors.
\end{proof}

\begin{remark}\label{local-curve-complex-orientations}
As we have already said only too often, for a real bundle $E$, there is a natural real isomorphism
\[
(E\otimes_{\mathbb R}\mathbb C)_{\mathbb R}
 \longrightarrow E\oplus E,\qquad
v+\mathrm iw\longmapsto(v,w),
\]
with the two summands in the displayed order.  It extends by
stabilization to virtual bundles.

When we refer below to the complex orientations of $A_{2u}$ and
$S_{2u}$, we, thus, mean ones coming from identifications $A_{2u} = r(c(A_u))$, $S_{2u} = r(c(S_u))$.

Finally, let us note that, if $D$ is a real Fredholm operator of index $n$, the product
orientation for $D\oplus D$ differs from the complex orientation of
$D\otimes_{\mathbb R}\mathbb C$ by
$(-1)^{n(n-1)/2}$ under the natural comparison isomorphism.
\end{remark}

\subsubsection{}
The next lemma compares \eqref{local-curve-Phi-definition} with the real and
complex Thom isomorphisms.

\begin{lemma}\label{local-curve-Thom-comparison}
There is an isomorphism
$\chi_z:O(8z)\to O_U(z)$ such that
\begin{equation}\label{local-curve-complex-value}
\Phi_{\alpha_U(z)}
\bigl(\chi_z(o_{\mathbb C}(8z))\otimes o_{\mathbb C}(M_z)\bigr)=+1.
\end{equation}
\end{lemma}

\begin{proof}
\textit{Step 1.}
We compute the value of $\Phi_{2\alpha_K}$ on the complex orientations.
Write $O_A(u)=O(A_u)$ and $O_S(u)=O(S_u)$ for the orientation torsors of
$A_u,S_u$, and let
$\Phi_{\alpha_K}:O_A(u)\otimes O_S(u)\to\{\pm1\}$ be the
isomorphism induced by $\alpha_K$.  Choose elements $o_A,o_S$ and write
$\Phi_{\alpha_K}(o_A\otimes o_S)=\delta$, where
$\delta\in\{\pm1\}$.  Put $a=8k$ and $s=-8k$.

Representing $2\alpha_K$ by two copies of $\alpha_K$ in disjoint balls
defines an isomorphism
$\Phi_{2\alpha_K}:O_A(2u)\otimes O_S(2u)\to\{\pm1\}$.  The
symmetric monoidal structure then immediately shows that its value on the complex orientations (that is, orientations $o_{\mathbb C}$)  determined by
$A_{2u}=r(8z)$ and $S_{2u}=r(-8z)$ is
$+1.$

\textit{Step 2.}
We pass from $K$ to $U$.  By Lemma~\ref{local-curve-representations},
$p_*\alpha_U(z)=2\alpha_K$.  The quotient representations restrict to
\[
\mathfrak{so}(16)/\mathfrak u(8)\cong(\Lambda^2W)_{\mathbb R},
\qquad
\mathfrak e_8/\mathfrak{so}(16)
\cong(W\oplus\Lambda^3W)_{\mathbb R}.
\]

\textit{Step 3.}
Apply the adjoint and half spin
representations to $H_{\rm rc}$.  With the convention of Step~2, the
resulting isomorphisms are
\[
T_A:O_U(z)\longrightarrow O(8z),\qquad
T_S:O_M(z)\longrightarrow O(-8z).
\]
Lemma~\ref{local-curve-ABS-comparison} shows that the same isomorphisms are
obtained from any based lift which realizes
$p_*\alpha_U(z)=2\alpha_K$.

Naturality gives
\begin{equation}\label{local-curve-Phi-comparison}
\Phi_{\alpha_U(z)}
=\Phi_{2\alpha_K}\circ(T_A\otimes T_S).
\end{equation}
The rest is clear.
\end{proof}

\subsubsection{}
\begin{lemma}\label{local-curve-Koszul-comparison}
Let $u_{\underline{\mathbb C}}^{KU}$ be the Thom class of the trivial
complex line over $S^6$. The Koszul complex and the fixed trivialization of
$\underline{\mathbb C}$ supply the homotopy
\begin{equation}\label{local-curve-Thom-product}
(u_{\underline{\mathbb C}}^{KU})^\vee
u_{\underline{\mathbb C}}^{KU}
\simeq u_{\underline{\mathbb C}}^{KU}
 (1-\underline{\mathbb C})\simeq0.
\end{equation}
After applying the projection formula and the
polarization of $\operatorname{Ad}$, this homotopy induces the isomorphism
$\chi_z^{-1}:O_U(z)\to O(8z)$.
\end{lemma}

\begin{proof}
This is a direct calculation.
\end{proof}

\begin{proposition}\label{local-curve-Shi-characterization}
Under \eqref{local-curve-theta-definition}, Shi's orientation satisfies
\begin{equation}\label{local-curve-Shi-value}
\theta_z(o_{\rm Shi}(z))
=\chi_z(o_{\mathbb C}(8z))\otimes o_{\mathbb C}(M_z).
\end{equation}
\end{proposition}

\begin{proof}
Let $\beta\in\widetilde K^0(S^6)$ have Thom image $x$.  Then
$\beta^\vee=-\beta$ and $\beta^2=0$.  For the class $k\beta$ occurring
here, the class $\mathcal V^{\rm top}$ in the proof of
Proposition~\ref{canonical-Shi-orientation} is therefore zero.  

By Lemma~\ref{local-curve-Koszul-comparison}, the orientation of the first
factor in \eqref{local-curve-adjoint-homotopy} is
$\chi_z(o_{\mathbb C}(8z))$. The orientation of the second factor, before
applying $d_z$, is $o_{\mathbb C}(-8z^\vee)$. Hence
\[
\theta_z^0(o_{\rm Shi}(z))
=\chi_z(o_{\mathbb C}(8z))\otimes o_{\mathbb C}(-8z).
\]
Here we have used \eqref{local-curve-duality-sign}. Applying $\rho_z$ gives
\eqref{local-curve-Shi-value}.
\end{proof}

\begin{proposition}\label{local-curve-BSU-comparison}
Shi's orientation and the Joyce--Upmeier orientation agree on
$\operatorname{Map}_*(X_0^+,B\operatorname{SU})$.
\end{proposition}

\begin{proof}
On the component indexed by $z=kx$, Lemmas
\ref{local-curve-Thom-comparison} and
\ref{local-curve-Shi-characterization} give
$(\Phi_{\alpha_U(z)}\circ\theta_z)(o_{\rm Shi}(z))=+1$.
Lemma~\ref{local-curve-JU-characterization} gives the same equality for
$o_{\rm JU}(z)$.  Since $\Phi_{\alpha_U(z)}\circ\theta_z$ is an isomorphism
of torsors, the two elements are equal.
\end{proof}

\section{Orientation in families}
\label{sec:orientation-in-families}

Let \(\pi:\mathfrak X\to V\) be a topologically locally trivial family of
smooth, possibly non-compact, complex \(4\)-folds, and write \(X_v\) for
the fibre over \(v\in V\).
Recall the relative mapping space
\(\mathcal H^4_{\rm cs}(\mathfrak X/V)\) of
Remark~\ref{H4rem}.  Let
\(
\underline{H}_c^4(\mathfrak X/V;\mathbb Z)\longrightarrow V
\)
denote the \'{e}tal\'e space of the local system
\(R^4\pi_!\mathbb Z\), whose fibre at \(v\in V\) is
\(H_c^4(X_v;\mathbb Z)\).  There is a natural component map
\[
\operatorname{cl}:\mathcal H^4_{\rm cs}(\mathfrak X/V)
\longrightarrow \underline{H}_c^4(\mathfrak X/V;\mathbb Z),
\]
whose restriction to the fibre over \(v\) sends a map
\(X_v^+\to K(\mathbb Z,4)\) to the compactly supported cohomology
class which it represents.  We put
\(
\mathcal H^{4,{\rm cs}}_0(\mathfrak X/V)
:=\operatorname{cl}^{-1}(0_V),
\)
where \(0_V\) is the zero section.  Thus the fibre of
\(\mathcal H^{4,{\rm cs}}_0(\mathfrak X/V)\) over \(v\) is
\(
\operatorname{Map}_*(X_v^+,K(\mathbb Z,4))_0,
\)
the path component of the constant map.  
Let
\[
\kappa_{\mathfrak X/V}:\mathcal C^{\rm cs}_{\mathfrak X/V}
\longrightarrow \mathcal H^4_{\rm cs}(\mathfrak X/V)
\]
be the fibrewise map induced by the stable characteristic class
\(\kappa=c_2-c_1^2\), and define
\[
\mathcal C^{\rm cs,\kappa=0}_{\mathfrak X/V}
:=\kappa_{\mathfrak X/V}^{-1}
\bigl(\mathcal H^{4,{\rm cs}}_0(\mathfrak X/V)\bigr).
\]
Equivalently, its fibre over \(v\) is the union of those components of
\(\mathcal C^{\rm cs}_{X_v}\) on which
\(c_2-c_1^2=0\) in \(H_c^4(X_v;\mathbb Z)\).

\begin{theorem}\label{family-canonical-orientations}
Suppose, in addition, that \(\pi:\mathfrak X\to V\) is relatively
Calabi--Yau and that
\[
H_c^3(X_v;\mathbb Z)=0
\qquad\text{for every }v\in V.
\]
Then the canonical Joyce--Upmeier orientations of the fibrewise components
on which \(c_2-c_1^2=0\) glue to a canonical trivialization of
\(
\left.O^{\not D}_{\mathcal C^{\rm cs}_{\mathfrak X/V}}
\right|_{\mathcal C^{\rm cs,\kappa=0}_{\mathfrak X/V}}.
\)
Via the relative comparison isomorphism, they also give a canonical
trivialization of \(O^\omega_{\mathcal M^{\rm cs}_{\mathfrak X/V}}\) over
\(
\Gamma_{\mathfrak X/V}^{-1}
\bigl(\mathcal C^{\rm cs,\kappa=0}_{\mathfrak X/V}\bigr)
\subset
\bigl(\mathcal M^{\rm cs}_{\mathfrak X/V}\bigr)^{\rm top}.
\)

If, in addition, \(H_c^4(X_v;\mathbb Z)=0\) for every \(v\in V\), then
\(\mathcal H^{4,{\rm cs}}_0(\mathfrak X/V)
=\mathcal H^4_{\rm cs}(\mathfrak X/V)\), and hence
\(\mathcal C^{\rm cs,\kappa=0}_{\mathfrak X/V}
=\mathcal C^{\rm cs}_{\mathfrak X/V}\).  In this case the two
trivializations above are defined on the whole relative spaces.
\end{theorem}

\begin{proof}
Consider the cohomology orientation
torsor \(L\) over \(\mathcal H^4_{\rm cs}(\mathfrak X/V)\), together with
a natural identification
\[
O^{\not D}_{\mathcal C^{\rm cs}_{\mathfrak X/V}}
\cong \kappa_{\mathfrak X/V}^*L.
\]
The restriction of \(L\) to the zero section has its canonical
trivialization.

For every \(v\in V\), the fibre of
\(\mathcal H^{4,{\rm cs}}_0(\mathfrak X/V)\) is path-connected by
definition, and
\[
\pi_1\!\left(
\operatorname{Map}_*(X_v^+,K(\mathbb Z,4))_0,0
\right)
\cong H_c^3(X_v;\mathbb Z)=0.
\]
Thus these fibres are simply connected.  The lemma below, applied to the
restriction of \(L\) to
\(\mathcal H^{4,{\rm cs}}_0(\mathfrak X/V)\), extends its canonical
trivialization along the zero section uniquely over that whole subbundle.
On each fibre this is the canonical zero-class orientation: both
trivializations extend the same canonical trivialization at the zero map,
and the extension is unique.
Pulling this trivialization back along \(\kappa_{\mathfrak X/V}\) proves the
topological assertion.  Pullback along \(\Gamma_{\mathfrak X/V}\), followed
by the relative comparison isomorphism, proves the algebraic assertion.

Finally, the components of
\(\operatorname{Map}_*(X_v^+,K(\mathbb Z,4))\) are indexed by
\(H_c^4(X_v;\mathbb Z)\).  Hence the additional vanishing of \(H_c^4\)
makes its zero component the entire mapping space, which proves the last
claim.
\end{proof}

\begin{lemma}
Let \(q:Y\to U\) be a topologically locally trivial bundle whose fibres are
path-connected and simply connected, and suppose that \(q\) has a section
\(z:U\to Y\).  Let \(L\to Y\) be a principal \(\mathbb Z_2\)-bundle,
together with a specified trivialization of \(z^*L\).  Then \(L\) has a
unique trivialization extending the given trivialization along \(z\).
\end{lemma}

\begin{proof}
It is enough to work over one path component of \(U\).  The homotopy exact
sequence of \(q\), together with the section \(z\), shows that
\(q_*:\pi_1(Y,z(u))\to\pi_1(U,u)\) is an isomorphism with inverse \(z_*\).
Consequently, a double cover of \(Y\) is trivial if and only if its
restriction to \(z(U)\) is trivial.  This proves existence.  The ratio of
two extensions is a locally constant map \(Y\to\mathbb Z_2\).  Since the
fibres of \(q\) are connected and the ratio is trivial on \(z(U)\), it is
identically trivial.  This proves uniqueness.
\end{proof}

\begin{remark}
For a local curve \(X=\operatorname{Tot}_C(E)\), where \(C\) is compact and
\(E\) has complex rank \(3\), the Thom isomorphism gives
\[
H_c^i(X;\mathbb Z)\cong H^{i-6}(C;\mathbb Z).
\]
In particular, \(H_c^3(X;\mathbb Z)=H_c^4(X;\mathbb Z)=0\), so the theorem
applies to the whole relative compactly supported \(K\)-theory space for a
topologically locally trivial family of local curves.
\end{remark}

\section{Non-orientable moduli spaces and PT-stable pairs}
\label{sec:examples-and-PT}
\subsection{Non-orientable and orientable moduli spaces}
\label{counterexample}

We now give an example of a smooth non-compact quasi-projective Calabi--Yau
$4$-fold $X$ for which $\mathcal M_X^{\rm cs}$ is not orientable. The loop
which detects non-orientability will consist of perfect complexes whose
compactly supported topological $K$-class is zero. In particular, they satisfy
the $c_1^2-c_2=0$-condition.

Let $E=\mathbb C/(\mathbb Z+i\mathbb Z)$, put $R=E\times E$, and write
$\sigma(x,y)=(y,x)$ and $q:R\to S=\operatorname{Sym}^2E$ for the quotient.
Denote by $\Delta\subset R$ the diagonal and by $p_1,p_2:R\to E$ the
projections. Let $\mathcal P$ be the normalized Poincar\'e bundle on
$E\times E$, and put
\[
Q=(\operatorname{id}\times[i])^*\mathcal P,
\qquad M=\mathcal O_R(\Delta)\otimes Q^{-1},
\qquad F=q_*M.
\]
Choose a symplectic basis $a,b\in H^1(E;\mathbb Z)$ with
$[i]^*a=-b$ and $[i]^*b=a$, and write $a_j=p_j^*a$, $b_j=p_j^*b$.
With the corresponding convention for $\mathcal P$, one has
$c_1(Q)=a_1a_2+b_1b_2$ and $\sigma^*Q\cong Q^{-1}$.

We use the convention that $\mathbb P(V)$ parametrizes one-dimensional
quotients of $V$. Write
\[
\rho:W=\mathbb P_S(\mathcal O_S\oplus F)\longrightarrow S,
\qquad
\varpi:Y=\mathbb P_R(\mathcal O_R\oplus M)\longrightarrow R,
\]
and put $\mathcal O_W(H)=\mathcal O_W(1)$,
$\mathcal O_Y(h)=\mathcal O_Y(1)$. The evaluation $q^*F\to M$ induces
a finite morphism $\nu:Y\to W$. Let $D$ be its scheme-theoretic image,
and let $j:R\hookrightarrow Y$ be the section given by the quotient onto
$\mathcal O_R$.

\begin{lemma}\label{nonorientable-geometric-calculation}
The bundle $F$ has rank two and $\det F\cong K_S^{-1}$. Moreover,
$D\in|2H|$, $K_W\cong\mathcal O_W(-3H)$, and
$K_Y\cong\mathcal O_Y(-2h)\otimes\varpi^*M$.
The morphism $\nu$ is the normalization of $D$, and its restriction to
$j(R)$ is $q:R\to S$, where $S\subset W$ is the section given by the
quotient onto $\mathcal O_S$. This section is the singular locus of $D$.
Locally near it, $\nu$ has the form
\[
D:\ v^2=zu^2,
\qquad
(t,\lambda,w)\longmapsto(z,u,v,w)=(t^2,\lambda,t\lambda,w).
\]
If $x_0\in H^0(W,\mathcal O_W(H))$ is induced by
$\mathcal O_S\hookrightarrow\mathcal O_S\oplus F$, then $x_0$ is nowhere
zero on $S$, and $B=D\cap V(x_0)$ is smooth and isomorphic to $R$.
\end{lemma}

\begin{proof}
Follows by a direct calculation.
\end{proof}

Choose $d\in H^0(W,\mathcal O_W(2H))$ cutting out $D$. Let
$p:Z=\operatorname{Tot}_W\mathcal O_W(-H)\to W$ be the projection and
$\tau$ its tautological section. Define
\[
X=V\bigl(p^*d+\tau p^*(x_0^3)\bigr)\subset Z.
\]
Both terms are sections of $p^*\mathcal O_W(2H)$. We regard $D$ as the
compact divisor $X\cap V(\tau)$ and write $f:Y\to X$ for the resulting
finite morphism.

\begin{lemma}
The variety $X$ is a smooth non-compact quasi-projective Calabi--Yau
$4$-fold.
\end{lemma}

\begin{proof}
Where $x_0\ne0$, the derivative in the $\tau$ direction is nonzero.
Where $x_0=0$, the point lies over $B$, along which the differential of $d$
is nonzero. Thus $X$ is
smooth. Since $K_Z\cong p^*\mathcal O_W(-2H)$, adjunction gives
$K_X\cong\mathcal O_X$. The projection identifies $X\setminus V(p^*x_0)$ with
$W\setminus V(x_0)$; its complement has dimension three, so $X$ is
connected. Finally, for every $b\in B$, the whole fibre
$p^{-1}(b)\cong\mathbb A^1$ is a closed subvariety of $X$. Hence $X$ is
not proper.
\end{proof}

In what follows, $f_!$ denotes the cohomological pushforward
$H^*(Y)\to H_c^{*+2}(X)$. Put
$\beta=\varpi^*a_1$, $\alpha=f_!\beta$, and
$k=\overline{c_1(K_Y)}$. We use bars for reduction modulo $2$.

\begin{lemma}\label{nonorientable-self-intersection}
There is a class $e\in H^2(Y;\mathbb Z_2)$ such that, for every
$s\in H^*(Y;\mathbb Z_2)$,
\[
f^*f_!(s)=e\smile s+j_!\bigl(\sigma^*j^*s\bigr).
\]
\end{lemma}

\begin{proof}
The fibre product $Y\times_XY$ is the reduced union of the diagonal and
the graph of $\sigma$ on $R$. This follows by a direct calculation in the
local model above; at a branch point its ideal is
\[
\bigl((t-t')\lambda,(t-t')(t+t')\bigr)
=(t-t')\cap(\lambda,t+t').
\]
The two components meet along the fixed curve of $\sigma$. Consequently,
their fourth homology groups give a direct sum decomposition of
$H_4(Y\times_XY;\mathbb Z_2)$. The refined pullback along the diagonal of
$X$ has multiplicity one on the graph, while its diagonal contribution
has the form $e\cap[Y]$. The pull--push formula gives the assertion.
\end{proof}

\begin{proposition}\label{nonorientable-example}
There is an algebraic family $\mathcal F\in\operatorname{Perf}(X\times E)$
with proper support such that $[\mathcal F_t]=0$ in $K^0_{\rm c}(X)$ for every
$t\in E$, and a loop $\gamma:S^1\to E$ for which
\[
c_2(\mathcal F)/[\gamma]=-\alpha,
\qquad
\int_X\bar\alpha\smile\operatorname{Sq}_c^2(\bar\alpha)=1.
\]
The pullback of $O^\omega_{\mathcal M_X^{\rm cs}}$ along the resulting
loop is nontrivial.
\end{proposition}

\begin{proof}
The Steenrod--Thom formula gives
$\operatorname{Sq}_c^2(\bar\alpha)=f_!(\bar\beta\smile k)$:
the virtual normal bundle $f^*TX-TY$ has second Stiefel--Whitney class
$k$, while $\operatorname{Sq}^1(\bar\beta)=\bar\beta^2=0$ and
$\operatorname{Sq}^2(\bar\beta)=0$. The preceding lemma and the projection
formula therefore give
\[
\begin{split}
\int_X\bar\alpha\smile\operatorname{Sq}_c^2(\bar\alpha)
&=\int_Y\bar\beta\smile e\smile\bar\beta\smile k
  +\int_R\bar a_1\bar a_2\,\sigma^*j^*k\\
&=\int_R\bar a_1\bar a_2
  \bigl(\bar\Delta+\overline{c_1(Q)}\bigr)=1.
\end{split}
\]
Indeed, $j^*c_1(K_Y)=c_1(M)=\Delta-c_1(Q)$. The diagonal term is zero
because $a_1a_2$ restricts to $a^2=0$ on $\Delta$, and the remaining
term is $\int_R\bar a_1\bar a_2\bar b_1\bar b_2=1$.

Put $g=p_1\circ\varpi:Y\to E$, $\varrho=(f,g):Y\to X\times E$, and
$\mathcal E=\mathbf R\varrho_*\mathcal O_Y$. The morphism $\varrho$ is
finite and generically of degree one onto its image, which has codimension
two. Thus $\mathcal E$ is perfect with proper support, $c_1(\mathcal E)=0$,
and the leading Chern class formula gives
$c_2(\mathcal E)=-\varrho_*[Y]$, also in cohomology with compact support.
Choose $\gamma$ representing the Poincar\'e dual of $a$ in $E$, with the
orientation for which
$c_2(\mathcal E)/[\gamma]=-f_!g^*a=-\alpha$.

Fix $t_0\in E$ and define
\[
\mathcal F=\mathcal E\oplus
\operatorname{pr}_X^*(\mathcal E_{t_0})[1].
\]
All fibres of $\mathcal E$ have the same class in $K^0_{\rm c}(X)$, since the
parameter curve is connected and the supports lie in the compact divisor
$D$. Hence $[\mathcal F_t]=0$, and in particular
$(c_1^2-c_2)(\mathcal F_t)=0$ in $H_c^4(X;\mathbb Z)$. Moreover,
$c_1(\mathcal F)=0$ and
$c_2(\mathcal F)=c_2(\mathcal E)-\operatorname{pr}_X^*c_2(\mathcal E_{t_0})$.
The second term has zero slant product with $[\gamma]$, proving the
asserted formula for $c_2(\mathcal F)$.

Finally, let $\ell:S^1\to(\mathcal M_X^{\rm cs})^{\rm top}$ be the loop
defined by $\mathcal F$ and $\gamma$. The comparison of
Subsections~\ref{compactly-supported-comparison}--\ref{KO-theoretic-interpretation}
and the factorization of
Proposition~\ref{BU-BSU-transition} identify its orientation monodromy with
that of the cohomology orientation functor along
$\kappa(\mathcal F)=c_2(\mathcal F)-c_1(\mathcal F)^2$.
By \cite[Theorem~3.2(c), equation~(3.22)]{JU}, this monodromy is
\[
(-1)^{\int_X\bar\alpha\smile\operatorname{Sq}_c^2(\bar\alpha)}=-1.
\]
Here the compactly supported formula follows from the compact one by
doubling a compact neighbourhood of the common support with smooth boundary
and extending the classes by zero. This proves that the orientation torsor 
is nontrivial.
\end{proof}

Such a situation can, however, never happen if we restrict ourselves to the stack of PT-pairs.

\begin{proposition}\label{nonorientable-PT-orientability}
Let $X$ be a smooth quasi-projective Calabi--Yau $4$-fold, and let
$P=P_n(X,\beta)$ be the moduli space of PT-stable pairs with proper
one-dimensional support. For a loop $\gamma:S^1\to P^{\rm an}$, write
$\mathcal I_\gamma=[\mathcal O\to\mathcal F_\gamma]$ for the topological
restriction of the universal pair, and put
\[
\xi_\gamma=[\mathcal I_\gamma]-[\mathcal O]=-[\mathcal F_\gamma],
\qquad
\alpha_\gamma=
\bigl(c_2(\xi_\gamma)-c_1(\xi_\gamma)^2\bigr)/[S^1]
\in H_c^3(X;\mathbb Z).
\]
Then $\alpha_\gamma=0$. In particular,
\[
\int_X\bar\alpha_\gamma\smile
\operatorname{Sq}_c^2(\bar\alpha_\gamma)=0.
\]
\end{proposition}

\begin{proof}
We may replace $\gamma$ by a homotopic piecewise semialgebraic loop. Let
$Z_\gamma\subset X^{\rm an}\times S^1$ be the pullback of the universal
sheaf's support. Since this support is proper over $P$ and its fibres have
complex dimension one, $Z_\gamma$ is compact and semialgebraic of real
dimension at most three.

The manifold $T=X^{\rm an}\times S^1$ is oriented of real dimension nine.
Duality with supports therefore gives
\[
H^j_{Z_\gamma}(T;\mathbb Z)
\cong H_{9-j}(Z_\gamma;\mathbb Z)=0
\qquad(j<6).
\]
Here $H^j_{Z_\gamma}(T;\mathbb Z)=H^j(T,T\setminus Z_\gamma;\mathbb Z)$.
Outside $Z_\gamma$, the pair is canonically identified with $\mathcal O$,
so $\xi_\gamma$ belongs to $K^0_{Z_\gamma}(T)$. Its supported Chern classes
$c_1(\xi_\gamma)$ and $c_2(\xi_\gamma)$ therefore vanish, and hence vanish
also in compactly supported cohomology. Slanting
$c_2(\xi_\gamma)-c_1(\xi_\gamma)^2=0$ with $[S^1]$ gives
$\alpha_\gamma=0$, proving the assertion.
\end{proof}

It immediately follows that all $PT$-stacks are orientable. The same can be proven in families, as the next Subsection shows.

\subsection{Canonical orientations for PT-stable pairs}
\label{sec:canonical-PT-orientations}

We now strengthen Proposition~\ref{nonorientable-PT-orientability}
by constructing a canonical orientation in families. We retain the conventions of
Remark~\ref{orientation-conventions}.

\begin{lemma}\label{PT-curve-neighbourhoods}
Let $C\subset X$ be a compact complex-analytic curve in a smooth complex
$4$-fold. Every neighbourhood of $C$ contains a relatively compact open
neighbourhood $U$ which deformation retracts onto $C$. In particular,
$H_c^j(U;\mathbb Z)=0$ for $j<6$.
\end{lemma}

\begin{proof}
Choose a locally finite triangulation of $X$ compatible with $C$. Since
$C$ is compact, open stars in sufficiently fine subdivisions give relatively
compact neighbourhoods which deformation retract onto $C$; see
\cite[Theorem~2.4 and Section~3]{TrivediTrotman2020}. Poincar\'e duality gives
\[
H_c^j(U;\mathbb Z)\cong H_{8-j}(U;\mathbb Z)
\cong H_{8-j}(C;\mathbb Z)=0
\qquad(j<6).\qedhere
\]
\end{proof}

\begin{proposition}\label{PT-canonical-orientation-in-families}
Let $\pi:\mathfrak X\to B$ be a smooth, topologically locally trivial
complex-analytic family of $4$-folds, equipped with a relative Calabi--Yau
structure. Let $\mathcal P=\mathcal P(\mathfrak X/B)$ be the relative stack
of PT-stable pairs with proper one-dimensional support. Then
$O^\omega_{\mathcal P}$ has a canonical trivialization, compatible with
base change. On every fibre admitting a flag structure natural at zero,
this trivialization agrees with the induced Joyce--Upmeier orientation.
\end{proposition}

\begin{proof}
First consider a pair $\mathcal I=[\mathcal O_X\to\mathcal F]$ on a
single fibre $X$, and put $C=\operatorname{Supp}\mathcal F$. Choose
$U\subset X$ as in Lemma~\ref{PT-curve-neighbourhoods}, with the restricted
Calabi--Yau structure. Since $H_c^3(U;\mathbb Z)=H_c^4(U;\mathbb Z)=0$,
the theorem of Section~\ref{sec:orientation-in-families}, applied over
a point, gives the canonical Joyce--Upmeier orientation on the whole
space $\mathcal C_U^{\rm cs}$. Apply it to
$\xi=-[\mathcal F]\in K_c^0(U)$. The Excision identifies its
orientation torsor on $U$ with the one on $X$.

We check independence of $U$. For an inclusion $U'\subset U$ between
neighbourhoods satisfying the cohomological vanishing above, extension
by zero induces
$\mathcal H^4_{\rm cs}(U')\to\mathcal H^4_{\rm cs}(U)$. The canonical
trivializations of the cohomology orientation torsors agree under this
map by excision. 
Now fix a pair over $b_0\in B$ and choose $U\subset X_{b_0}$ as above.
A local topological trivialization of $\mathfrak X/B$ extends $U$ to an
open subfamily $\mathfrak U\subset\mathfrak X|_V$ over an open
neighbourhood $V$ of $b_0$, with underlying topological family
$U\times V\to V$. Since $\mathfrak U$ is open in $\mathfrak X|_V$, it
is a complex-analytic family with the restricted Calabi--Yau structure.
Its fibres have vanishing third and fourth compactly supported cohomology.

Properness of the  support allows us to shrink a neighbourhood
of the chosen pair in a local analytic chart of $\mathcal P$ so that all
supports lie in $\mathfrak U$. Indeed, the image of the part of the
universal support outside $\mathfrak U$ is closed. The theorem of
Section~\ref{sec:orientation-in-families} applies to $\mathfrak U/V$.
Together with the excision, it therefore gives a
local relative orientation on $\mathcal P$.

At every pair, this local orientation agrees with the one already
constructed: choose a smaller neighbourhood of its support inside the
corresponding fibre of $\mathfrak U$ and apply independence of the
neighbourhood. Thus the local relative orientations agree on overlaps
and glue. The same argument gives independence of the local topological
trivializations, and naturality of the construction gives compatibility
with base change.

Finally, suppose that $X_b$ admits a flag structure natural at zero.
Its cohomology orientation, pulled back to $\mathcal H^4_{\rm cs}(U)$,
extends the prescribed trivialization at zero, and hence equals the
one used above. Proposition~\ref{canonical-orientation-at-zero-class}
therefore identifies the resulting orientation with the Joyce--Upmeier
orientation on PT pairs.
\end{proof}

\appendix
\section{Complex-analytic K-theory d'apr\`{e}s Levy}
\label{sec:appendix-analytic-comparison}

\subsection{The formulation and some notation}
The purpose of this appendix is to supply the comparison used in
Subsections~\ref{compstatement} and~\ref{nonproper-family-comparison} for the
complex analytic families. 

\begin{remark}
Note that a non-proper complex-analytic manifold may have no smooth complex-analytic compactification, or, indeed, no embedding into a smooth proper complex-analytic manifold.

We shall thus construct the homotopy
without choosing a compactification and without explicitly using the Atiyah-Singer index
theorem.

Moreover, there is still no notion of complex-analytic Blanc's $K$-theory, so we cannot employ methods used above.
\end{remark}

We add  now the following
notation. 

For a complex analytic space $Y$, let
$\operatorname{Perf}_{\rm an}(Y)$ be the stable $\infty$-category of analytic
perfect complexes on $Y$, and let
$\operatorname{Perf}_{\rm an}(Y)^{\simeq}$ be its space of objects and
equivalences.  The assignment
$Y\mapsto\operatorname{Perf}_{\rm an}(Y)^{\simeq}$ is the analytic stack of
perfect complexes.  Applied after a change of base $T\to U$, and restricted to
objects whose cohomological support is proper over $T$, it gives the analytic
version of the stack $\mathcal M_{\mathfrak X/U}^{\rm cs}$. 

For a stable $\infty$-category
$\mathcal A$, we write $\mathbf K^{\rm cn}(\mathcal A)$ for its
connective algebraic $K$-theory spectrum and
$K(\mathcal A)=\Omega^\infty\mathbf K^{\rm cn}(\mathcal A)$ for its
$K$-theory space.  Thus, \textit{in this Appendix, $K(\mathcal A)$ always stands for a space, not a group}; its set of
components is the usual Grothendieck group.  

Let us finally also mention that $\pi_*^{\rm alg}$ is the notation for the analytic
derived direct image, and $\pi_!^{KU}$ for the pushforward induced by the $\operatorname{Spin}^c$-orientation.

The result proved below is the following.

\begin{theorem}\label{thm:analytic-comparison}
(a) Let $\pi:\mathfrak X\to U$ be a separated holomorphic submersion between
complex manifolds, and suppose that its underlying topological family is
locally trivial.  There is a natural homotopy making the square
\[
\begin{tikzcd}[column sep=5.6em]
 (\mathcal M_{\mathfrak X/U}^{\rm cs})^{\rm top}
   \arrow[r,"\Gamma_{\mathfrak X/U}"]
   \arrow[d,"\pi_*^{\rm alg}"'] &
 \mathcal C_{\mathfrak X/U}^{\rm cs}
   \arrow[d,"\pi_!^{KU}"] \\
 U\times\operatorname{Perf}(\mathbb C)^{\rm top}
   \arrow[r,"\operatorname{id}_U\times\Gamma_{\operatorname{pt}}"'] &
 U\times KU
\end{tikzcd}
\]
commute.  

(b) The comparison maps respect direct sums, tensor products, and
duals.  The homotopy is compatible with composition, change of base, the
projection formula, and pushforward by a closed holomorphic embedding. 

(c) The homotopy carries a natural ``2-cell'' identifying relative analytic Serre
duality with the topological Serre duality of
Subsection~\ref{compstatement}. 
\end{theorem}

The main instrument is Levy's Riemann--Roch theorem for complex spaces
\cite{Levy1987,Levy2008}. In fact, all the main work was already done by Levy; our remaining task is to explain how his argument for \textit{$K$-theory groups} can be adapted to \textit{$K$-theory spaces}.

Levy's main instruments are certain systems of local resolutions; he then passes from such local resolutions to
complexes of nuclear Fr\'echet spaces and then to Fredholm families.  He
formulates the construction using Quillen's $Q$-construction and the plus
construction and states the conclusion in terms of the higher groups $K_i$.  

We
retain his analytic machinery, but we identify Levy's model with the
Waldhausen $K$-theory space. We then carry his continuous deformation through
Waldhausen's construction $S_\bullet$; this gives the homotopy in the theorem.

Properness is imposed on the support from the beginning, so the argument is
intrinsic to a non-compact $\mathfrak X$, and does not use any choice of a compactification: this is a general feature of Levy's approach.  

To tackle Serre duality, we carry over the machinery of the analytic relative trace of
Ramis--Ruget--Verdier \cite{RRV} to Levy's formalism.  On a Koszul
complex this trace becomes the pairing defined by the complex $KU$ Thom
class, which gives the asserted compatibility with Serre duality.

\subsection{Perfect complexes, connective \texorpdfstring{$K$}{K}-theory,
and the object map}
\label{sec:appendix-object-map}

A perfect
complex on a complex analytic space $Y$ is a complex of $\mathcal O_Y$-modules
which is locally quasi-isomorphic to a bounded complex of finite rank locally
free modules.  (Note that this is a local condition.) 

The stable $\infty$-category
$\operatorname{Perf}_{\rm an}(Y)$ contains the perfect complexes and all
their morphisms, whereas
$\operatorname{Perf}_{\rm an}(Y)^{\simeq}$ retains only objects and
equivalences.  These spaces satisfy descent and form the analytic stack of
perfect complexes; see \cite[Subsection~4.1]{HolsteinPorta}.

For a test manifold $T\to U$, put
$\mathfrak X_T=\mathfrak X\times_U T$ and
\[
 \mathcal A_T=
 \operatorname{Perf}_{\rm an}^{\rm cs}(\mathfrak X_T/T)
 =\bigl\{E\in\operatorname{Perf}_{\rm an}(\mathfrak X_T):
   \operatorname{Supp}H^*(E)\longrightarrow T\text{ is proper}\bigr\}.
\]
This is a stable, idempotent complete subcategory, and
$\mathcal A_T^{\simeq}=\mathcal M_{\mathfrak X/U}^{\rm cs}(T)$.  

Derived pullback preserves this condition, so the
categories $\mathcal A_T$ vary contravariantly with $T$.

We will be concerned with $K$-theory of analytic perfect complexes. For a small stable $\infty$-category $\mathcal A$, its connective
$K$-theory is obtained from Waldhausen's construction $S_\bullet$; we use
the formulation of \cite[Section~7]{BGT}.  The space
$\bigl(S_n\mathcal A\bigr)^{\simeq}$ may be pictured as the space of filtrations
$0=E_0\to E_1\to\cdots\to E_n$, together with their cofibers, and
$K(\mathcal A)=\Omega^\infty\mathbf K^{\rm cn}(\mathcal A)$ is naturally
equivalent to
$\Omega\bigl|\bigl(S_\bullet\mathcal A\bigr)^{\simeq}\bigr|$.  Since
$S_1\mathcal A\simeq\mathcal A$, an object gives a loop and hence a natural
map
\[
 \iota_{\mathcal A}:\mathcal A^{\simeq}\longrightarrow K(\mathcal A),
 \qquad E\longmapsto[0\to E].
\]
A cofiber sequence $E'\to E\to E''$ is a point of
$\bigl(S_2\mathcal A\bigr)^{\simeq}$ and supplies the path expressing
$[E]=[E']+[E'']$.  
The following lemma is standard.

\begin{lemma}\label{lem:appendix-object-map}
For every exact functor $F:\mathcal A\to\mathcal B$, the square
\[
\begin{tikzcd}[column sep=4.5em]
 \mathcal A^{\simeq}\arrow[r,"F"]
   \arrow[d,"\iota_{\mathcal A}"']&
 \mathcal B^{\simeq}\arrow[d,"\iota_{\mathcal B}"]\\
 K(\mathcal A)\arrow[r,"K(F)"']&K(\mathcal B)
\end{tikzcd}
\]
commutes naturally.  The resulting squares are compatible with composition,
natural equivalences between exact functors, and direct sums.
\end{lemma}

The maps $\iota_{\mathcal A_T}$ therefore assemble into a map from
$\mathcal M_{\mathfrak X/U}^{\rm cs}$ to the presheaf
$T\mapsto K(\mathcal A_T)$.  Properness of the support and Grauert's direct
image theorem imply that
$\mathbf R\pi_{T*}E$ is a bounded coherent complex on $T$; it is perfect
because $T$ is smooth.  Thus $\mathbf R\pi_{T*}:\mathcal A_T\to
\operatorname{Perf}_{\rm an}(T)$ is exact, and the lemma applies to it.
In the resulting square, the topological realization of the upper horizontal
map is $\pi_*^{\rm alg}$.  

Levy's construction will thus provide a comparison between the functorialities
of analytic and topological $K$-theory spaces: this is similar in spirit to the passage to Blanc's $K$-theory in Section \ref{sec:blanc-recollection}.  

\begin{remark}
Observe that no detail in this section uses either a compactification (or
an embedding into a smooth compact complex manifold) or the resolution
property.
\end{remark}

\subsection{Levy's construction}
\label{sec:appendix-levy}

Let $f:X\to Y$ be a holomorphic map between compact complex manifolds.
Write $q_X:X\to\operatorname{pt}$ and
$q_Y:Y\to\operatorname{pt}$, so that $q_X=q_Y\circ f$.
For any compact complex manifold $Z$, write
$\operatorname{cyc}_Z:K(\operatorname{Perf}_{\rm an}(Z))\to KU(Z)$ for
the comparison map from analytic to topological $K$-theory constructed
below.  On a bounded complex of locally free $\mathcal O_Z$-modules of
finite rank, it is induced by forgetting the holomorphic structure.  For an
arbitrary perfect complex it is obtained from local finite locally free
resolutions by Levy's construction, which makes it a natural map of
$K$-theory spaces.
The comparison to be constructed is a homotopy
\begin{equation}\label{eq:appendix-fixed-comparison}
 H_f:\operatorname{cyc}_Y\circ K(\mathbf Rf_*)
 \simeq
 f_!^{KU}\circ\operatorname{cyc}_X
 :
 K\bigl(\operatorname{Perf}_{\rm an}(X)\bigr)\longrightarrow KU(Y).
\end{equation}
 We now outline Levy's construction of this homotopy using modern language and formalism.

\subsubsection{The analytic models}

Given a finite diagram of perfect complexes on $X$, choose finite Stein
covers of $X$ and $Y$.  Refine the cover of $X$ on multiple intersections
so that all terms of the diagram admit bounded resolutions by free modules
of finite rank.
Levy's lifting lemmas (\cite[Lemmas 1.6-1.9, Remark 1.10]{Levy2008}) allow the resolutions and the arrows
between them to be chosen simultaneously after a common refinement. 

Levy denotes the resulting category of local resolutions by
$\operatorname{LocFfd}(q_X)$.  

\begin{remark}\label{Levylocsyst}

Let us spell out what ``local'' means here.
If $\mathfrak V=\{V_i\}_{i\in I}$ is the chosen cover, put
$V_\alpha=\bigcap_{i\in\alpha}V_i$ whenever this intersection is nonempty.
An object consists of a finite complex $L^\bullet_\alpha$ of free
$\mathcal O_{V_\alpha}$-modules of finite rank for every such $\alpha$ and,
for $\beta\subset\alpha$, a monomorphic quasi-isomorphism
$L^\bullet_\beta|_{V_\alpha}\to L^\bullet_\alpha$.  These connecting
morphisms compose for $\gamma\subset\beta\subset\alpha$ and are the identity
for $\beta=\alpha$.  A morphism of local systems is a collection of
morphisms of complexes which commutes with the connecting morphisms. See
\cite[Definition~1.2, pp.~574--575]{Levy2008}.
\end{remark}

Write
$\Gamma_f=(f,\operatorname{id}_X):X\hookrightarrow Y\times X$ for the graph
embedding, so that $f=\operatorname{pr}_Y\circ\Gamma_f$.  After refining
both covers once more, for each nonempty intersection $V_\alpha$ choose a
coordinate chart $U_\alpha\subset Y$ containing $f(V_\alpha)$, a polydisc
$D_\alpha\subset\mathbb C^r$ centered at the origin, and holomorphic
functions
$w_\alpha:V_\alpha\to D_\alpha$ such that
\begin{equation}\label{eq:appendix-local-graph-embedding}
 j_\alpha=(f|_{V_\alpha},w_\alpha):V_\alpha
 \lhook\joinrel\longrightarrow U_\alpha\times D_\alpha,
 \qquad
 \operatorname{pr}_{U_\alpha}\circ j_\alpha=f|_{V_\alpha},
\end{equation}
is a closed holomorphic embedding.  In source coordinates one may take
$w_\alpha$ to be the coordinate map; adding zero coordinates gives a common
$r$ for the finite system.  Thus $j_\alpha$ is the expression of
$\Gamma_f$ in the chosen source and target charts.  This is Levy's local
product embedding; see \cite[\S2, p.~583]{Levy2008} and
\cite[proof of Lemma~3.4(d), pp.~181--182]{Levy1987}.

Three further exact categories enter the construction. 

The category
$\operatorname{LocF}(q_X)$ consists of local systems of essentially
holomorphic complexes of nuclear Fr\'echet spaces; the Koszul construction
below takes values in this category.  Here essentially holomorphic means
that locally the complex is the direct sum of a holomorphic complex and an
exact complex which depends smoothly on the parameter.  

The category
$\operatorname{LocFred}(q_X)$ consists of local
systems of continuous complexes of Fr\'echet spaces which are uniformly
Fredholm, that is, uniformly quasi-isomorphic to bounded continuous
complexes of vector bundles of finite rank; their connecting morphisms are
monomorphic uniform quasi-isomorphisms with closed image. 

Finally,
$\operatorname{LocCBun}_0(q_X)$ consists of local systems of bounded
complexes of continuous vector bundles of finite rank, together with a
fixed trivialization of every bundle on the corresponding member of the
cover; see \cite[Definitions~3.3 and~3.4, pp.~593--594]{Levy2008}.

One has approximations 
\begin{equation}\label{eq:appendix-levy-models}
\begin{gathered}
 K\bigl(\operatorname{LocFfd}(q_X)\bigr)
 \simeq K\bigl(\operatorname{Perf}_{\rm an}(X)\bigr),\\
 K\bigl(\operatorname{LocCBun}_0(q_X)\bigr)
 \xrightarrow{\ \sim\ }K\bigl(\operatorname{LocFred}(q_X)\bigr).
\end{gathered}
\end{equation}
The first is analytic descent together with local finite resolutions
\cite[Proposition~2.1]{Levy2008}.  The second is Levy's approximation theorem
\cite[Lemma~3.5]{Levy2008}.

We recall now the local construction that supplies an alternative description (for the equivalence of the two, see equation \ref{CechtoK}) of the natural map $\operatorname{LocFfd}(q_X)\to\operatorname{LocF}(q_X)$. This description is important since it is well-behaved with respect to pushforward operations.

Let $V$
be a bounded Stein coordinate domain, let $z=(z_1,\ldots,z_n)$ be its
coordinates, and let $L^\bullet$ be a local finite resolution, i.e. an object of Levy's category of local systems of resolutions recalled in Remark \ref{Levylocsyst}.

Multiplication by $z_i$ acts on the nuclear Fr\'echet complex
$\Gamma(V,L^\bullet)$.  Levy forms the Koszul complex
\begin{equation}\label{eq:appendix-local-koszul}
 K_V(L^\bullet)(\lambda)=
 K\bigl(M_{z_1}-\lambda_1,\ldots,M_{z_n}-\lambda_n;
        \Gamma(V,L^\bullet)\bigr).
\end{equation}
For $x\in V$, evaluation of sections at $x$ annihilates the operators
$M_{z_i}-z_i(x)$ and hence gives an augmentation from the Koszul complex at
$\lambda=z(x)$ to $L^\bullet_x$.  Sheafifying in $x$ gives a
quasi-isomorphism from the sheaf of holomorphic sections of
$K_V(L^\bullet)$ to the extension by zero of $L^\bullet|_V$; see
\cite[Lemma~2.3, pp.~165--166]{Levy1987} and
\cite[\S3.2, pp.~590--591]{Levy2008}.

On an inclusion of intersections $V_\alpha\subset V_{\alpha'}$, the
restriction map need not commute with the Koszul differentials associated
with the two coordinate systems. Levy constructs higher maps
$r_{m,\alpha',\alpha}$ which correct this failure. Together with the
Koszul differentials and the restriction maps, they define a total complex
$KC^\bullet(\mathfrak V,L^\bullet)$
\cite[Lemma~3.3, pp.~177--179]{Levy1987}; the construction is recalled in
\cite[\S3.2, pp.~591--593]{Levy2008}. On each $V_\alpha$, the map given by evaluation in
Koszul degree zero and by zero in positive Koszul degrees defines a local
evaluation quasi-isomorphism. These maps commute with the corrected total
differential and therefore assemble into a natural quasi-isomorphism
\begin{equation}\label{CechtoK}
 \epsilon:KC^\bullet(\mathfrak V,L^\bullet)
 \longrightarrow \check C^\bullet(\mathfrak V,L^\bullet).
\end{equation}
(Here $\check C^\bullet(\mathfrak V,L^\bullet)$ is the globally defined
complex obtained from $L^\bullet$ by Levy's \v Cech functor.  Its
restriction to each $V_\alpha$ is quasi-isomorphic to
$L^\bullet_\alpha$, so it represents the global perfect complex determined
by the local system.  If $L^\bullet$ is the system obtained by restricting
a globally defined complex $E^\bullet$, this identification is induced by
the usual augmentation
$E^\bullet\to\check C^\bullet(\mathfrak V,L^\bullet)$
\cite[Lemma~1.4, pp.~578--579, and \S3.2, pp.~592--594]{Levy2008}.)

As $L^\bullet$ and its morphisms vary, the corrected Koszul construction
gives Levy's exact functor
$$KC:\operatorname{LocFfd}(q_X)\to\operatorname{LocF}(q_X).$$

After the
holomorphic structure has been forgotten, it takes values in
$\operatorname{LocFred}(q_X)$, and $K(KC)$ in the diagram below is the map
induced by this composite.

\begin{remark}
Let us also make explicit the two equivalences
in \eqref{eq:appendix-levy-models}.  Levy's \v Cech functor
$\check C:\operatorname{LocFfd}(q_X)\to
\operatorname{Perf}_{\rm an}(X)$ sends a local system $L^\bullet$ to the
total complex formed from the extensions by zero of the
$L^\bullet_\alpha$, with the ordinary \v Cech differential induced by its
connecting morphisms.  Its restriction to each $V_\alpha$ is
quasi-isomorphic to $L^\bullet_\alpha$; hence
\cite[Lemma~1.4 and Proposition~2.1]{Levy2008} show that $K(\check C)$ is
an equivalence.  

The other functor is Levy's natural embedding
$\jmath_X:\operatorname{LocCBun}_0(q_X)\to
\operatorname{LocFred}(q_X)$.  Using the chosen local trivializations, it
regards a bounded complex of continuous vector bundles of finite rank as a
continuous family of complexes of finite-dimensional Fr\'echet spaces,
which is uniformly Fredholm. $K(\jmath_X)$ is an equivalence by \cite[Lemma~3.5]{Levy2008}.
\end{remark}
\subsubsection{Some preparation}

Let $\operatorname{LocCBun}(q_X)$ be obtained from
$\operatorname{LocCBun}_0(q_X)$ by omitting the chosen trivializations.
Forgetting the holomorphic structure on a local finite resolution defines
an exact functor
\[
 \alpha_X:\operatorname{LocFfd}(q_X)
 \longrightarrow\operatorname{LocCBun}(q_X).
\]

We record the following very standard proposition in order to show that the use of Waldhausen's $S_{\bullet}$-construction upgrades Levy's comparison map (defined a priori at the level of higher $K$-theory groups) to a map of spectra.

\begin{proposition}\label{prop:appendix-topologization}
There are natural maps of spectra
\[
 \rho_X:K\bigl(\operatorname{LocCBun}(q_X)\bigr)\longrightarrow KU(X),
 \qquad
 \tau_X:K\bigl(\operatorname{LocFred}(q_X)\bigr)\longrightarrow KU(X),
\]
for which the diagram
\[
\begin{tikzcd}[column sep=4.8em]
 K\bigl(\operatorname{LocFfd}(q_X)\bigr)
   \arrow[r,"K(KC)"]
   \arrow[dr,"\rho_X\circ K(\alpha_X)"'] &
 K\bigl(\operatorname{LocFred}(q_X)\bigr)
   \arrow[d,"\tau_X"] \\
 & KU(X)
\end{tikzcd}
\]
commutes up to a natural homotopy.  Under the first equivalence in
\eqref{eq:appendix-levy-models}, the common composite is the promised map
$\operatorname{cyc}_X:K(\operatorname{Perf}_{\rm an}(X))\to KU(X)$.
\end{proposition}

\begin{proof}
 Let $\operatorname{CBun}(X)$ be the Waldhausen
category of bounded complexes of continuous vector bundles of finite rank
on $X$.  The inclusion
\[
 \operatorname{CBun}(X)\longrightarrow\operatorname{LocCBun}(q_X)
\]
induces an equivalence on $K$-theory: Levy proves that every local system is
quasi-isomorphic to a globally defined bundle complex
\cite[\S3.1, pp.~588--590, especially Lemma~3.2]{Levy2008}.

Levy writes the resulting map on homotopy groups as
\begin{equation}\label{eq:appendix-Q-plus-topologization}
 K_i\bigl(\operatorname{CBun}(X)\bigr)
 \longrightarrow K_i^Q(X)\longrightarrow K^i(X).
\end{equation}
Here $K_i^Q(X)=\pi_{i+1}BQ\operatorname{Vect}(X)$ is Quillen's
$K$-group of the exact category of continuous vector bundles.  The first
arrow in \eqref{eq:appendix-Q-plus-topologization} is the
Gillet--Waldhausen comparison.  For the second, the comparison of the
$Q$- and plus constructions identifies its source with the plus
construction on $GL(C(X))$ with the discrete topology; the identity
homomorphism to the same matrix group with the compact-open topology gives
the arrow to $K^i(X)$.  This is Levy's construction in
\cite[\S3, pp.~587--588, especially Lemma~3.1]{Levy2008}.  Applying the
construction at every stage of Waldhausen's $S_\bullet$ construction gives
the map of spectra $\rho_X$.

It remains to compare the bundle and Fredholm descriptions.  Levy's
Koszul construction gives the exact functor
$KC:\operatorname{LocFfd}(q_X)\to\operatorname{LocFred}(q_X)$ recalled
above.  Define $\tau_X$ by transporting $\rho_X$ through the equivalences
\[
 K\bigl(\operatorname{LocFred}(q_X)\bigr)
 \xleftarrow{\ \sim\ }
 K\bigl(\operatorname{LocCBun}_0(q_X)\bigr)
 \xrightarrow{\ \sim\ }
 K\bigl(\operatorname{LocCBun}(q_X)\bigr).
\]
The first is \cite[Lemma~3.5, pp.~594--595]{Levy2008}; the second forgets
the chosen local trivializations.  Levy's comparison of his two
constructions \cite[Proposition~3.6, pp.~594--595]{Levy2008} gives the
homotopy in the diagram.  Composing with
$K(\operatorname{LocFfd}(q_X))\simeq
K(\operatorname{Perf}_{\rm an}(X))$ gives $\operatorname{cyc}_X$.
\end{proof}

For a finite compact polyhedron $P$, write
$\operatorname{LocFred}(q_Y;P)$ for the exact category of continuous
$P$-families in $\operatorname{LocFred}(q_Y)$.  Levy's approximation
theorem and the preceding topologization apply to this category after
replacing $Y$ by $Y\times P$.

The following is clear.

\begin{lemma}\label{lem:appendix-interval}
Let $\mathcal E$ be an exact category, and let $P$ be a finite compact
polyhedron.  An exact functor
$\mathcal E\to\operatorname{LocFred}(q_Y;P)$ to Levy's
Fredholm model over $Y\times P$ naturally determines a map
\begin{equation}\label{eq:appendix-parameter-map}
 K(\mathcal E)\longrightarrow\operatorname{Map}\bigl(P_+,KU(Y)\bigr).
\end{equation}
For $P=[0,1]$, its endpoint restrictions form a homotopy between the maps
induced by the two exact functors obtained by restricting the family to
$t=0$ and $t=1$.
\end{lemma}

\subsection{The comparison between pushforwards}

We now give the local map from which our main comparison is obtained.  For each
nonempty intersection $V_\alpha$, let $L_\alpha^\bullet$ be the chosen local
resolution, and put $B_\alpha=U_\alpha\times D_\alpha$.  Write
$\lambda=(\lambda_1,\ldots,\lambda_n)$ for the coordinates on $U_\alpha$
and $\mu=(\mu_1,\ldots,\mu_r)$ for those on $D_\alpha$, and put
$z_\alpha=\lambda\circ f|_{V_\alpha}$.  Thus, in these coordinates,
$j_\alpha(x)=(z_\alpha(x),w_\alpha(x))$.  The symbols $\lambda$ and $\mu$
are variables on $B_\alpha$, whereas $z_\alpha$ and $w_\alpha$ are
holomorphic functions on $V_\alpha$.  For $I=[0,1]$, define
\begin{equation}\label{eq:appendix-local-scaling-map}
 p_\alpha:V_\alpha\times I\longrightarrow B_\alpha\times I,
 \qquad (x,t)\longmapsto\bigl(z_\alpha(x),t w_\alpha(x),t\bigr),
\end{equation}
and write $p_{\alpha,t}$ for its restriction at $t$.

Let $M_{z_{\alpha,i}}$ and $M_{w_{\alpha,j}}$ denote multiplication on
$\Gamma(V_\alpha,L_\alpha^\bullet)$.  

Denote Levy's local Koszul sheaf on
$B_\alpha\times I$ by $\mathcal K^{\rm loc}_{\alpha,I}(L^\bullet)$ and its
restriction at $t$ by $\mathcal K^{\rm loc}_{\alpha,t}(L^\bullet)$.  It is
holomorphic in $(\lambda,\mu)$ and continuous in $t$; for an open set
$U\subset B_\alpha$, its restriction at $t$ has sections
\begin{equation}\label{eq:appendix-levy-interval}
 \mathcal K^{\rm loc}_{\alpha,t}(L^\bullet)(U)=
 K\!\left(M_{z_\alpha}-\lambda,tM_{w_\alpha}-\mu;
 \mathcal O(U)\widehat\otimes
 \Gamma(V_\alpha,L_\alpha^\bullet)\right).
\end{equation}
Let $L_{\alpha,I}^\bullet$ be the pullback of $L_\alpha^\bullet$ to
$V_\alpha\times I$, regarded as a family continuous in $t$.  There is a
morphism of complexes on
$B_\alpha\times I$
\begin{equation}\label{eq:appendix-local-evaluation}
 \epsilon_\alpha:\mathcal K^{\rm loc}_{\alpha,I}(L^\bullet)
 \longrightarrow (p_\alpha)_*L_{\alpha,I}^\bullet.
\end{equation}
In Koszul degree zero it is given by restriction along $p_\alpha$: a local section
$g\otimes s$ is sent to
$$g(z_\alpha(x),t w_\alpha(x),t)s(x).$$  It is zero in nonzero Koszul
degrees.  This is a chain map because the functions
$z_{\alpha,i}-\lambda_i$ and $t w_{\alpha,j}-\mu_j$ vanish along
$p_\alpha$.  The map $\epsilon_\alpha$ is a quasi-isomorphism of sheaf complexes by
\cite[Lemma~2.3$'$, p.~166, and proof of Lemma~3.4(d), p.~182]{Levy1987}.

It remains to put these local constructions together.  Choose $N$
sufficiently large and an almost complex embedding
$e=(f,\rho):X\hookrightarrow Y\times\mathbb C^N$, with the complex
structure on its normal bundle determined by $TX$ and $TY$.
Let $i:Y\hookrightarrow Y\times\mathbb C^N$ be the zero section.
Choose a closed ball $D\subset\mathbb C^N$ large enough that Levy's
family is uniformly exact near $Y\times\partial D$, and put
$A=Y\times D$, $B=Y\times\partial D$.

Levy first regards the local complexes as complexes on this same space,
using the changes of coordinates in
\cite[proof of Lemma~3.4(d), pp.~182--183]{Levy1987}, which depend
continuously on $t$.
His maps $r_{m,\alpha',\alpha,t}$ then add the terms required to make the
total differential square to zero.  They are constructed in
\cite[Lemma~3.3, pp.~177--179]{Levy1987}; their extension over $I$ is in
\cite[proof of Lemma~3.4(d), pp.~182--184]{Levy1987}, and is recalled in
\cite[\S4.2, p.~597]{Levy2008}.  Denote the resulting total complex by
$\mathcal K_t(L^\bullet)$.
The changes of coordinates are necessary here, since the sheaves
$(p_{\alpha,t})_*L_\alpha^\bullet$ are defined on different spaces
$B_\alpha$ and cannot be compared merely by restricting on the source.

Levy proves that this family is uniformly Fredholm by restricting the
target charts and refining the covers
\cite[proof of Lemma~3.4(d), pp.~184--185]{Levy1987}.
The construction is functorial in $L^\bullet$
\cite[\S4.2, p.~597]{Levy2008}: the correcting matrices are chosen
for the structure sheaf and then used for every local system
\cite[\S3.2, pp.~592--593]{Levy2008}.
Since the cover is finite, totalization preserves exact sequences.
We therefore obtain an exact functor
\begin{equation}\label{eq:appendix-total-evaluation}
 \mathcal H_f:\operatorname{LocFfd}(q_X)
 \longrightarrow\operatorname{LocFred}(A\times I,B\times I).
\end{equation}
The target consists of local systems of continuous uniformly Fredholm
complexes, uniformly exact on $B\times I$.
For the source, we use local resolutions on a neighbourhood of the graph
of $f$, as in \cite[\S4.1]{Levy2008}.  Their $K$-theory is identified
with that of $\operatorname{Perf}_{\rm an}(X)$ by the first equivalence
in \eqref{eq:appendix-levy-models}.

We now identify the complexes at $t=0$ and $t=1$.

At $t=0$, the auxiliary coordinate operators vanish.  Levy identifies
the resulting complex, by a natural uniform quasi-isomorphism, with the
tensor product of his Koszul complex for $\mathbf Rf_*L^\bullet$, pulled
back from $Y$, and $K(-\mu_1,\ldots,-\mu_N)$.
The second factor represents the Thom class of the zero section $i$.
One computes the direct image  by the relative \v Cech complex of
\cite[\S4.1]{Levy2008}.  To compare it with a local finite resolution,
one includes the resolution and its quasi-isomorphism to this complex
among the data; forgetting this choice induces an equivalence on
$K$-theory by \cite[Proposition~2.1]{Levy2008}.

At $t=1$, the local maps satisfy $p_{\alpha,1}=j_\alpha$.
The resulting construction gives the pushforward $e_!$: one tensors
with the Thom class of the complex normal bundle of $e$ and then extends
the support from a tubular neighbourhood to $(A,B)$.
These descriptions of the two ends of the family are proved in
\cite[proof of Lemma~3.4(d), pp.~183--185]{Levy1987} and
\cite[\S4.2, p.~597]{Levy2008}.
Proposition~3.6 of \cite{Levy2008} identifies the Koszul construction
with the cycle map defined using local finite resolutions.

We may now apply approximation by complexes of finite rank vector bundles and the
passage to topological $K$-theory to the compact pair $(A,B)$
\cite[Lemmas~3.1 and~3.5]{Levy2008}.
Write $T_N=i_!:KU(Y)\xrightarrow{\sim}KU(A,B)$ for the Thom map,
which is an equivalence by Bott periodicity.
The maps at $t=0$ and $t=1$ are therefore
$T_N\circ\operatorname{cyc}_Y\circ K(\mathbf Rf_*)$ and
$e_!\circ\operatorname{cyc}_X$, respectively, after composing with
the equivalence from the category of local resolutions.
Composing with the inverse of $T_N$ gives the two maps in
\eqref{eq:appendix-fixed-comparison}, since
$f_!^{KU}=T_N^{-1}\circ e_!$.

\begin{proposition}\label{prop:appendix-compact-levy}
Waldhausen's $S_\bullet$-construction applied to this family, followed
by $T_N^{-1}$, determines a homotopy between the two maps
\[
 K\bigl(\operatorname{LocFfd}(q_X)\bigr)\longrightarrow KU(Y).
\]
\end{proposition}

Indeed, the relative version of Lemma~\ref{lem:appendix-interval} follows
from $KU(A\times I,B\times I)\simeq\operatorname{Map}_*(I_+,KU(A,B))$.
Using the first equivalence in \eqref{eq:appendix-levy-models}, we obtain
the homotopy \eqref{eq:appendix-fixed-comparison}.

\subsubsection{Independence of the auxiliary choices}

The construction uses a Stein cover, local resolutions, the embeddings
$j_\alpha$ in \eqref{eq:appendix-local-graph-embedding}, an ordering of
their auxiliary coordinates, and correcting maps.  

We now establish \textit{canonicity}: 
the homotopy obtained above is
independent of these choices. 

Let $\mathcal D$ be any finite collection of perfect complexes, morphisms,
and filtrations together with their quotients, and let $P$ be a finite cube
of parameters.  After refining the Stein covers, Levy's resolution lemmas
give local resolutions for the whole collection at once.  His product
construction has the following concrete form.  Suppose that the finite
diagram involves embeddings
$j_\alpha^{(a)}=(f|_{V_\alpha},w_\alpha^{(a)})$, $1\leq a\leq N$.
After a common refinement, they are dominated by the single embedding
\begin{equation}\label{eq:appendix-product-embedding}
 j_\alpha^{\rm prod}=
 (f|_{V_\alpha},w_\alpha^{(1)},\ldots,w_\alpha^{(N)}):V_\alpha
 \lhook\joinrel\longrightarrow
 U_\alpha\times D_\alpha^{(1)}\times\cdots\times D_\alpha^{(N)}.
\end{equation}
The factors $D_\alpha^{(a)}$ give the coordinate blocks, ordered once and
for all.  This is the construction in
\cite[proof of Lemma~3.2(b), pp.~174--175]{Levy1987}.
The Koszul construction and the correcting maps may then be formed
simultaneously over $P$: the formulas of Lemma~3.3 depend algebraically on
the commuting coordinate operators and hence apply to a finite cube of
continuous parameters; see \cite[Remark~(2), p.~179]{Levy1987}.  The
Koszul contraction of \cite[Lemma~2.3$'$, p.~166]{Levy1987} is uniform on
the compact cube $P$.  

We call
the resulting global \v Cech complexes, together with their Koszul
augmentation quasi-isomorphisms to the \v Cech complexes of Fr\'echet
sections representing the given perfect complexes, a \textit{Levy model for
$(\mathcal D,P)$.}

The following lemma is then standard and finishes the proof of canonicity.

(If $Q$ is a union of faces of $P$ and a Levy model has been chosen over
$Q$, by an extension we mean a Levy model over $P$ together with an
equivalence between its restriction to $Q$ and the chosen model which is
compatible with these augmentations.)

\begin{lemma}\label{lem:appendix-levy-canonicity}
After uniform quasi-isomorphisms have been inverted, the
space of Levy models for $(\mathcal D,P)$, with all comparisons required
to commute with the augmentations, is nonempty and contractible.  If a model has
already been chosen on some faces of $P$, or for part of $\mathcal D$,
the space of extensions is again nonempty and contractible.
\end{lemma}

\subsection{Proper supports}
\label{sec:appendix-proper-support}

We now discuss the compactly supported analogue of the above constructions. First of all, it is easy to see that the above proofs generalize to the following context.

\begin{proposition}\label{prop:appendix-proper-support}
Let $f:X\to Y$ be a holomorphic map of complex manifolds, and let
$Z\subset X$ be a closed analytic subset such that $f|_Z:Z\to Y$ is
proper.  Derived direct image defines an exact functor
$\mathbf Rf_*:\operatorname{Perf}_Z(X)\to
\operatorname{Perf}_{\rm an}(Y)$, and there is a canonical homotopy making
the square
\[
\begin{tikzcd}[column sep=5.4em]
 K\bigl(\operatorname{Perf}_Z(X)\bigr)
   \arrow[r,"K(\mathbf Rf_*)"]
   \arrow[d,"\operatorname{cyc}_{X,Z}"']&
 K\bigl(\operatorname{Perf}_{\rm an}(Y)\bigr)
   \arrow[d,"\operatorname{cyc}_Y"]\\
 KU_Z(X)\arrow[r,"f_!^{KU}"']&KU(Y)
\end{tikzcd}
\]
commute.  If $Z\subset Z'$ and $Z'$ is also proper over $Y$, the
restriction of the homotopy for $Z'$ to
$K(\operatorname{Perf}_Z(X))$ is canonically identified with the homotopy
for $Z$, with all higher compatibilities.
\end{proposition}

Now let $\operatorname{Perf}_{\rm pr/Y}(X)$ be the full stable subcategory of
perfect complexes whose cohomological support is proper over $Y$.  Every
finite diagram in this category is contained in
$\operatorname{Perf}_Z(X)$ for one closed analytic $Z$ proper over $Y$.
Connective $K$-theory preserves this filtered union, as explained in
Subsection~\ref{sec:appendix-object-map}.  

Write $\operatorname{cyc}_X^{\rm pr/Y}$ for the map
$K(\operatorname{Perf}_{\rm pr/Y}(X))\simeq
\operatorname*{colim}_Z K(\operatorname{Perf}_Z(X))\to
\operatorname*{colim}_Z KU_Z(X)$, and write $f_!^{KU}$ for the map from
the latter colimit to $KU(Y)$.  Proposition~\ref{prop:appendix-proper-support}
gives a canonical homotopy
\begin{equation}\label{eq:appendix-proper-support-comparison}
 \operatorname{cyc}_Y\circ K(\mathbf Rf_*)
 \simeq f_!^{KU}\circ\operatorname{cyc}_X^{\rm pr/Y}:
 K\bigl(\operatorname{Perf}_{\rm pr/Y}(X)\bigr)\longrightarrow KU(Y).
\end{equation}
Here both colimits range over the closed analytic subsets $Z\subset X$
proper over $Y$.  

\subsection{Compatibility with Serre duality}
\label{sec:appendix-serre-duality}

We now trace the Serre duality throughout the comparison homotopy. 

Let
$q:X\to S$ be a separated holomorphic submersion between complex manifolds,
of relative complex dimension $d$, and let $Z\subset X$ be a closed analytic subset proper
over $S$.  We use the notation of
Subsection~\ref{sec:serre-duality-without-quasi-projectivity}:
\[
\begin{gathered}
 \omega_q=\det(T_{X/S})^\vee,\\
 D_q(E)=\mathbf R\mathcal Hom_X(E,\omega_q[d]),\qquad
 D_S(F)=\mathbf R\mathcal Hom_S(F,\mathcal O_S).
\end{gathered}
\]
Thus $D_q:\operatorname{Perf}_Z(X)^{\rm op}\to
\operatorname{Perf}_Z(X)$ and
$D_S:\operatorname{Perf}_{\rm an}(S)^{\rm op}\to
\operatorname{Perf}_{\rm an}(S)$ are exact dualities.  Write $\tau$ for
complex dualization in $KU$, and put
\begin{equation}\label{eq:appendix-topological-relative-duality}
 S_q(x)=(-1)^d[\omega_q]\tau(x).
\end{equation}
Here $(-1)^d$ is additive inversion when $d$ is odd; it is the image of
the shift $[d]$ under additivity in $K$-theory.  This is the operation
denoted $S_X$ in Subsection~\ref{sec:serre-duality-without-quasi-projectivity}
when $S=\operatorname{pt}$.

Put $\mathcal A_Z=\operatorname{Perf}_Z(X)$, and denote the opposition
equivalences by
\[
 o_Z:K(\mathcal A_Z^{\rm op})\xrightarrow{\sim}K(\mathcal A_Z),
 \qquad
 o_S:K(\operatorname{Perf}_{\rm an}(S)^{\rm op})
 \xrightarrow{\sim}K(\operatorname{Perf}_{\rm an}(S)).
\]
Duals, tensor products, and shifts of finite locally free resolutions
commute with forgetting holomorphic structure.  Thus there is a
natural homotopy
\begin{equation}\label{eq:appendix-realization-duality}
 \delta_{q,Z}:\operatorname{cyc}_{X,Z}K(D_q)
 \simeq S_q\operatorname{cyc}_{X,Z}o_Z:
 K(\mathcal A_Z^{\rm op})\longrightarrow KU_Z(X).
\end{equation}
For the identity of $S$ this gives
$\delta_S:\operatorname{cyc}_S K(D_S)\simeq
\tau\operatorname{cyc}_S o_S$.  Below we use the naturality identification
$o_S K((\mathbf Rq_*)^{\rm op})\simeq K(\mathbf Rq_*)o_Z$.

\subsubsection{The square to be filled}

The analytic Serre equivalence and its topological counterpart will be
constructed below.  Their sources and targets are
\begin{equation}\label{eq:appendix-two-serre-maps}
 s_q^{\rm an}:\mathbf Rq_*D_q(E)\xrightarrow{\sim}
 D_S(\mathbf Rq_*E),
 \qquad
 s_q^{KU}:q_!^{KU}S_q(x)\xrightarrow{\sim}\tau q_!^{KU}(x).
\end{equation}
Let $H_{q,Z}$ denote the homotopy of
Proposition~\ref{prop:appendix-proper-support}.  The assertion to be proved
is a canonical filling of
\begin{equation}\label{eq:appendix-serre-square}
\begin{tikzcd}[column sep=2em,row sep=4.0em]
 \operatorname{cyc}_S K(\mathbf Rq_*D_q)
   \arrow[r,"H_{q,Z}D_q"]
   \arrow[d,"\operatorname{cyc}_S(s_q^{\rm an})"']&
 q_!^{KU}\operatorname{cyc}_{X,Z}K(D_q)
   \arrow[r,"q_!^{KU}\delta_{q,Z}"]
   \arrow[d,phantom,"\scriptstyle\Downarrow\,\Xi_{q,Z}" description]&
 q_!^{KU}S_q\operatorname{cyc}_{X,Z}o_Z
   \arrow[d,"s_q^{KU}\operatorname{cyc}_{X,Z}o_Z"]\\
 \operatorname{cyc}_S K\bigl(D_S\circ(\mathbf Rq_*)^{\rm op}\bigr)
   \arrow[r,"\delta_S(\mathbf Rq_*)^{\rm op}"']&
 \tau\operatorname{cyc}_S K(\mathbf Rq_*)o_Z
   \arrow[r,"\tau H_{q,Z}o_Z"']&
 \tau q_!^{KU}\operatorname{cyc}_{X,Z}o_Z .
\end{tikzcd}
\end{equation}

We now recall some ingredients.

\subsubsection{The analytic trace with supports}

We start with the first ingredient.

The relative residue trace of Ramis--Ruget
\cite[Th\'eor\`eme~1, pp.~91--92, Remarque~3, p.~93, and \S6]{RamisRuget1974}
gives a canonical map $\operatorname{Tr}_q:\mathbf Rq_!\omega_q[d]\to\mathcal O_S$,
and hence
\[
 \begin{aligned}
 \operatorname{Tr}_{q,Z}:\mathbf Rq_*\mathbf R\Gamma_Z\omega_q[d]
 &\simeq\mathbf Rq_!\mathbf R\Gamma_Z\omega_q[d]\\
 &\longrightarrow\mathbf Rq_!\omega_q[d]
 \xrightarrow{\operatorname{Tr}_q}\mathcal O_S;
 \end{aligned}
\]
for a proper map, this is the trace of Ramis--Ruget--Verdier
\cite{RamisRuget1974,RRV}.  Cup product, evaluation, and this trace give
the pairing
\begin{equation}\label{eq:appendix-supported-serre-pairing}
 \mathbf Rq_*D_q(E)\otimes^{\mathbf L}\mathbf Rq_*E
 \longrightarrow \mathbf Rq_*\bigl(D_q(E)\otimes^{\mathbf L}E\bigr)
 \longrightarrow \mathbf Rq_*\mathbf R\Gamma_Z\omega_q[d]
 \xrightarrow{\operatorname{Tr}_{q,Z}}\mathcal O_S.
\end{equation}

The following follows from proper analytic duality
\cite[\S\S1--2]{RRV} and 
\cite[Th\'eor\`eme~20, pp.~121--123]{RamisRuget1974}.

\begin{proposition}\label{prop:appendix-supported-serre}
For $E\in\operatorname{Perf}_Z(X)$, the adjoint of
\eqref{eq:appendix-supported-serre-pairing} is the equivalence
$s_{q,E}^{\rm an}$ in \eqref{eq:appendix-two-serre-maps}.  It is
independent of the analytic support chosen to contain
$\operatorname{Supp}\mathcal H^*(E)$ and is compatible with enlargement
of that support.
\end{proposition}

Fix $s\in S$, and choose Stein coordinate balls
$s\in V_0\Subset V_1\subset S$.  Thus $V_0$ is an open subset of the
base.  Put $K=Z\cap q^{-1}(\overline V_0)$.  Choose a finite Stein
system in $q^{-1}(V_1)$ whose inner members cover $K$.  We write
$V_\alpha$ for its intersections after restriction to $q^{-1}(V_0)$,
choose local finite resolutions $L_\alpha$ there, and put
$L_\alpha^!=\mathcal Hom(L_\alpha,\omega_q[d])$.  The \v Cech
complex $C_{V_0}^+(L)$, formed as the total of the derived sections
$\mathbf R\Gamma(V_\alpha,E)$ computed locally using $L_\alpha$, represents
$\mathbf R\Gamma(V_0,\mathbf Rq_*E)$.

For each intersection $V_\alpha$, let 
$\mathcal A_c(V_\alpha/V_0,L_\alpha^!)$ be the relative Cousin complex of
Ramis--Ruget \cite[\S\S4--6]{RamisRuget1974}, with supports proper over $V_0$.
Consider also the \textit{Cech--Cousin complex} $C_{V_0}^-(L^!)$ representing
$\mathbf R\Gamma(V_0,\mathbf Rq_!D_q(E))$. 

\begin{lemma}
\label{lem:appendix-cech-cousin-trace}
There are natural
$\mathcal O(V_0)$-linear pairings in the derived category
\[
 C_{V_0}^-(L^!)\otimes^{\mathbf L}_{\mathcal O(V_0)}C_{V_0}^+(L)
 \longrightarrow\mathcal O(V_0).
\]
Its adjoint
\begin{equation}\label{eq:appendix-cech-cousin-pairing}
 \rho_{V_0,E}:C_{V_0}^-(L^!)\longrightarrow
 \mathbf R\operatorname{Hom}_{\mathcal O(V_0)}
 \bigl(C_{V_0}^+(L),\mathcal O(V_0)\bigr)
\end{equation}
is a quasi-isomorphism and represents
$s_{q,E}^{\rm an}$.  After shrinking $V_0$, it is represented on bounded
finite free $\mathcal O(V_0)$-models by a continuous chain homotopy equivalence.
\end{lemma}

\begin{proof}
On a relative coordinate chart $p:U\to V_0$, an element 
$\eta\in\mathcal A_c(U/V_0,L^!)$ and a holomorphic section
$s\in\Gamma(U,L)$ are paired by
\begin{equation}\label{eq:appendix-local-cousin-pairing}
 \langle\eta,s\rangle_U=
 \operatorname{res}_{U/V_0}\bigl(\operatorname{ev}(\eta,s)\bigr).
\end{equation}
Here evaluation contracts the $L^!$ and $L$ factors and produces a class
in the relative Cousin complex for $\omega_q[d]$; in relative
coordinates $z_1,\ldots,z_d$, the normalization is
\[
 \operatorname{res}_0
 \frac{dz_1\wedge\cdots\wedge dz_d}{z_1\cdots z_d}=1.
\]
Thus \eqref{eq:appendix-local-cousin-pairing} is the local chain form of
the trace $\operatorname{Tr}_{q,Z}$ in
\eqref{eq:appendix-supported-serre-pairing}.

Suppose that $V_\beta\subset V_\alpha$, and let
$e_{\beta,\alpha}$ denote extension by zero of analytic functionals.
Cauchy's theorem gives
\[
 \langle e_{\beta,\alpha}\eta,s\rangle_{V_\alpha}
 =\langle\eta,s|_{V_\beta}\rangle_{V_\beta}.
\]
Hence the extension by zero maps in $C_{V_0}^-(L^!)$ are adjoint to the
restriction maps in $C_{V_0}^+(L)$.  Taking the \v Cech totals with the usual
graded signs gives the asserted derived pairing.
The formula for change of variables for the Grothendieck residue shows
that this pairing is independent of the relative coordinates.

Under the identifications of the two total complexes above, the resulting pairing
is the derived global sections of \eqref{eq:appendix-supported-serre-pairing}.
Its adjoint is therefore $s_{q,E}^{\rm an}$ by
Proposition~\ref{prop:appendix-supported-serre}.
After shrinking $V_0$, choose bounded finite free models
$P\simeq C_{V_0}^+(L)$ and $Q\simeq C_{V_0}^-(L^!)$.
Cartan's theorem~B identifies the derived dual with
$\operatorname{Hom}_{\mathcal O(V_0)}(P,\mathcal O(V_0))$; compare
\cite[Lemmes~21--22, pp.~121--122]{RamisRuget1974}.
The adjoint is represented by a quasi-isomorphism $Q\to P^\vee$,
hence a chain homotopy equivalence.  Its maps and homotopies are holomorphic
matrices, so are continuous on sections.
\end{proof}

\subsubsection{The pairing along Levy's interval}

The preceding subsection gave a chain-level representative of the
analytic Serre pairing.  We now consider the local Koszul pairing
used in the interval which defines $H_{q,Z}$.

Locally, the value at $t=0$ will recover evaluation followed by the
relative residue, while the value at $t=1$ will be the graph Thom pairing.

  Consider a relative
coordinate polydisc $U\subset q^{-1}(V_0)$, write
$a=q|_U:U\to V_0$, and let
$w=(w_1,\ldots,w_r):U\to\mathbb C^r$ be the auxiliary map used in
Levy's local embedding.  Put
\[
 Y_U=U\times\mathbb C_\mu^r,
 \qquad j_t:U\longrightarrow Y_U,
 \qquad j_t(x)=(x,tw(x)).
\]
Then $a=(a\circ\operatorname{pr}_U)\circ j_t$.  

The relative residue of
the preceding subsection is taken in relative coordinates on $U/V_0$.
Put $b=a\circ\operatorname{pr}_U$.  

Let $F=\bigoplus_{i=1}^r\mathcal O_{Y_U}e_i$, let $N=F^\vee$, and set
\[
 f_{t,i}=tw_i-\mu_i,
 \qquad K_t^{-j}=\Lambda^jF,
 \qquad d_t=\iota_{f_t}.
\]
Thus $K_t$ is the Koszul resolution of $j_{t*}\mathcal O_U$.  The normal bundle of
$j_t$ is $j_t^*N$, with conormal identification $e_i\mapsto df_{t,i}$,
where the differentials are relative over $V_0$.

The following is now a direct calculation.

\begin{proposition}
\label{prop:appendix-scaled-koszul-pairing}
One has a chain isomorphism and a trace map
that vary continuously for $0\leq t\leq1$:
\begin{equation}\label{eq:appendix-koszul-self-duality}
 \kappa_t:K_t^\vee\xrightarrow{\sim}
 (K_t\otimes\det(N))[-r],
 \qquad
 \epsilon_t:(K_t\otimes\det(N))[-r]\longrightarrow\mathcal O_{Y_U}.
\end{equation}
If $L$ is a finite locally free complex on $U$ and
$L^!=\mathcal Hom(L,\omega_{U/V_0}[d])$, there is a natural chain isomorphism
\begin{equation}\label{eq:appendix-scaled-local-duality}
 K_t\otimes\operatorname{pr}_U^*L^!
 \xrightarrow{\sim}
 \mathcal Hom_{Y_U}\!\left(
 K_t\otimes\operatorname{pr}_U^*L,
 \omega_{Y_U/V_0}[d+r]\right).
\end{equation}
At $t=0$ the auxiliary residue reduces to evaluation on $U$, giving
\eqref{eq:appendix-local-cousin-pairing} after the relative residue on $U/V_0$;
at $t=1$ it is the Koszul duality pairing for the Thom class of $j_1$.
\end{proposition}

We leave the check that these local compatibilities glue via Levy's machinery to an interested reader. This finishes the proof of the compatibility with Serre duality.


\begin{thebibliography}{39}
\addcontentsline{toc}{section}{References}

\bibitem{AdamsExceptional}
J.F. Adams, {\it Lectures on Exceptional Lie Groups}, Chicago Lectures in
Mathematics, University of Chicago Press, Chicago, IL, 1996. Edited by
Z.~Mahmud and M.~Mimura, with a foreword by J.P. May.

\bibitem{ABS}
M.F. Atiyah, R. Bott, and A. Shapiro, {\it Clifford modules}, Topology
3, Supplement 1 (1964), 3--38.

\bibitem{BaeKoolPark}
  Y.~Bae, M.~Kool, and H.~Park, {\it
Counting surfaces on {Calabi--Yau} 4-folds {I}: Foundations},
Geometry \& Topology, \textit{to appear}.

\bibitem{BFM}
P. Baum, W. Fulton, and R. MacPherson, {\it Riemann--Roch and topological
$K$-theory for singular varieties}, Acta Math. 143 (1979), no. 3--4,
155--192.

\bibitem{Blanc}
A. Blanc, {\it Topological $K$-theory of complex noncommutative spaces},
Compositio Math. 152 (2016), no. 3, 489--555.

\bibitem{BRTV}
A. Blanc, M. Robalo, B. To\"en, and G. Vezzosi, {\it Motivic realizations
of singularity categories and vanishing cycles}, J. \'Ec. polytech.
Math. 5 (2018), 651--747.

\bibitem{BGT}
A.J. Blumberg, D. Gepner, and G. Tabuada, {\it A universal characterization
of higher algebraic $K$-theory}, Geom. Topol. 17 (2013), no. 2, 733--838.

\bibitem{Bojko2021}
A. Bojko, {\it Orientations for DT invariants on quasi-projective
Calabi--Yau $4$-folds}, Adv. Math. 388 (2021), no. 107859.

\bibitem{BorisovJoyce2017}
D. Borisov and D. Joyce, {\it Virtual fundamental classes for moduli
spaces of sheaves on Calabi--Yau four-folds}, Geom. Topol. 21 (2017),
no. 6, 3231--3311.

\bibitem{BottStable}
R. Bott, {\it The stable homotopy of the classical groups}, Ann. of Math.
(2) 70 (1959), no. 2, 313--337.

\bibitem{CHN}
B. Calm\`es, Y. Harpaz, and D. Nardin, {\it A motivic spectrum representing
hermitian $K$-theory}, 2024.

\bibitem{CaoGrossJoyce2020}
Y. Cao, J. Gross, and D. Joyce, {\it Orientability of moduli spaces of
$\operatorname{Spin}(7)$-instantons and coherent sheaves on Calabi--Yau
$4$-folds}, Adv. Math. 368 (2020), no. 107134.

\bibitem{Cao1}
Y. Cao, M. Kool, S. Monavari, {\it Stable pair invariants of local Calabi-Yau $4$-folds}, 
\href{https://arxiv.org/pdf/2004.09355}{arXiv:2004.09355},
2020.

\bibitem{Cao2}
Y. Cao, G. Zhao, Z. Zhou, {\it A degeneration formula for Donaldson-Thomas theory on local Calabi-Yau $4$-folds}, 
\href{https://arxiv.org/pdf/2402.16103}{arXiv:2402.16103},
2024.

\bibitem{CMT2022}
Y. Cao, D. Maulik, and Y. Toda, {\it Stable pairs and Gopakumar--Vafa
type invariants for Calabi--Yau $4$-folds}, J. Eur. Math. Soc. 24
(2022), no. 2, 527--581.

\bibitem{Carvalho}
C.C. Carvalho, {\it Cobordism invariance of the family index}, Math.
Nachr. 285 (2012), no. 7, 808--820.

\bibitem{CF}
P.E. Conner and E.E. Floyd, {\it The Relation of Cobordism to
$K$-Theories}, Lecture Notes in Math. 28, Springer-Verlag,
Berlin--Heidelberg, 1966.

\bibitem{Dyer}
E. Dyer, {\it Cohomology theories}, Mathematics Lecture Note Series 27,
W.A. Benjamin, New York--Amsterdam, 1969.

\bibitem{FultonHarris}
W. Fulton and J. Harris, {\it Representation Theory: A First Course},
Graduate Texts in Math. 129, Springer-Verlag, New York, 1991.

\bibitem{GMTW}
S. Galatius, I. Madsen, U. Tillmann, and M. Weiss, {\it The homotopy type
of the cobordism category}, Acta Math. 202 (2009), no. 2, 195--239.

\bibitem{Genauer}
J. Genauer, {\it Cobordism categories of manifolds with corners}, Trans.
A.M.S. 364 (2012), no. 1, 519--550.

\bibitem{HK}
Z. Halladay and Y. Kamel, {\it Real spin bordism and orientations of
topological $K$-theory}, arXiv e-prints, 2024.

\bibitem{HolsteinPorta}
J. Holstein and M. Porta, {\it Analytification of mapping stacks},
Algebraic Geometry 12 (2025), no. 1, 13--83.

\bibitem{Jin}
F. Jin, {\it Algebraic $G$-theory in motivic homotopy categories}, 2018.

\bibitem{JoyceTanakaUpmeier2020}
D. Joyce, Y. Tanaka, and M. Upmeier, {\it On orientations for
gauge-theoretic moduli spaces}, Adv. Math. 362 (2020), no. 106957.

\bibitem{JoyceUpmeierSpin2021}
  D. Joyce, M. Upmeier, 
  {\it On spin structures and orientations for gauge-theoretic moduli
             spaces}, Adv. Math. 381
 (2021), no. 107630.


\bibitem{3-folds}
D. Joyce and M. Upmeier, {\it Orientation data for moduli spaces of
coherent sheaves over Calabi--Yau $3$-folds}, Adv. Math. 381 (2021),
no. 107627.

\bibitem{JU}
D. Joyce and M. Upmeier, {\it Bordism categories and orientations of
moduli spaces}, \href{https://arxiv.org/abs/2503.20456}{arXiv:2503.20456},
March 2025.

\bibitem{Karpov}
I. Karpov, {\it Canonical orientations and Gopakumar--Vafa invariants},
to appear.

\bibitem{KlemmPandharipande2008}
A. Klemm and R. Pandharipande, {\it Enumerative geometry of Calabi--Yau
$4$-folds}, Comm. Math. Phys. 281 (2008), 621--653.

\bibitem{Laures}
G. Laures, {\it On cobordism of manifolds with corners}, Trans. A.M.S.
352 (2000), no. 12, 5667--5688.

\bibitem{Levy1987}
R.N. Levy, {\it The Riemann--Roch theorem for complex spaces}, Acta
Math. 158 (1987), 149--188.

\bibitem{Levy2008}
R.N. Levy, {\it Riemann--Roch theorem for higher bivariant $K$-functors},
Ann. Inst. Fourier 58 (2008), no. 2, 571--601.

\bibitem{MaulikToda}
  D. Maulik, Y. Toda, {\it Gopakumar--Vafa invariants via vanishing cycles},
  Invent. Math.,
  213,
  2018.

\bibitem{Miller}
H.R. Miller, {\it The Riemann--Roch formula},
\url{https://math.mit.edu/~hrm/papers/rrt.pdf}.
A note on the topological Riemann--Roch theorem.

\bibitem{Moulinos}
T. Moulinos, {\it Derived azumaya algebras and twisted $K$-theory},
Adv. Math. 351 (2019), 761--803.

\bibitem{OhThomas2023}
J. Oh and R.P. Thomas, {\it Counting sheaves on Calabi--Yau $4$-folds, I},
Duke Math. J. 172 (2023), no. 7, 1333--1409.

\bibitem{PandharipandeThomas2010}
R. Pandharipande and R.P. Thomas, {\it Stable pairs and BPS invariants},
J. Amer. Math. Soc. 23 (2010), no. 1, 267--297.

\bibitem{Pardon2023}
J. Pardon, {\it Universally counting curves in Calabi--Yau threefolds},\hfil\break
\href{https://arxiv.org/abs/2308.02948}{arXiv:2308.02948}, August 2023.

\bibitem{RamisRuget1974}
J.-P. Ramis and G. Ruget, {\it R\'esidus et dualit\'e}, Invent. Math.
26 (1974), no. 2, 89--131.

\bibitem{RRV}
J.-P. Ramis, G. Ruget, and J.-L. Verdier, {\it Dualit\'e relative en
g\'eom\'etrie analytique complexe}, Invent. Math. 13 (1971), no. 4,
261--283.

\bibitem{Shi}
Y. Shi, {\it Orientation data for coherent sheaves on the local
projective plane}, arXiv e-prints, 2018.

\bibitem{SteenrodEpstein1962}
N.E. Steenrod, {\it Cohomology Operations}, Annals of Mathematics
Studies 50, Princeton University Press, Princeton, NJ, 1962. Lectures
by N.E. Steenrod, written and revised by D.B.A. Epstein.

\bibitem{TrivediTrotman2020}
S. Trivedi and D. Trotman, {\it Deformation retracts to intersections
of Whitney stratifications}, J. Singularities 22 (2020), 315--320.

\end{thebibliography}
\end{document}